\documentclass{amsart}

\usepackage[T1]{fontenc}
\usepackage[hidelinks]{hyperref}
\usepackage{cleveref}
\usepackage{amsthm}
\usepackage{bm}
\usepackage{quiver}

\usepackage{import}

\usepackage{graphicx}

\usepackage{tikzit}

\tikzstyle{green}=[text={black!30!green}]
\tikzstyle{blue}=[text=blue]
\tikzstyle{red}=[text=red]
\tikzstyle{puncture}=[fill=white, draw=red, shape=circle, minimum size=1pt]
\tikzstyle{blackpuncture}=[fill=white, draw=black, shape=circle, minimum size=1pt]
\tikzstyle{cyan}=[text=cyan]

\tikzstyle{markings}=[-, draw=red, line width=1pt, line cap=round]
\tikzstyle{overbraid}=[-, draw=white, fill=none, line width=6pt]
\tikzstyle{thick}=[-, line width=2pt, draw=blue]
\tikzstyle{dashedline}=[-, dashed]
\tikzstyle{dottedline}=[-, dash pattern=on 0.75pt off 0.75pt, line width=0.75pt]
\tikzstyle{thin red}=[-, line width=0.25pt, draw=black]
\tikzstyle{tangle}=[-, draw=blue, line width=1pt, fill={blue!20}]
\tikzstyle{scc}=[-, draw={black!30!green}, fill={blue!20}, line width=1pt]
\tikzstyle{inner boundary}=[-, fill=white]
\tikzstyle{outer boundary}=[-, fill={red!20}]
\tikzstyle{lowerboundery}=[-, line width=1.5pt, line cap=round, draw=red]
\tikzstyle{upperboundery}=[-, line width=1.5pt, line cap=round, draw=blue]
\tikzstyle{dottedcycle}=[-, draw=blue, dash pattern=on 0.5pt off 1pt on 4pt off 1pt, decoration={markings, mark=at position 0.5 with {\arrow{>}}}, postaction=decorate]
\tikzstyle{cycle}=[-, draw=blue, decoration={markings, mark=at position 0.5 with {\arrow{>}}}, postaction=decorate]
\tikzstyle{path}=[-, draw=cyan, line width=0.25pt]
\tikzstyle{arrowpath}=[-, draw=cyan, line width=0.25pt, decoration={markings, mark=at position 0.5 with {\arrow{>}}}, postaction=decorate]
\tikzstyle{orientedpath}=[-, line width=0.25pt, decoration={markings, mark=at position 0.5 with {\arrow{<}}}, postaction=decorate]
\tikzstyle{inner square}=[-, fill={blue!20}]
\tikzstyle{outer square}=[-, fill={red!20}]
\tikzstyle{blueline}=[-, draw=blue]
\tikzstyle{greenline}=[-, draw=green]
\tikzstyle{bluesquare}=[-, draw=blue, fill={blue!20}]

\usetikzlibrary{decorations.markings}

\usepackage{caption}
\usepackage{xparse}

\numberwithin{figure}{section}

\newtheorem{theorem}{Theorem}[section]
\newtheorem*{theorem*}{Theorem}

\newtheorem{proposition}[theorem]{Proposition}
\newtheorem*{proposition*}{Proposition}

\newtheorem{corollary}[theorem]{Corollary}
\theoremstyle{remark}
\newtheorem{remark}[theorem]{Remark}
\newtheorem{notation}[theorem]{Notation}
\newtheorem*{question*}{Question}
\theoremstyle{definition}
\newtheorem{definition}[theorem]{Definition}

\newcommand{\ext}{\mathrm{ext}}
\newcommand{\Disk}{\mathbb{D}}
\newcommand{\OK}{\mathcal{O}_K}

\newcommand{\sline}{\overline{s}}

\newcommand{\SO}{\mathrm{SO}(3)}
\newcommand{\SU}{\mathrm{SU}(2)}
\newcommand{\sltwo}{\mathfrak{sl}_2}

\newcommand{\pr}[1]{\mathrm{pairs}(#1)}
\newcommand{\cp}[1]{\mathrm{comp}(#1)}

\newcommand{\Cn}{C_\infty}

\newcommand{\scc}{simple closed curve}

\newcommand{\lra}{\longrightarrow}
\newcommand{\lla}{\longleftarrow}
\newcommand{\ra}{\rightarrow}

\newcommand{\Nu}{\mathcal{V}}
\newcommand{\Line}{\mathcal{L}}

\newcommand{\Lq}{\Line^q}

\newcommand{\PMod}[1]{\mathrm{PMod}(#1)}
\newcommand{\Mod}[1]{\mathrm{Mod}(#1)}
\newcommand{\Modp}[2]{\mathrm{Mod}_{#1}(#2)}

\newcommand{\PModl}[2]{\mathrm{PMod}^{#1}(#2)}
\newcommand{\Modl}[2]{\mathrm{Mod}^{#1}(#2)}

\protected\def\myphantom#1{\vphantom{#1}}
\newcommand{\myleftidx}[3]{{\myphantom{#2}}#1\!#2#3}
\newcommand{\mysecondleftidx}[3]{{\myphantom{#2}}#1#2#3}

\newcommand{\sM}[1]{\mathcal{M}_{0,{#1}}}
\newcommand{\M}[1]{\overline{\mathcal{M}}_{0,{#1}}}
\newcommand{\Mrp}[2]{\overline{\mathcal{M}}_{0,{#1}}(#2)}
\newcommand{\Mrpo}[2]{\mysecondleftidx{^1}{\overline{\mathcal{M}}}{_{0,{#1}}}(#2)}
\newcommand{\sMrpo}[2]{\myleftidx{^1}{\mathcal{M}}{_{0,{#1}}}(#2)}
\newcommand{\Crp}[2]{\overline{\mathcal{C}}_{0,{#1}}(#2)}
\newcommand{\Mrpb}[3]{\overline{\mathcal{M}}_{0,{#1}}^{#2}(#3)}
\newcommand{\sMrpb}[3]{\mathcal{M}_{0,{#1}}^{#2}(#3)}
\newcommand{\Mrpbo}[3]{\mysecondleftidx{^1}{\overline{\mathcal{M}}}{_{0,{#1}}^{#2}}(#3)}
\newcommand{\sMrpbo}[3]{\myleftidx{^1}{\mathcal{M}}{_{0,{#1}}^{#2}}(#3)}
\newcommand{\Mrb}[2]{\overline{\mathcal{M}}_{0}^{#1}(#2)}
\newcommand{\Mrbo}[2]{\mysecondleftidx{^1}{\overline{\mathcal{M}}}{_{0}^{#1}}(#2)}

\newcommand{\Mp}{\mathcal{M}'}
\newcommand{\justM}{\mathcal{M}}
\newcommand{\Mpb}{\overline{\mathcal{M}}'}
\newcommand{\Mb}{\overline{\mathcal{M}}}
\newcommand{\pb}{\overline{p}}

\newcommand{\Conf}[2]{\mathrm{Conf}_{#1}(#2)}

\newcommand{\Nuab}{\Nu^{\mathrm{ab}}_{\zeta_r}}
\newcommand{\Nusl}{\Nu_{\zeta_r}}
\newcommand{\NuO}{\Nu^{\mathcal{O}}}
\newcommand{\NuabO}{\Nu^{\mathrm{ab},\mathcal{O}}_{\zeta_r}}
\newcommand{\NuslO}{\Nu^{\mathcal{O}}_{\zeta_r}}
\newcommand{\NuOroot}{\Nu^{\mathcal{O},\mathrm{root}}}
\newcommand{\NuOaphi}{\Nu^{\mathcal{O},a,\phi}}

\newcommand{\Hmid}{\mathrm{H}}
\newcommand{\sheafHmid}{\mathcal{H}}

\newcommand{\B}{{\mathcal{B}}}

\newcommand{\GL}[1]{\mathrm{GL}(#1)}
\newcommand{\U}[2]{\mathrm{U}_{#1}(#2)}

\newcommand{\GLn}[2]{\mathrm{GL}_{#1}(#2)}
\newcommand{\PU}[2]{\mathrm{PU}(#1,#2)}

\newcommand{\tr}[1]{\mathrm{tr}(#1)}
\newcommand{\qtr}[1]{\mathrm{qtr}(#1)}
\newcommand{\bhom}[2]{\overline{\mathrm{hom}}({#1},{#2})}
\newcommand{\myhom}[2]{\mathrm{hom}({#1},{#2})}

\newcommand{\myim}[1]{\mathrm{im}\left({#1}\right)}

\newcommand{\Z}{\mathbb{Z}}
\newcommand{\N}{\mathbb{N}}

\newcommand{\C}{\mathbb{C}}
\newcommand{\Q}{\mathbb{Q}}
\newcommand{\R}{\mathbb{R}}

\newcommand{\id}{\mathrm{id}}

\newcommand{\Dc}{\mathcal{D}}
\newcommand{\oD}{\overline{D}}

\newcommand{\ua}{\underline{a}}
\newcommand{\um}{\underline{m}}
\newcommand{\us}{\underline{s}}

\newcommand{\WF}[2]{W^{\overline{F}}_{#2}(#1)}
\newcommand{\WFt}[1]{W^{\overline{F}}(#1)}
\newcommand{\WE}[2]{W^{E}_{#2}(#1)}
\newcommand{\WEt}[1]{W^{E}(#1)}

\newcommand{\Fb}[1]{\overline{F}^{({#1})}}
\newcommand{\Eb}[1]{E^{({#1})}}

\newcommand{\UF}{U^{\overline{F}}_{\zeta_r}}
\newcommand{\UE}{U^{E}_{\zeta_r}}
\newcommand{\Ufin}{U_{\zeta_r}^{\mathrm{fin}}}

\newcommand{\VF}[1]{\hat{V}^{\overline{F}}_{#1}}
\newcommand{\VE}[1]{\hat{V}^{E}_{#1}}

\newcommand{\qbin}[2]{\begin{bmatrix}{#2}\\{#1}\end{bmatrix}}
\newcommand{\Kr}{\begin{bmatrix}K\\r\end{bmatrix}}

\newcommand{\Lset}{\{0,1,\dotsc,r-2\}}

\title{\texorpdfstring{
    Construction of Hodge structures on the $\SO$ modular functors
    }{
    Construction of Hodge structures on the SO(3) modular functors
    }
}
\author{Pierre Godfard}

\begin{document}

\begin{abstract}
    We prove that $\SO$ modular functors in genus $0$ have geometric origin and support
    integral variations of Hodge structures for any odd level $r$ and $r$-th root of unity $\zeta_r\in\C$.
    We identify the TQFT intersection forms and integral structures with the geometric ones.
    Moreover, the gluing property of the modular functors is recovered geometrically as a Künneth formula.
    The construction is based on the homological models of Felder-Wieczerkowski and Martel.
\end{abstract}

\maketitle

\section{Introduction}

This article is the first in a series of two which aims to construct and study a Hodge structure on the vector bundles associated with TQFTs.
This structure comes from a geometric interpretation of the fibers of these bundles as cohomology
of a configuration space in $\mathbb{P}^1$ minus a finite number of points with coefficients in a local system.

In this article, we explicit this geometric interpretation and show that it is compatible with the TQFT hermitian and integral structure,
and the TQFT gluing isomorphisms.

In the second article \cite{godfardHodgeTheoryConformalInpreparation} of this series,
we will use Non-Abelian Hodge theory to prove that all semi-simple TQFTs support unique Hodge structures
and that the characteristic classes of these Hodge structures form Cohomological Field Theories.
We will also use the constructions of this first paper to compute the dimensions of the Hodge decomposition (Hodge numbers)
for all the $\SO$ representations (for any genus).

\subsection{Modular Functors}

The modular functors we study are families of representations of mapping class groups with compatibility conditions.
In this paper, we restrict to the cases of the so-called $\SO$\footnote{or equivalently $\SU$ since we are in genus $0$.}
and $\mathrm{U}(1)$ modular functors and mapping class groups of genus $0$ surfaces.
These representations can be constructed via Skein theory (see \Cref{subsectionskein}), representations of the quantum group $\mathrm{U}_{\zeta_r}(sl_2)$
(see \Cref{subquantumconstruction})
or the monodromy of the Knizhnik–Zamolodchikov connection (see \cite[5.16]{ohtsuki2001quantum}).
For $\SO$, they are known to be isomorphic to the spaces of $sl_2$ conformal blocks.

More precisely, for each $n\geq 0$ and $b,a_1,\dotsc,a_n\in\Lambda=\{0,1,\dotsc,r-2\}$, the $\SO$ modular functor of level $r$ gives a representation:
\begin{equation*}
    \rho_{n}(b;\ua):PB_n\lra \mathrm{GL}_{d}(\Q(\zeta_r))
\end{equation*}
where $PB_n$, the pure braid group on $n$ strands, is identified to the mapping class group of the disk with $n$ punctures, and $\Q(\zeta_r)$
is the cyclotomic field of $r$-th roots of unity. The dimension $d$ of the representation depends on $b$ and $\ua$.

The representation $\rho_{n}(b;\ua)$ factors through the quotient of $PB_n$ by the subgroup generated by the $r$-th
powers of Dehn twists.
This quotient is the fundamental group of the compact moduli space 
$\Mrpo{n}{r}$ of genus $0$ $r$-twisted nodal curves with $n$ marked points and $1$ orbifold point of order $r$ (see \Cref{definitionmodulispaces}).
Hence, we can equivalently see $\rho_{n}(b;\ua)$ as a $\Q(\zeta_r)$-vector bundle with flat connection:
\begin{equation*}
    \Nusl(b;a_1,\dotsc,a_n)\lra \Mrpo{n}{r}.
\end{equation*}

Similarly, in the case $G=\mathrm{U}(1)$, for each $n\geq 0$ and $b,a_1,\dotsc,a_n\in\Lambda=\Z/r\Z$, we have a flat $\Q(\zeta_r)$-line bundle:
\begin{equation*}
    \Nuab(b;a_1,\dotsc,a_n)\lra \Mrpo{n}{r}.
\end{equation*}
The modular functors for $G=\mathrm{U}(1)$ are called the \textit{abelian} modular functors.

\subsection{Geometric construction and existence of Hodge structures}\label{CVHS}

The bundles $\Nusl(b;a_1,\dotsc,a_n)$ are rigid for $r=5$ and expected to be rigid for all $r$ by Kazhdan's property (T) (yet to be proved for mapping class groups).
As $\Mrpo{n}{r}$ is a compact Kähler orbifold, such rigidity would imply that for each inclusion $\Q(\zeta_r)\hookrightarrow \C$,
$\Nusl(b;a_1,\dotsc,a_n)\otimes\C$ supports a complex variation of Hodge structures.
See \cite[1.2]{godfardRigidityFibonacciRepresentations2023} for a discussion of rigidity.

\begin{definition}
    Let $(E,\nabla)\ra B$ be a flat $\C$-local system over an analytic space. A complex variation of Hodge structure on $(E,\nabla)$ is a decomposition:
    \begin{equation}\label{equationHodgedecomposition}
        E = \bigoplus_{p+q=n}E^{p,q}
    \end{equation}
    into smooth subbundles satisfying Griffiths transversality, ie. for all $p,q$:
    \begin{equation*}
        \nabla E^{p,q}\subset E^{p-1,q+1}\otimes\Omega_B^{1,0}\oplus E^{p,q}\otimes\Omega_B^{1}\oplus E^{p+1,q-1}\otimes\Omega_B^{0,1}.
    \end{equation*}
    If the local system is irreducible, then a complex variation of Hodge structure on $E$ is unique up to translation
    (ie. $\forall p,q,\;E'^{p,q}=E^{p+t,q+s}$).

    Let $(E,\nabla)\ra B$ be a flat $\mathcal{O}_K$-local system over an analytic space for $K$ a number field.
    An integral variation of Hodge structures on $(E,\nabla)$ is the data of a positive integer $n$ called weight, and for each embedding
    $\iota:K\hookrightarrow\C$, of a decomposition of $E_\iota=E\otimes \C$ as in \Cref{equationHodgedecomposition} 
    satisfying Griffiths transversality. We also ask that if $\iota:K\hookrightarrow\C$ and
    $\overline{\iota}:K\hookrightarrow\C$ are complex conjugate embeddings, than for all $p,q$, $\overline{E_\iota^{p,q}}=E_{\overline{\iota}}^{q,p}$.
\end{definition}

A large class of flat bundles that support integral variations of Hodge structures is that of bundles of geometric origin:

\begin{definition}\label{definitiongeometricorigin}
    A flat bundle $E$ on a variety $B$ is called \textbf{of geometric origin} if on some dense Zariski open subset $U\subset B$,
    there exists a proper fibration $p:X\ra U$ and $m\geq 0$ such that $E$ is a flat subquotient of $R^mp_*\underline{\Z}$.
\end{definition}

We will indeed show that the $\SO$ modular functors have a geometric origin and thus supports an integral variation of Hodge structures.
More precisely, the bundle $\Nusl(b;a_1,\dotsc,a_n)$ of the $\SO$ modular functor is obtained from the abelian modular functor
by a forgetful pushforward between moduli spaces as follows:

\begin{theorem}[\ref{theoremgeometricconstruction}]
    Let $r\geq 3$ be an odd integer, $n\geq 2$ and $b,a_1,\dotsc,a_n\in\{0,1,\dotsc,r-2\}$,
    such that $m=\frac{a_1+\dotsb+a_n-b}{2}$ is an integer.
    Consider the $\Q(\zeta_r)$-local system $\Line$ 
    over $\sMrpo{n+m}{r}/S_m$ corresponding to the antisymmetric part of $\Nuab(b;a_1,\dotsc,a_n,-2,\dotsc,-2)$.
    Let us consider the forgetful map:
    \begin{equation*}
        p:\sMrpo{n+m}{r}/S_m\lra \sMrpo{n}{r}.
    \end{equation*}
    Then there is an isomorphism:
    \begin{equation*}
        \Nusl(b;a_1,\dotsc,a_n) \simeq \mathrm{im}\left(R^mp_!\Line\lra R^mp_*\Line\right).
    \end{equation*}
\end{theorem}

The theorem can be rephrased as saying that the fiber of $\Nusl(b;a_1,\dotsc,a_n)$ over the smooth curve $x=[C]\in \sMrpo{n}{r}$ is:
\begin{align*}
    \Nusl(b;a_1,\dotsc,a_n)_x &\simeq \mathrm{im}\left(H_c^m(\Conf{m}{C};\Line)
        \lra H^m(\Conf{m}{C};\Line)\right)\\
        &\simeq \mathrm{im}\left(H_m(\Conf{m}{C};\Line)
        \lra H_m^{BM}(\Conf{m}{C};\Line)\right)
\end{align*}
where $C$ is the curve with the $n+1$ marked points removed.
Here $\sMrpo{n}{r}$ is the open locus of smooth curves in $\Mrpo{n}{r}$.

After the completion of this article, we learned that in their recent paper \cite[4.17]{belkaleConformalBlocksGenus2023a},
P. Belkale and N. Fakhruddin identify $\mathrm{im}\left(R^mp_!\Line\lra R^mp_*\Line\right)$
with spaces of conformal blocks. See \Cref{subsectionrelations} below for comments.

Now, as the flat bundles of the abelian modular functor have finite monodromy, they are of geometric origin, and hence so is $\Nusl(b;a_1,\dotsc,a_n)$:

\begin{corollary}
    The flat bundle $\Nusl(b;a_1,\dotsc,a_n)$ is of geometric origin and supports an integral
    variation of Hodge structures of weight $m=\frac{a_1+\dotsb+a_n-b}{2}$.
\end{corollary}

Note, that if $\Nusl(b;a_1,\dotsc,a_n)$ is irreducible, then for each choice of root of unity, this constructs \textbf{the} complex variation of Hodge structures on
$\Nusl(b;a_1,\dotsc,a_n)\otimes\C$. The irreducibility of $\Nusl(b;a_1,\dotsc,a_n)$  has been proved for $r$ prime
(see \cite{koberdaIrreducibilityQuantumRepresentations2018} and \cite[A.5]{godfardRigidityFibonacciRepresentations2023} for a discussion).
See \Cref{sectionproofHodge} for details on the construction of the Hodge structures.

\subsection{Identification of the polarization and integral structure}

The geometric construction of $\Nusl(b;a_1,\dotsc,a_n)$ gives it an integral structure and a flat hermitian structure,
which we will prove to be a polarization for the Hodge structures in \cite{godfardHodgeTheoryConformalInpreparation}.
The Skein module construction of $\rho_n(b;\ua)$ also comes with a natural invariant hermitian form and, for $r$ prime, with a natural
integral structure (see \Cref{subsectionskein}). Although these structures are constructed in very different ways, we show that they coincide:

\begin{theorem}[\ref{theoremintersectionform}, \ref{intersectionformpolarizes} and \ref{theoremintegralstructure}]
    Let $r\geq 3$ be an odd integer, $n\geq 2$ and $b,a_1,\dotsc,a_n\in\{0,1,\dotsc,r-2\}$,
    such that $m=\frac{a_1+\dotsb+a_n-b}{2}$ is an integer.
    Let $s_{\Dc}$ be the intersection form on $\Nusl(b;a_1,\dotsc,a_n)$ coming from the geometric construction
    and let $h_{\Dc}$ be the hermitian on $\Nusl(b;a_1,\dotsc,a_n)$ form coming from the Skein module construction.
    Then:
    \begin{equation*}
        s_{\Dc}=(\zeta_r^2-\zeta_r^{-2})^mh_{\Dc}
    \end{equation*}
    and $(-1)^{\frac{m(m-1)}{2}}s_{\Dc}$ is a polarization for the Hodge structure on $\Nusl(b;a_1,\dotsc,a_n)$.
    In particular, if $\Nusl(b;a_1,\dotsc,a_n)=\bigoplus_{p+q=m}H^{p,q}$ is the Hodge decomposition,
    then it is orthogonal for $h_{\Dc}$
    and $(-1)^{\frac{m(m-1)}{2}}i^{p-q}(\zeta_r^2-\zeta_r^{-2})^mh_{\Dc}$ is positive definite on $H^{p,q}$.

    Moreover if $r$ is prime, then the geometric integral structure coincides with the Skein module one.
\end{theorem}

See \Cref{definitionpolarization} for the definition of a polarization on a Hodge structure.
For the hermitian structure, this is not surprising for $r$ prime, as an irreducible representation can have only one invariant hermitian form,
up to a scalar factor.

\subsection{The gluing property seen geometrically}

Perhaps the most important property of modular functors is the gluing axiom.
Let $n\geq 2$ and $n_1+n_2=n$ with $n_1\geq 2$ and $n_2\geq 1$. Then we have an an inclusion of boundary divisor:
\begin{equation*}
    q:\Mrpo{n_1}{r}\times\Mrpo{n_2+1}{r}\lra \Mrpo{n}{r}
\end{equation*}
induced by the gluing the $0$-th section of a curve in $\Mrpo{n_1}{r}$
to the $(n_2+1)$-th section of a curve in $\Mrpo{n_2+1}{r}$ (see \Cref{gluing_map} and \Cref{subsubframingluing}).

\begin{figure}
    \ctikzfig{gluing_map}
    \caption{The gluing map $q:\Mrpo{n_1}{r}\times\Mrpo{n_2+1}{r}\ra \Mrpo{n}{r}$ represented at the level of generic Riemann surfaces. The curve on the right is nodal.}
    \label[figure]{gluing_map}
\end{figure}

Then, as part of the data of the $\SO$ modular functor of level $r$, we have an isomorphism:
\begin{equation}\label{equationgluingintro}
    q^*\Nusl(b;a_1,\dotsc,a_n)\simeq \bigoplus_c\Nusl(c;a_1,\dotsc,a_{n_1},\mu)\otimes
    \Nusl(b;c,a_{n_1+1},\dotsc,a_n).
\end{equation}

This isomorphism has a geometric interpretation as a Künneth formula on configuration spaces.
We prove an exact identification between the geoemtric Künneth isomorphism and the $\SO$ modular functor gluing isomorphism.
See \Cref{theoremgluinghomological,theoremgluinggeometric} for details.

From the geometric interpretation of gluing maps we will deduce the following.

\begin{theorem}[\ref{gluingishodge}]
    The gluing map of \Cref{equationgluingintro} is an isomorphism of rational variations of Hodge structures.
\end{theorem}

This theorem is the main point behind the proof that the characteristic classes of these Hodge structures satisfy
the axioms of a Cohomological Field Theory.

After the completion of this article, we learned that in their recent paper \cite{belkaleMotivicFactorisationKZ2023},
P. Belkale and N. Fakhruddin and S. Mukhopadhyay prove similar geometric factorization rules 
for conformal blocks and prove in particular that they preserve Hodge structures. See \Cref{subsectionrelations} below for more comments.

\subsection{Relation to previous works}\label{subsectionrelations}

\subsubsection*{Works on cohomological models via holomorphic forms}

Extensive work has been done on geometric models for solutions of the Knizhnik-Zamolodchikov connection,
notably by Schechtman and Varchenko (see \cite{schechtmanArrangementsHyperplanesLie1991a} for example).
The TQFT representations we consider are known to be isomorphic to representations constructed from the monodromy of some special cases
of the  Knizhnik-Zamolodchikov connection. The proof of this isomorphism is non trivial and due, in the generic case,
to Kohno and Drinfel'd (see \cite{kohnoMonodromyRepresentationsBraid1987,kohnoLinearRepresentationsBraid1988}
and \cite{drinfeldQuasiHopfAlgebras1990,drinfeldQuasitriangularQuasiHopfAlgebras1991}).
Varchenko also gives a geometric proof in the generic case in his book \cite{varchenkoMultidimensionalHypergeometricFunctions1995}.

Most of the work in this topic is done in the generic case. However, in this article, we deal with the non-generic case.
We thus mention here works that apply to the non-generic case.

The spaces $R^mp_*\Line$ in \Cref{theoremgeometricconstruction} have been considered by Silvotti in \cite{silvottiLocalSystemsComplement1994}
and Schechtman-Varchenko in \cite{schechtmanArrangementsHyperplanesLie1991a}.
Silvotti conjectured that the space of conformal blocks was the degree $0$ part of the weight filtration of $R^mp_*\Line$ and proved it for $n=2$.

In \cite{feiginAlgebraicEquationsSatisfied1995}, Feigin-Schechtman-Varchenko identify
conformal blocks as a subspace of $R^mp_*\Line$.
The case of $\sltwo$ is further studied by Varchenko in his book \cite{varchenkoMultidimensionalHypergeometricFunctions1995}.

Using the results in \cite{schechtmanArrangementsHyperplanesLie1991a}, in \cite{ramadasHarderNarasimhanTraceUnitarity2009}
Ramadas proves that $\sltwo$ conformal blocks give square integrable forms in $R^m p_*\Line$.
Belkale then extended these results to all simple Lie algebras in \cite{belkaleUnitarityKZHitchin2012}.
In \cite{looijengaUnitaritySLConformal2009}, Looijenga
then showed that $\sltwo$ conformal blocks correspond to all $L^2$-integrable forms.
Again, this result was extended to all classical Lie algebras and $G_2$ by Belkale and Mukhopadhyay in \cite{belkaleConformalBlocksCohomology2014}.

After the completion of this paper, we learned that in the recent paper
\cite{belkaleConformalBlocksGenus2023a}, Belkale and Fakhruddin identify spaces of conformal blocks with
the image of $R^mp_!\Line$ in $R^mp_*\Line$ for any Lie algebra,
and that in the other recent paper \cite{belkaleMotivicFactorisationKZ2023}, Belkale, Fakhruddin and Mukhopadhyay
study geometric gluing isomorphisms of genus $0$ conformal blocks for any simple Lie algebra, prove that they preserve Hodge
structures and give an algorithm to compute corresponding Hodge numbers in genus $0$ for $\mathfrak{sl}_n$.

We chose not to use the approach of holomorphic forms as it is not best suited for our purpose.
More precisely, we want results over $\Q(\zeta_r)$ and not just $\C$
and we want to be able to identify the TQFT hermitian form with the geometric intersection form.
This would not be possible with these cohomological models as there is no good formula for the TQFT form on solutions of the Knizhnik-Zamolodchikov
equation. Actually, one of the motivations behind the works of Ramadas, Looijenga and Belkale mentioned above was to construct a hermitian
structure on spaces of conformal blocks.

\subsubsection*{Works on homological models via cycles}

Our approach relies heavily on results of Martel \cite{martelHomologicalModelUq2022}, elucidating and improving work of Felder-Wieczerkowski
\cite{felderTopologicalRepresentationsQuantum1991}.
These results identify some semi-relative version of $R^mp_*\Line$ with tensor products of Verma modules.
The main advantage of this work is that it is done for general local systems
but in such a way that it remains true for any choice of the quantum parameter $q$. Hence it is well suited for our case, where $q$ is a root of unity.
See \Cref{subsubbases,subsubUaction} below for the statements we will need from Martel's work.

\subsubsection*{Cell decomposition for configuration spaces}

\Cref{theoremgeometricconstruction} is not a direct corollary of Martel's work.
Indeed, Martel works with a semi-relative version of $R^mp_*\Line$. Hence we need a way to compute $R^mp_*\Line$
and link it to Martel's work.

The space $R^mp_*\Line$ is the cohomology of a configuration space. Several cell decompositions have been used to compute this cohomology,
for example with the Salvetti complex (see, for example \cite[chp. 2]{varchenkoMultidimensionalHypergeometricFunctions1995}).
However these decompositions have a lot of cells and do not work well with Martel's results.
In \Cref{subsubcellularcomplex}, we solve this issue by describe an ad-hoc version of the Fox-Neuwirth-Fuks stratification of configuration spaces
that has less cells and such that its $m$-cells are the elements of Martel's basis.
This cell complex is isomorphic to some Hochschild homology complex and recovers in the generic case
the Hochschild complex described by Schechtman-Varchenko in \cite[3.8]{schechtmanQuantumGroupsHomology1991}.
Hence our complex extends the Hochschild complex described by Schechtman-Varchenko
to the non-generic case and gives a direct proof that it computes $R^mp_*\Line$.

Similar cell complexes on the plane rather than the plane minus a finite number of points are used by
Ellenberg, Tran and Westerland in \cite{ellenbergFoxNeuwirthFuksCellsQuantum2023} and by
Kapranov and Schechtman in \cite{kapranovShuffleAlgebrasPerverse2020a}.

\subsubsection*{The Lawrence representations}

The idea of considering the image of $R^mp_!\Line$ in $R^mp_*\Line$ comes from Bigelow's formulation in
\cite{bigelowHomologicalRepresentationsIwahori2004} of
Lawrence's work \cite{lawrenceHomologicalRepresentationsHecke1990a} on geometrical models for representations of Temperley-Lieb algebras.
These representations have non trivial intersection with the quantum representations we study (see \Cref{subsubBigelow} below for examples).

Bigelow proves his results on the image of $R^mp_!\Line$ in $R^mp_*\Line$ for generic values of $q$ and conjectures that it holds in general
(\cite[Conjecture 6.1]{bigelowHomologicalRepresentationsIwahori2004}). Our approach, which is different from Bigelow's,
proves this conjecture when the representations of the Temperley-Lieb algebra are quantum representations and can be used to prove it in general.
The image of $R^mp_!\Line$ in $R^mp_*\Line$ is also mentioned in the generic case in \cite[(5.2)]{schechtmanQuantumGroupsHomology1991}.

\subsubsection*{Main contributions}

Our main contribution to \Cref{theoremgeometricconstruction} is to use Martel's homological models to identify the TQFT representations 
with the image of $R^mp_!\Line$ in $R^mp_*\Line$, and doing so over $\Q(\zeta_r)$ with explicit isomorphisms with the Skein module and quantum groups constructions
of these representations.
This then enables our identification of the intersection form
and the gluing morphisms in the geometric context with the TQFT hermitian form and gluing isomorphisms (\Cref{theoremintersectionform,theoremgluinggeometric}),
as well as the study of rational Hodge structures of TQFT representations conducted in the upcoming work \cite{godfardHodgeTheoryConformalInpreparation}.

An advantage of the homological approach we use is that the results of Martel have been partially extended to higher
by De Renzi and Martel in \cite{derenziHomologicalConstructionQuantum2023},
whereas the picture on the side of holomorphic forms is more complicated
(see the recent paper \cite{looijengaConformalBlocksCohomology2021} of Looijenga on the subject).
We plan to extend our results to higher genus
using De Renzi and Martel's paper.

\subsection{Outline of the paper}

In \Cref{sectionmodularfunctors} we axiomatize the notion of genus $0$ modular functor, as well as rootings and framings of such functors.
We also define the abelian modular functors and discuss properties of the $\SO$ modular functors.

\Cref{sectionmainresults} is devoted to the statements of the main results of the paper on the geometric construction.

In \Cref{sectionmodularfunctorconstruction}, we give the two constructions of the $\SO$ modular functors that are relevant to us: the skein module construction
(\Cref{subsectionskein}) and the quantum group construction (\Cref{subquantumconstruction}). We also discuss Verma modules for quantum
groups associated to $sl_2$ and study maps between them.

The proof of the geometric construction is carried out in \Cref{sectiongeometricproofs}. The section also includes statements from
the work \cite{martelHomologicalModelUq2022} of J. Martel that we use and the construction of a Borel-Moore cellular decomposition for configuration spaces
(see \Cref{subsubcellularcomplex}).

Finally, in \Cref{sectionproofHodge}, we prove the existence of the integral variations of Hodge structures on the $\SO$ modular functor in genus $0$.

\subsection{Acknowledgements}

This paper forms part of the PhD thesis of the author.
The author thanks Julien Marché for his help in writing this paper.
The author also thanks Prakash Belkale, Bertrand Deroin, Javier Fresán, Jules Martel, Ramanujan Santharoubane
and Alexander Zakharov for helpful discussions and correspondances.


\section{\texorpdfstring{Modular functors in genus $0$}
{Modular functors in genus 0}}\label{sectionmodularfunctors}


\subsection{Axiomatic definition}


\subsubsection{Topological definition}

\begin{definition}\label{definitioncolourset}
    A set of colors is a finite set $\Lambda$ with a preferred element $0\in \Lambda$ and an involution $\lambda\mapsto \lambda^\dagger$
    such that $0^\dagger=0$.
\end{definition}

\begin{definition}\label{definitionsurfaces}
    Let $n\geq 0$. We define $S_0^n$ to be the compact sphere with $n$ boundary components,
    and $S_{0,n}$ to be the sphere with $n$ punctures and no boundary.
\end{definition}

We can now define the source category of modular functors.

\begin{definition}\label{definitioncolouredcategory}
    Let $\Lambda$ be a set of colors. The category of genus $0$ surfaces colored with $\Lambda$ is such that:
    \begin{description}
        \item[(1)] its objects are compact oriented surfaces $S$ of genus $0$ together with an identification
        $\varphi_B: B\simeq S^1$ and a color $\lambda_B\in \Lambda$ for every component $B$ of $\partial S$ ;
        \item[(2)] its morphisms from $\Sigma_1=(S_1,\varphi^1,\underline{\lambda}^1)$
        to $\Sigma_2=(S_2,\varphi^2,\underline{\lambda}^2)$
        are homeomorphisms $f:S_1\lra S_2$ preserving orientation such for every component $B_1\subset\partial S_1$
        and its image $f(B_1)=B_2\subset \partial S_2$, we have $\lambda_{B_1}=\lambda_{B_2}$ and $\varphi^2_{B_2}\circ f=\varphi^1_{B_1}$.
        \item[(3)] the composition of $f_1:\Sigma_0\lra\Sigma_1$ and $f_2:\Sigma_1\lra\Sigma_2$ is simply $f_2\circ f_1$.
    \end{description}
    This category has a natural monoidal structure induced by the disjoint union $\sqcup$.
\end{definition}

In the rest of the paper, $\Sigma=(S,\varphi,\underline{\lambda})$ will be abbreviated $(S,\underline{\lambda})$, or even
$(S,\lambda_1,\lambda_2,\dotsc)$ where $\lambda_1,\lambda_2,\dotsc$ are the colors relevant to the argument and the other colors are omitted.

\begin{definition}\label{definitionmappingclassgroupofspheres}
    Let $S$ be a genus $0$ surface with boundary and punctures. We note by $\Mod{S}$ its mapping class group, ie. the group of connected components
    of the group of orientation preserving homeomorphisms of $S$ fixing the boundary $\partial S$ pointwise. This group may permute
    punctures but not boundary components. We denote by $\PMod{S}$ the subgroup of $\Mod{S}$ of homeomorphisms that do not permute the punctures.

    Let $r\geq 1$, we will denote by $\PModl{r}{S}$ (respectively $\Modl{r}{S}$) the quotient of $\PMod{S}$ (respectively $\Mod{S}$)
    by all $r$-th powers of Dehn twists.
\end{definition}

Let $S$ be a compact genus $0$ surface and $\partial_+S\sqcup\partial_-S\subset\partial S$ be two components of its boundary
such that $\partial_+S$ and $\partial_-S$ lie on different connected components of $S$.
Let $\varphi_{\partial_{\pm}S}:\partial_{\pm}S\simeq S^1$ be identifications of these components with $S^1$.

Let $S_{\pm}$ be the genus $0$ surface obtained from $S$ by gluing $\partial_+S$ to $\partial_-S$
along $\varphi_{\partial_-S}^{-1}\circ\varphi_{\partial_+S}$.
Then $S_{\pm}$ is called the gluing of $S$ along $\partial_{\pm}S$.

We now introduce the notion of genus $0$ modular functor.

\begin{definition}[Modular Functor]\label{definitionmodularfunctor}
    Let $\Lambda$ be a set of colors and $\mathrm{C}$ be the associated category of colored surfaces
    as defined in \Cref{definitioncolouredcategory}.
    Then a modular functor is the data an integer $r\geq 1$ called \textit{level} and of a monoidal functor:
    $$\Nu:\mathrm{C}\lra \Q(\zeta_r)-\text{vector spaces}$$
    where the monoidal structure on $\Q(\zeta_r)$-vector spaces is understood to be the tensor product.
    This data is augmented by the following isomorphisms.
    \begin{description}
        \item[(G)] For any genus $0$ surface $S$ and pair of boundary components $\partial_{\pm}S$ lying on distinct connected components of $S$,
        let $S_{\pm}$ be the gluing of $S$ along $\partial_{\pm}S$. For any coloring $\underline{\lambda}$
        of the components of $\partial S_{\pm}$, an isomorphism as below is given:
        \begin{equation}\label{gluingunrooted}
            \Nu(S_{\pm},\underline{\lambda})\simeq \bigoplus_{\mu\in\Lambda}\Nu(S,\mu,\mu^\dagger,\underline{\lambda})
            \otimes \Nu(S_0^2,\mu,\mu^\dagger)^\vee.
        \end{equation}
    \end{description}
    The isomorphisms of \textbf{(G)} are assumed to be functorial and compatible with disjoint unions.
    This rule, also sometimes called fusion or factorization rule, is the most important property of modular functors.
    The functor is also assumed to verify $3$ more axioms:
    \begin{description}
        \item[(1)] $\dim\Nu(S_0^1,\lambda)=1$ if $\lambda=0$ and $0$ otherwise;
        \item[(2)] $\dim\Nu(S_0^2,\lambda,\mu)=1$ if $\lambda=\mu^\dagger$ and $0$ otherwise. Moreover, $\Nu(S_0^2,\lambda,\lambda^\dagger)$
        is canonically isomorphic to $\Q(\zeta_r)$;
        \item[(3)] For every color $\lambda$, the left Dehn twist in $S_0^2$ acts on $\dim\Nu(S_0^2,\lambda,\lambda^\dagger)$
        by multiplication by some $r$-th root of unity, denoted by $t_\lambda=t_{\lambda^\dagger}$.
    \end{description}
\end{definition}

\begin{remark}
    One can think of \Cref{gluingunrooted} as the gluing of $\partial S_{\pm}$ to the boundaries of a cylinder $S_0^2$.
    From the additional axiom \textbf{(2)}, one has $\Nu(S_0^2,\mu,\mu^\dagger)\simeq \Q(\zeta_r)$ canonically.
    Hence we could write \Cref{gluingunrooted} simply as:
    \begin{equation*}
        \Nu(S_{\pm},\underline{\lambda})\simeq \bigoplus_{\mu\in\Lambda}\Nu(S,\mu,\mu^\dagger,\underline{\lambda}).
    \end{equation*}
    However, this would make the isomorphism not an isometry in the case where $\Nu$ has a Hermitian structure.
    See \Cref{subsubhermitianintegral} for details. This is also one of the reasons for rooting modular functors,
    see \Cref{subrooting}.
\end{remark}

\begin{remark}\label{remarklevel}
    Let $S_{\pm}$ be a genus $0$ colored surface constructed as a gluing of $S$ along $\partial_{\pm}S$.
    Let $\gamma$ denote the simple closed curve that is the image of $\partial_{\pm}S$ in $S_{\pm}$.
    Then the Dehn twist $T_\gamma$ acts block-diagonally on the decomposition \textbf{(G)}.

    Moreover, one can easily see that it acts on the block $\Nu(S,\mu,\mu^\dagger,\underline{\lambda})\otimes \Nu(S_0^2,\mu,\mu^\dagger)^\vee$ by the scalar $t_\mu$,
    that depends only on $\mu$, and not on the surface $S$.
\end{remark}

From the modular functor, we get representations of mapping class groups.

\begin{definition}[Quantum representations]
    Let $\Nu$ be a modular functor of level $r$.
    Then, for any genus $0$ surface $S_0^n$ and coloring $\underline{\lambda}$ of its boundary components, the functor yields a representation:
    $$\rho_{n}(\underline{\lambda}):\Mod{S_0^n}\lra \GL{\Nu(S_0^n,\underline{\lambda})}$$
    which factors as a representation:
    $$\rho_{n}^r(\underline{\lambda}):\Modl{r}{S_0^n}\lra \GL{\Nu(S_0^n,\underline{\lambda})}.$$

    These latter representations will be called the (quantum) representations associated to the modular functor $\Nu$.
\end{definition}


\subsubsection{Geometric interpretation}\label{sectiongeometricdefinitions}

The quantum representations have geometric counterparts, that enable an equivalent geometric definition of genus $0$ modular functors.

\begin{definition}\label{definitionmodulispaces}
    Let $n\geq 3$. Let us denote by $\sM{n}$ the moduli space of smooth genus $0$ Riemann surface with $n$ distinct sections.
    We shall denote by $\Mrp{n}{r}$ the moduli space of stable $r$-twisted curves with $n$ distinct smooth sections
    (see \cite[1.3]{chiodoStableTwistedCurves2008}).
    
    Consider the universal curve $\Crp{n}{r}$ over $\Mrp{n}{r}$ with orbifold structure of order $r$ at the marked points.
    Then for each $1\leq i\leq n$, the $i$-th marked point is a $\mu_r$-gerbe over $\Mrp{n}{r}$. Denote it by $\sigma_i$.

    For $0\leq m\leq n$, denote by $\Mrpb{n-m}{m}{r}$ the fiber product of $\sigma_1,\dotsc,\sigma_m$ over $\Mrp{n}{r}$.
    For $m=n$ we abbreviate the notation to $\Mrb{n}{r}$.
\end{definition}

\begin{proposition}
    The orbifold fundamental group of $\sM{n}$ is $\PMod{S_{0,n}}$, where the identification is given by the monodromy of the universal curve
    $\mathcal{C}_{0,n}\ra\sM{n}$.
    Then the orbifold fundamental group of $\Mrp{n}{r}$ is $\PModl{r}{S_{0,n}}$ and that of $\Mrpb{m}{n}{r}$ is $\PModl{r}{S_{0,m}^{n}}$.
    In particular, the orbifold fundamental group of $\Mrb{n}{r}$ is $\Modl{r}{S_0^n}$.
\end{proposition}

\begin{proof}
    Let $\justM'\subset\Mrp{n}{r}$ be the locus of curves with at most one node. Then as $\Mrp{n}{r}\setminus\justM'$
    has codimension $2$, we have $\pi_1(\justM')=\pi_1(\Mrp{n}{r})$. Set $D=\justM'\setminus\sM{n}$ and let
    $U\subset \justM'$ be a tubular neighborhood of $D$. Now the monodromies around the components of $D$
    correspond to the $r$-th powers of Dehn twists under the isomorphism $\pi_1(\sM{n})=\PMod{S_{0,n}}$.
    Hence, applying groupoïd Van-Kampen to the covering $U\cup\sM{n}=\justM'$, we get $\pi_1(\justM')=\PModl{r}{S_{0,n}}$.

    To compute the fundamental group of $\Mrpb{m}{n}{r}$, we use the fibration $(B\mu_r)^m\ra\Mrpb{m}{n}{r}\ra\Mrp{n+m}{r}$.
    This fibration shows that $\pi_1(\Mrpb{m}{n}{r})$ is a central extension of $\pi_1(\Mrp{n+m}{r})$ by a quotient of $\mu_r^n$.
    We need only compute the extension on the loci of smooth curves. Let $\mathcal{M}_{0,m}^n$ be the $(\C^*)^n$-bundle
    over $\mathcal{M}_{0,n+m}$ corresponding to choice of non-zero tangent bundles at the first $n$ sections.
    It is well known that the extension $\pi_1((C^*)^n)\ra\pi_1(\mathcal{M}_{0,m}^n)\ra\pi_1(\sM{n+m})$
    is naturally identified to $\Z^n\ra \PMod{S_{0,m}^{n}} \ra \PMod{S_{0,n+m}}$. Now $\sMrpb{m}{n}{r}$ is constructed from
    $\mathcal{M}_{0,n}^m$ by replacing the $(\C^*)^n$-bundle by the corresponding $\mu_r^n$-gerbe. Hence the result.
\end{proof}

We can define the geometric counterpart to the gluing maps.

\begin{proposition}
    Let $n,m\geq 2$. Then the gluing map:
    \begin{equation*}
        \Mrb{n+1}{r}\times\Mrb{m+1}{r}\lra \Mrb{n+m}{r}
    \end{equation*}
    is the composition of a $\mu_r$-gerbe and the embedding of a divisor.
    The induced map of orbifold fundamental groups is the morphism:
    \begin{equation*}
        \Modl{r}{S_{0}^{n+1}}\times\Modl{r}{S_{0}^{m+1}}\lra \Modl{r}{S_{0}^{n+m}}
    \end{equation*}
    obtained by gluing the $n+1$-th boundary component of $S_{0}^{n+1}$ to the $m+1$-th boundary component of $S_{0}^{m+1}$.
\end{proposition}

We can now explain the quantum representations of a modular functor in geometric terms.

\begin{proposition}\label{gluinggeometric}
    Let $\Nu$ be a modular functor of level $r$. Let $n\geq 1$ and $\lambda_1,\dotsc,\lambda_n$ be some colors.
    Then the quantum representation $\rho_{n}^r(\underline{\lambda})$ corresponds to a flat $\Q(\zeta_r)$-vector bundle
    $\Nu(\lambda_1,\dotsc,\lambda_n)$ over $\Mrb{n}{r}$.

    Now let $n=n_1+n_2$ with $n_1,n_2\geq 2$. Denote by $q$ the gluing map:
    \begin{equation*}
        \Mrb{n_1+1}{r}\times\Mrb{n_2+1}{r}\lra \Mrb{n}{r}.
    \end{equation*}
    Then the gluing axiom for $\Nu$ translates into the isomorphism of flat bundles:
    \begin{equation*}
        q^*\Nu(\lambda_1,\dotsc,\lambda_n)\simeq \bigoplus_\mu\Nu(\lambda_1,\dotsc,\lambda_{n_1},\mu)\otimes
        \Nu(\lambda_{n_1+1},\dotsc,\lambda_n,\mu^\dagger).
    \end{equation*}
\end{proposition}

Since a modular functor is a little more than just its quantum representations, we have a little more structure.

\begin{proposition}\label{datageometricmodularfunctor}
    Let $\Nu$ be a modular functor of level $r$. Let $n\geq 1$ and $\mu_1,\dotsc,\mu_n$ be some colors.
    Then we have a flat bundle $\Nu^\mathrm{sym}_m(\mu_1,\dotsc,\mu_n)$ over the orbifold quotient $\Mrb{m+n}{r}/S_m$,
    such that its pullback to $\Mrb{m+n}{r}$ is:
    \begin{equation*}
        \bigoplus_{\lambda_1,\dotsc,\lambda_m\in \Lambda} \Nu(\lambda_{1},\dotsc,\lambda_{m},\mu_1,\dotsc,\mu_n).
    \end{equation*}
    Then the gluing isomorphisms translate as follows. For $n_1+n_2=n$ and $m_1+m_2=m$ with $m_1+n_1\geq 2$ and $m_1+n_1\geq 2$,
    we have a gluing map:
    \begin{equation*}
        q: \Mrb{m_1+n_1+1}{r}/S_{m_1}\times\Mrb{m_2+n_2+1}{r}/S_{m_2} \lra \Mrb{m+n}{r}/S_{m}.
    \end{equation*}
    Then the gluing is an isomorphism:
    \begin{equation*}
        q^*\Nu^\mathrm{sym}_m(\mu_1,\dotsc,\mu_n)\simeq \bigoplus_\nu\Nu^\mathrm{sym}_{m_1}(\mu_1,\dotsc,\mu_{n_1},\nu)\otimes
        \Nu^\mathrm{sym}_{m_2}(\mu_1,\dotsc,\mu_{n_2},\nu^\dagger).
    \end{equation*}
    Moreover, if we consider the forgetful map:
    \begin{equation*}
        p: \Mrb{m+1}{r}/S_{m} \lra \Mrb{m}{r}/S_{m}
    \end{equation*}
    and property \textbf{(1)} of \Cref{definitionmodularfunctor} induce a forgetting isomorphism:
    \begin{equation*}
        p^*\Nu^\mathrm{sym}_m\simeq \Nu^\mathrm{sym}_m(0).
    \end{equation*}
\end{proposition}

The bundles $\Nu^\mathrm{sym}_m(\mu_1,\dotsc,\mu_n)$ together with the gluing and forgetting isomorphisms
described in \Cref{datageometricmodularfunctor} contain all the data of the modular functor.


\subsubsection{The abelian modular functors}

Here we explicitly define the so called abelian modular functors.
They are called abelian because they can be constructed through quantization of $\mathrm{U}(1)$-character varieties of surfaces.
It can be constructed with classical $\Theta$ functions (see \cite{gelcaClassicalThetaFunctions2014} for example).
However, we only use the genus $0$ part of the abelian modular functor, which can be given explicitly as follows.

\begin{definition}
    Let $r\geq 1$ be an odd integer. Let $\Lambda=\Z/2r\Z$ with $0$ as distinguished element and $a^\dagger= -a$ as involution.
    Then define the abelian genus $0$ modular functor of level $2r$ by:
    \begin{equation*}
        \Nuab(S_0^n,a_1,\dotsc,a_n) = \begin{cases}
            \Q(\zeta_r)\text{ if }a_1+\dotsb+a_n=0\;[2r] \\
            0\text{ otherwise.}
        \end{cases}
    \end{equation*}
    Then the other data of $\Nuab$ are completely determined by:
    \begin{description}
        \item[(1)] the value of $t_a$ for $a\in \Z/2r\Z$, that is set to:
        \begin{equation*}
            t_a = q^{-\frac{a^2}{2}}\text{ where }q^{\frac{1}{2}}=-\zeta_r.
        \end{equation*}
        \item[(2)] the action of half-twists. Let $\sigma_1$ be the map of colored surface:
        \begin{equation*}
            \sigma_1: (S_0^n,a_1,a_2,a_3\dotsc,a_n)\lra (S_0^n,a_2,a_1,a_3\dotsc,a_n)
        \end{equation*}
        corresponding to the left half-twist exchanging the first two boundary components in $S_0^n$.
        Then it acts by multiplication by $q^{-\frac{a_1a_2}{2}}$.
    \end{description}
\end{definition}

\begin{remark}
    Although $\Nuab$ has level $2r$, we will often call it the abelian modular functor of level $r$.
    This can be explained by the fact that its rooted and framed version has level $r$ (see \Cref{remarklevelframing}).
\end{remark}

\begin{remark}
    On can compute the action of a Dehn twist around $\gamma\subset S_0^n$ on $\Nuab(S_0^n,a_1,\dotsc,a_n)$ as follows.
    The {\scc} separates the boundary components into $I$ and $J$ with $I\sqcup J=\{1,\dotsc,n\}$.
    Let $b=\sum_{i\in I}a_i$, then the left twist along $\gamma$ acts by $q^{-\frac{b^2}{2}}$.

    As for the gluing axiom at $\gamma$, let $S_1\sqcup S_2$ be the surface obtained by cutting along $\gamma$,
    with the boundary components numbered by $I$ in $S_1$.
    Then:
    \begin{equation*}
        \Nuab(S_0^n,a_1,\dotsc,a_n)=\Nuab(S_1,(a_i)_{i\in I},-b)\otimes \Nuab(S_2,(a_j)_{j\in J},b)
    \end{equation*}
    as $\Nuab(S_1,(a_i)_{i\in I},-c)\neq 0$ if and only if $c=b$.
\end{remark}


\subsubsection{\texorpdfstring{The $\SO$ modular functors}
{The SO(3) modular functors}}

\begin{theorem}\label{existenceSOmodularfunctors}
    Let $r\geq 3$ be an odd integer. Let $\Lambda=\{0,1,2,\dotsc,r-2\}$, with distinguished element $0$ and trivial involution.

    Then there exists a genus $0$ modular functor $\Nusl$ on the surfaces colored with $\Lambda$,
    called the $\SO$ modular functor of level $2r$. It satisfies:
    \begin{description}
        \item[(1)] For $a\in \{0,1,2,\dotsc,r-2\}$, $t_a=q^{\frac{a(a+2)}{2}}$;
        \item[(2)] For every $a,b,c\in\{0,1,2,\dotsc,r-2\}$, $\Nusl(S_0^3,a,b,c)$ has dimension $0$ or $1$;
        \item[(3)] For every $\Sigma=(S,\lambda_1,\dotsc,\lambda_n)$, $\Nu(\Sigma)$ is $0$ if $\sum_i\lambda_i$ is odd.
    \end{description}
\end{theorem}

\begin{remark}
    Although $\Nusl$ has level $2r$, we will often call it the $\SO$ modular functor of level $r$.
    This can be explained by the fact that its rooted and framed version has level $r$ (see \Cref{remarklevelframing}).
\end{remark}

Two constructions of these functors are given in \Cref{sectionmodularfunctorconstruction}.

A triplet $(a,b,c)$ such that $\Nusl(S_0^3,a,b,c)\neq 0$ will be called admissible.

\begin{proposition}\label{colorconditions}
    Let $a,b,c\in\{0,1,\dotsc,r-2\}$. Then $(a,b,c)$ is admissible for $\Nusl$ if and only if:
    \begin{itemize}
        \item $a+b+c$ is even;
        \item $a,b,c$ verify triangular inequalities, ie. $|a-b|\leq c \leq a+b$;
        \item $a+b+c<2r-2$.
    \end{itemize}
\end{proposition}

\begin{remark}
    As we are only dealing with genus $0$ surfaces, the $\SO$ and $\SU$ modular functors are isomorphic
    (see \cite[1.5]{blanchetTopologicalQuantumField1995}). We call it the $\SO$ modular functor because
    we use $\SO$ conventions of \cite{blanchetTopologicalQuantumField1995}.
\end{remark}


\subsubsection{Hermitian and integral structures}\label{subsubhermitianintegral}

\begin{definition}[Hermitian structure]
    Let $\Nu$ be a genus $0$ modular functor. A Hermitian structure on $\Nu$ is the data, for each colored surface $\Sigma$,
    of a Hermitian form $h_\Sigma$ on $\Nu(\Sigma)$ such that:
    \begin{description}
        \item[(1)] For each morphism $f:\Sigma_1\ra\Sigma_2$ in $\mathrm{C}$, $f$ is an isometry;
        \item[(2)] Each gluing isomorphism
        \begin{equation*}
            \Nu(S_{\pm},\underline{\lambda})\simeq \bigoplus_{\mu\in\Lambda}\Nu(S,\mu,\mu^\dagger,\underline{\lambda})
            \otimes \Nu(S_0^2,\mu,\mu^\dagger)^\vee
        \end{equation*} 
        as in \Cref{gluingunrooted}, is an isometry.
    \end{description}
\end{definition}

\begin{remark}
    If $\Nu$ has a Hermitian structure $h$, then for each $\Sigma=(S_0^n,\underline{\lambda})$, the quantum representation
    $\rho_n(\underline{\lambda})$ takes values in the group $\U{h_\Sigma}{\Nu(\Sigma)}$ of isomorphisms preserving $h_\Sigma$.
    If we choose an embedding $\Q(\zeta_r)\subset \C$, then $\rho_n(\underline{\lambda})$ takes values in $\PU{p}{q}$,
    where $(p,q)$ is the signature of $h_\Sigma\otimes\C$ for this embedding.
\end{remark}

\begin{definition}[Integral structure]
    Let $\Nu$ be a genus $0$ modular functor. An integral structure on $\Nu$ is the data, for each colored surface $\Sigma$,
    of a $\Z[\zeta_r]$-lattice $\NuO(\Sigma)\subset \Nu(\Sigma)$ such that:
    \begin{description}
        \item[(1)] For each morphism $f:\Sigma_1\ra\Sigma_2$ in $\mathrm{C}$, $f$ maps $\NuO(\Sigma_1)$ into $\NuO(\Sigma_2)$;
        \item[(2)] Each gluing isomorphism
        \begin{equation*}
            \Nu(S_{\pm},\underline{\lambda})\simeq \bigoplus_{\mu\in\Lambda}\Nu(S,\mu,\mu^\dagger,\underline{\lambda})
            \otimes \Nu(S_0^2,\mu,\mu^\dagger)^\vee
        \end{equation*} 
        as in \Cref{gluingunrooted}, induces an isomorphism:
        \begin{equation*}
            \NuO(S_{\pm},\underline{\lambda})\simeq \bigoplus_{\mu\in\Lambda}\NuO(S,\mu,\mu^\dagger,\underline{\lambda})
            \otimes \NuO(S_0^2,\mu,\mu^\dagger)^\vee;
        \end{equation*}
        \item[(3)] For each $\lambda\in\Lambda$, the canonical isomorphism $\Nu(S_0^2,\lambda,\lambda^\dagger)\simeq \Q(\zeta_r)$
        identifies $\NuO(S_0^2,\lambda,\lambda^\dagger)$ with $\Z[\zeta_r]$.
    \end{description}
    If, moreover, $\Nu$ is given a Hermitian structure $h$, we say that $\NuO$ is compatible with it if for each colored surface $\Sigma$,
    $h_\Sigma$ restricted to $\NuO(\Sigma)$ is a perfect $\Z[\zeta_r]$-valued Hermitian pairing.
\end{definition}

The abelian modular functors have straightforward compatible Hermitian and integral structures.

\begin{proposition}
    Let $r\geq 1$ be an odd integer. Then $\Nuab$ has a Hermitian structure $h$ given by $(a,b)\mapsto a\overline{b}$
    on each $\Nuab(\Sigma)=\Q(\zeta_r)$~or~$0$. This Hermitian structure is compatible with the integral structure:
    \begin{equation*}
        \NuabO(S_0^n,a_1,\dotsc,a_n) = \begin{cases}
            \Z[\zeta_r]\subset\Q(\zeta_r)\text{ if }a_1+\dotsb+a_n=0\;[2r] \\
            0\text{ otherwise}.
        \end{cases}
    \end{equation*}
\end{proposition}

The $\SO$ modular functors are given with a Hermitian structure. However, an integral structure is only known when the level $r$
is a prime number.

\begin{proposition}
    Let $r\geq 3$ be an odd integer. Then $\Nusl$ has a Hermitian structure $h$.
    If $r$ is prime, then $\Nusl$ has a compatible integral structure $\NuslO$.
\end{proposition}

The construction of the rooted Hermitian structure is given in \Cref{hermitianform}.
For the relation between rooted and unrooted Hermitian structures, see \Cref{rootinghermitianstructure}.
See \Cref{remarkintegral}, \Cref{skeingluingaxiom} and \Cref{theoremquantumgroupconstruction} for the integral structure at prime levels.


\subsection{Rooting of a modular functor}\label{subrooting}


\subsubsection{With the topological definition}\label{subsubrootingtopological}

The geometric construction will make the modular functors appear in a slightly different form,
that we call \textit{rooted}.

\begin{definition}\label{definitiondisks}
    Let $n\geq 0$. We define $D^n$ to be the closed disk with $n$ open disks removed.
    It is homeomorphic to a sphere with $n+1$ boundary components, $1$ outer and $n$ inner.
    We will call a disk any compact connected genus $0$ surface embedded in $\R^2$. Hence any disk is isomorphic to a $D^n$.

    We will denote by $D_n$ the open disk with $n$ points removed and $\oD_n$ the closed disk with $n$ interior points removed.
\end{definition}

We will often think of a disk $D$ as a framed sphere, where the framing is given by the inclusion $D\subset \R^2$.

\begin{proposition}
    The mapping class group $\Mod{D^n}$ is isomorphic to the group $UPB_n$ of pure framed braids on $n$ strands.
\end{proposition}

For details on the link between mapping class group of disks and braid groups see \cite[9.1.3]{farbPrimerMappingClass2011}.

We define a category of colored disks in a similar fashion to \Cref{definitioncolouredcategory}.

\begin{definition}\label{definitioncoloureddiskcategory}
    Let $\Lambda$ be a set of colors. The category of disks colored with $\Lambda$ is such that:
    \begin{description}
        \item[(1)] its objects are disks $D$ together with an identification
        $\varphi_B: B\simeq S^1$ and a color in $\Lambda$ for every component $B$ of $\partial D$.
        We will denote by $(\nu,\underline{\lambda})$ the colors, where $\nu$ is associated to the outer boundary;
        \item[(2)] its morphisms from $\Dc_1=(D^1,\varphi^1,\nu^1,\underline{\lambda}^1)$
        to $\Dc_2=(D^2,\varphi^2,\nu^2,\underline{\lambda}^2)$
        are homeomorphisms $f:D^1\lra D^2$ preserving orientation such for every inner component $B_1\subset\partial D^1$,
        its image $f(B_1)=D^2\subset \partial S_2$ is an inner component, $\lambda_{B_1}=\lambda_{B_2}$
        and $\varphi^2_{B_2}\circ f=\varphi^1_{B_1}$. We also ask that $\nu^1=\nu^2$.
        \item[(3)] the composition of $f_1:\Dc_0\lra\Dc_1$ and $f_2:\Dc_1\lra\Dc_2$ is simply $f_2\circ f_1$.
    \end{description}
\end{definition}

\begin{notation}
    We will often abbreviate $\Dc=(D,\varphi,\nu,\underline{\lambda})$ to $\Dc=(D,\nu,\underline{\lambda})$.
\end{notation}

The real distinction with colored surfaces is in the way gluings are represented.

For every colored disk $\Dc=(D,\nu,\underline{\lambda})$ together with a simple closed curve $\gamma$ in $\mathrm{int}(D)$,
we can cut $D$ along $\gamma$ to obtain two disks $D'$ and $D''$, where $D'$ has $\gamma$ as outer boundary.
For every color $\mu$, we will denote by $\Dc_{\mu}'$, respectively $\Dc_{\mu}''$ the disk $D'$ (respectively $D''$)
colored by $\nu,\underline{\lambda}$ at the boundary components it has in common with $D$
and by $\mu$ at the boundary component corresponding to $\gamma$.
See \Cref{cut_disk} for an example.

\begin{figure}
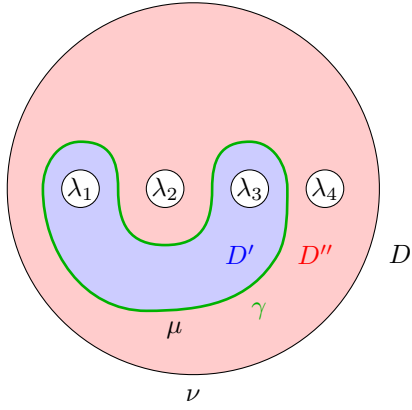

    \ctikzfig{cut_disk}
    \caption{An example of a cut of a disk $\Dc=(D,\nu,\underline{\lambda})$ with $D=D^4$ into $\Dc_{\mu}'=(D^2,\mu,(\lambda_1,\lambda_3))$
    and $\Dc_{\mu}''=(D^3,\nu,(\mu,\lambda_2,\lambda_4))$.}
    \label[figure]{cut_disk}
\end{figure}

\begin{definition}
    Let $\Dc=(D,\nu,\underline{\lambda})$ be a colored disk, then the associated colored surface is defined to be
    $\Sigma(\Dc)=(D,(\nu^\dagger,\lambda_1,\dotsc,\lambda_n))$.
\end{definition}

The reason for the use of $\nu^\dagger$ over $\nu$ is made obvious by the following definition.

\begin{definition}\label{definitionrooting}
    Let $\Nu$ be a genus $0$ modular functor of level $r$ with set of colors $\Lambda$.
    Denote by $\mathrm{C}$ the category of genus $0$ surfaces with set of colors $\Lambda$ and by $\mathrm{D}$ the category of colored disks
    with set of colors $\Lambda$.
    Then the associated rooted modular functor is:
    \[\begin{array}{rcl}
        \mathrm{D} &\lra     & \Q(\zeta_r)-\text{vector spaces} \\
        \Dc         &\mapsto & \Nu(\Sigma(\Dc)).
    \end{array}\]
    and will also be denoted $\Nu$.
    The additional gluing isomorphism translate as follows. Consider $\Dc=(D,\nu,\underline{\lambda})$ a colored disk together with
    a simple closed curve $\gamma$. Define $\Dc_{\mu}'$ and $\Dc_{\mu}''$ for each color $\mu$ as above. Then we have an isomorphism:
    \begin{equation}\label{gluingrooted}
        \Nu(\Dc)\simeq \bigoplus_{\mu\in\Lambda}\Nu(\Dc_{\mu}')\otimes\Nu(\Dc_{\mu}'')
    \end{equation}
    which commutes with the action of $\Mod{D'}\times\Mod{D''}$.
    The additional axioms become:
    \begin{description}
        \item[(1)] $\dim\Nu(D^0,\lambda,\varnothing)=1$ if $\lambda=0$ and $0$ otherwise;
        \item[(2)] $\dim\Nu(D^1,\lambda,\mu)=1$ if $\lambda=\mu$ and $0$ otherwise. Moreover, $\Nu(D^1,\lambda,\lambda)$
        is canonically isomorphic to $\Q(\zeta_r)$;
        \item[(3)] For every color $\lambda$, the left Dehn twist in $D^1$ acts on $\dim\Nu(D^1,\lambda,\lambda)$
        by multiplication by some $r$-th root of unity, denoted by $t_\lambda$.
    \end{description}
\end{definition}

Notice how the $(\cdot)^\dagger$ does not appear in the gluing or the additional axioms.


\subsubsection{Hermitian and integral structures in the rooted case}\label{subsubhemitianintegralrooted}

The Hermitian and integral structures translate to the rooted case. Any such structures on a genus $0$ modular functor induce a
structure on the rooting. However, the existence of a Hermitian or integral structure on the rooting does not imply the existence
of the same structure on the modular functor.

\begin{definition}[Rooted Hermitian structure]
    Let $\Nu$ be a genus $0$ modular functor. A rooted Hermitian structure on $\Nu$ is the data of, for each colored disk $\Dc$,
    of a Hermitian form $h_{\Dc}$ on $\Nu(\Dc)$ such that:
    \begin{description}
        \item[(1)] For each morphism $f:\Dc_1\ra\Dc_2$ in $\mathrm{D}$, $f$ is an isometry;
        \item[(2)] Each gluing isomorphism
        \begin{equation*}
            \Nu(\Dc)\simeq \bigoplus_{\mu\in\Lambda}\Nu(\Dc_{\mu}')\otimes\Nu(\Dc_{\mu}'')
        \end{equation*}
        as in \Cref{gluingrooted}, is an isometry.
    \end{description}
\end{definition}

\begin{definition}[Integral structure]
    Let $\Nu$ be a genus $0$ modular functor. A rooted integral structure on $\Nu$ is the data of, for each colored disk $\Dc$,
    of a $\Z[\zeta_r]$-lattice $\NuO(\Dc)\subset \Nu(\Dc)$ such that:
    \begin{description}
        \item[(1)] For each morphism $f:\Dc_1\ra\Dc_2$ in $\mathrm{D}$, $f$ maps $\NuO(\Dc_1)$ into $\NuO(\Dc_2)$;
        \item[(2)] Each gluing isomorphism
        \begin{equation*}
            \Nu(\Dc)\simeq \bigoplus_{\mu\in\Lambda}\Nu(\Dc_{\mu}')\otimes\Nu(\Dc_{\mu}'')
        \end{equation*}
        as in \Cref{gluingunrooted}, induces an isomorphism:
        \begin{equation*}
            \NuO(\Dc)\simeq \bigoplus_{\mu\in\Lambda}\NuO(\Dc_{\mu}')\otimes\NuO(\Dc_{\mu}'').
        \end{equation*}
        \item[(3)] For each $\lambda\in\Lambda$, the canonical isomorphism $\Nu(D^1,\lambda,\lambda)\simeq \Q(\zeta_r)$
        identifies $\NuO(D^1,\lambda,\lambda)$ to $\Z[\zeta_r]$.
    \end{description}
    If, moreover, $\Nu$ is given a rooted Hermitian structure $h$, we say that $\NuO$ is compatible with it
    if for each colored disk $\Dc$,
    $h_{\Dc}$ restricted to $\NuO(\Dc)$ is a perfect $\Z[\zeta_r]$-valued Hermitian pairing.
\end{definition}

\begin{remark}
    If $\Nu$ of level $r$ has a rooted Hermitian structure $h$, for $a\in \Q(\zeta_r)$ and $\phi:\Lambda\ra \N$ a map,
    we can define another rooted
    Hermitian structure $h^{a,\phi}$ on $\Nu$ by $h^{a,\phi}_{\Dc}=a^{-2m}h_{\Dc}$ for every $\Dc=(D,\nu,\underline{\lambda})$, where
    $m=\frac{\sum_i\phi(\lambda_i)-\phi(\nu)}{2}$.

    Similarly, if $\Nu$ has a rooted integral structure $\NuO$, then we can define another rooted integral structure $\NuOaphi$
    by $\NuOaphi(\Dc)=a^m\NuO(\Dc)$. Here we assume that $m$ is an integer whenever $\Nu(\Dc)$ is non zero.

    If $h$ and $\NuO$ are compatible, then so are $h^{a,\phi}$ and $\NuOaphi$.

    We will only be concerned with the case of the $\SO$ modular functors in the sequel. And in this case the function $\phi$
    will simply be the inclusion $\{0,1,\dotsc,r-2\}\subset\N$ so that $m$ is an integer whenever $\Nusl(\Dc)$ is non zero
    by \Cref{existenceSOmodularfunctors}.
\end{remark}

\begin{proposition}\label{rootinghermitianstructure}
    Let $\Nu$ be a genus $0$ modular functor. Any Hermitian structure $h$ on $\Nu$ induces a rooted Hermitian structure
    $h^{\mathrm{root}}$ by:
    \begin{equation*}
        h^{\mathrm{root}}_{\Dc}=h_{\Sigma(\Dc)}\otimes h^{-1}_{(S_0^2,\nu,\nu^\dagger)}
    \end{equation*}
    for every $\Dc=(D,\nu,\underline{\lambda})$, where $\Nu(S_0^2,\nu,\nu^\dagger)$ is identified with $\Q(\zeta_r)$.

    Any integral structure $\NuO$ on $\Nu$ induces a rooted integral structure $\NuOroot$ simply by
    $\NuOroot(\Dc)=\NuO(\Sigma(\Dc))$.

    If $h$ and $\NuO$ are compatible, then so are $h^{\mathrm{root}}$ and $\NuOroot$.
\end{proposition}

In the paper, we will often denote $h^{\mathrm{root}}$ and $\NuOroot$ simply by $h$ and $\NuO$.


\subsection{Framing of a rooted modular functor}


\subsubsection{Topological definition}

\begin{definition}
    For $n,m\geq 0$, let $D_m^n\subset\R^2$ be the closed disk with $n$ open disks and $m$ points removed, all disjoint.
\end{definition}

\begin{proposition}\label{mcgsection}
    Let $n\geq 0$, then blackboard framing induces a map:
    \begin{equation*}
        \PMod{D_n}\lra \Mod{D^{n}}
    \end{equation*}
    that is a section of projection $\Mod{D^{n}}\ra \PMod{D_n}$.

    Let us denote by $\Modp{m}{D^{n+m}}$ the group of orientation preserving homeomorphisms of $D^{n+m}$ up to isotopy,
    where we allow permutations of the $m$ last boundary components $B_1,\dotsc,B_m$ and require compatibility with
    some fixed identifications $S^1\simeq B_i$. 

    Denote by $\Modp{m}{D_{n+m}}$ the subgroup of homeomorphisms fixing the first $n$ marked points in $\Mod{D_{n+m}}$. 

    Then, as above, blackboard framing induces a map:
    \begin{equation*}
        \Modp{m}{D_{n+m}}\lra \Modp{m}{D^{n+m}}
    \end{equation*}
    that is a section of projection $\Modp{m}{D^{n+m}}\ra \Modp{m}{D_{n+m}}$.
\end{proposition}

\begin{definition}\label{definitionframing}
    Let $\Nu$ be a genus $0$ modular functor, with set of colors $\Lambda$.
    Let $n\geq 0$ and $\Dc=(D^{n},\nu,\underline{\lambda})$ be a colored disk.
    Then the first map of \Cref{mcgsection} induces a representation:
    \begin{equation*}
        \PMod{D_n}\lra \GL{\Nu(\Dc)}
    \end{equation*}
    that we call the framing of the quantum representation.
    
    Let $k\leq n$. If the colors $\lambda_{n-k+1},\dotsc,\lambda_{n}$ corresponding to the $k$ last boundary components are equal,
    then the second map of \Cref{mcgsection} induces a representation:
    \begin{equation*}
        \Modp{k}{D_n}\lra \GL{\Nu(\Dc)}.
    \end{equation*}
\end{definition}


\subsubsection{Geometric interpretation}

By \Cref{sectiongeometricdefinitions} and \Cref{mcgsection} we have maps $\pi_1(\Mrpb{n}{1}{r})\ra\pi_1(\Mrb{1+n}{r})$
and $\pi_1(\Mrp{1+n+m}{r}/S_m)\ra\pi_1(\Mrb{1+n+m}{r}/S_m)$. These maps depend on the choice of 
one of the $n+1$ sections to be the outer boundary of the disk.
Hence we will use more precise notations.

\begin{notation}
    Let us denote by $\Mrpbo{m}{n}{r}$ the space $\Mrpb{m}{n+1}{r}$, where the sections with automorphisms are numbered
    $0,\dotsc,n$. We think of the $0$-th section as the outer boundary of the disk.
    We shorten the notation to $\Mrbo{n}{r}$ if $m=0$ and $\Mrpo{m}{r}$ if $n=0$.
\end{notation}

\begin{proposition}
    Let $n\geq 2$, then the map $\pi_1(\Mrpo{n}{r})\ra\pi_1(\Mrbo{n}{r})$ induces a map
    \begin{equation*}
        \Mrpo{n}{r}\lra \Mrbo{n}{r}
    \end{equation*}
    of gerbes over $\Mrpo{n}{r}$.
    
    Let $n,m\geq 0$ such that $n+m\geq 2$. then the map $\pi_1(\Mrpo{n+m}{r}/S_m)\ra\pi_1(\Mrbo{n+m}{r}/S_m)$ induces a map
    \begin{equation*}
        \Mrpo{n+m}{r}/S_m\lra\Mrbo{n+m}{r}/S_m
    \end{equation*}
    of gerbes over $\Mrpo{n+m}{r}$.
\end{proposition}

\begin{definition}
    Let $\Nu$ be a genus $0$ modular functor, with set of colors $\Lambda$.
    Let $m,n\geq 0$ and $\nu,\lambda_1,\dotsc,\lambda_{n+m}$ be colors.
    Then pulling back $\Nu(\nu;\underline{\lambda})$ by the map $\Mrpo{n+m}{r}\ra\Mrbo{n+m}{r}$
    we get a flat vector bundle over $\Mrpo{n+m}{r}$, that we also denote $\Nu(\nu;\underline{\lambda})$.
    We call it the framing of the modular functor.

    Moreover, if $\lambda_{n+1}=\dotsb=\lambda_{n+m}=\mu$, then this bundle comes from a flat bundle
    $\Nu(\nu;\lambda_1,\dotsc,\lambda_n,\mu^{\bullet m})$ over $\Mrpo{n+m}{r}/S_m$.
\end{definition}

The monodromy of the flat bundle $\Nu(\nu;\underline{\lambda})$ over $\Mrpo{n+m}{r}$
is the framing of the associated quantum representation, as defined in \Cref{definitionframing}.

\begin{proposition}\label{propprojectiveequivalence}
    Let $n\geq 2$.
    Let $p: \Mrbo{n}{r}\lra \Mrpo{n}{r}$ be the projection and $s:\Mrpo{n}{r}\lra \Mrbo{n}{r}$ be the section given by blackboard framing.
    Let $\mathcal{E}$ be a $\Z[\zeta_r]$-local system on $\Mrbo{n}{r}$. Then there exists a $\Z[\zeta_r]$-flat line bundle $\Line$ on
    $\Mrbo{n}{r}$ with finite monodromy such that:
    \begin{equation*}
        p^*s^*\mathcal{E}\simeq \mathcal{E}\otimes \Line.
    \end{equation*}
\end{proposition}
\begin{proof}
    The flat bundles $p^*s^*\mathcal{E}$ and $\mathcal{E}$ are projectively equivalent.
    Let $\rho', \rho:\PModl{r}{D^n}\ra\GLn{d}{\Z[\zeta_r]}$
    be their associated monodromy representations. Then $\rho'\rho^{-1}$ is a central representation
    $\PModl{r}{D^n}\ra\Z[\zeta_r]^\times\ra\GLn{d}{\Z[\zeta_r]}$, and corresponds to a flat $\Z[\zeta_r]$-line bundle $\Line$ on $\Mrbo{n}{r}$.
    Clearly, $p^*s^*\mathcal{E}\simeq \mathcal{E}\otimes \Line$.
    Now any flat $\Z[\zeta_r]$-line bundle has finite monodromy.
\end{proof}

\begin{remark}\label{remarklevelframing}
    Let $r\geq 3$ be an odd integer. Although the modular functors $\Nuab$ and $\Nusl$ have level $2r$,
    the associated framed rooted flat vector bundles have monodromies of order dividing $r$ around boundary divisors.
    More precisely, for $b,a_1,\dotsc,a_n\in\Z/2r\Z$, the bundle $\Nuab(b;\ua)$ over $\Mrbo{n}{2r}$
    is the pullback of a bundle on $\Mrbo{n}{r}$. Similarly, for $b,a_1,\dotsc,a_n\in\{0,1,\dotsc,r-2\}$,
    $\Nusl(b;\ua)$ can be defined over $\Mrbo{n}{r}$. \textbf{Hence, form now on, when writing $\Nuab(b;\ua)$ or $\Nusl(b;\ua)$,
    we will mean the bundles over $\Mrbo{n}{r}$ and not $\Mrbo{n}{2r}$.}
\end{remark}


\subsubsection{Gluing, integral and hermitian structures for the framing}\label{subsubframingluing}

The framing changes nothing of the integral and hermitian structures on a rooted modular functor.
Here we discuss how the gluing map translates to the framing.
Let $n\geq 2$ and $n_1+n_2=n$ with $n_1\geq 2$ and $n_2\geq 1$. Then we have an an inclusion of boundary divisor:
\begin{equation*}
    q:\Mrpo{n_1}{r}\times\Mrpo{n_2+1}{r}\lra \Mrpo{n}{r}
\end{equation*}
induced by the gluing the $0$-th section of a curve in $\Mrpo{n_1}{r}$
to the $(n_2+1)$-th section of a curve in $\Mrpo{n_2+1}{r}$.
More precisely, we can see it as the composition:
\begin{equation*}
    q:\Mrpo{n_1}{r}\times\Mrpo{n_2+1}{r}\lra \Mrpo{n_1}{r}\times\Mrpbo{n_2}{1}{r}\lra \Mrpo{n}{r}
\end{equation*}
where the map $\Mrpo{n_2+1}{r}\lra \Mrpbo{n_2}{1}{r}$ is given by blackboard framing of the $(n_2+1)$-th strand.

If we use the map $q$ above for the gluing axiom in the framed case, everything works as in the rooted case.
For example, for $\Nu$ a genus $0$ modular functor and $b,a_1,\dotsc,a_n\in \Lambda$,
we have an isomorphism of flat bundle on $\Mrpo{n_1}{r}\times\Mrpo{n_2+1}{r}$:
\begin{equation}
    q^*\Nu(b;a_1,\dotsc,a_n)\simeq \bigoplus_c\Nu(c;a_1,\dotsc,a_{n_1},\mu)\otimes
    \Nu(b;c,a_{n_1+1},\dotsc,a_n).
\end{equation}


\section{Main results on the geometric construction}\label{sectionmainresults}


\subsection{\texorpdfstring{The $\SO$ quantum representations as Gauss-Manin connections}
{The SO quantum representations as Gauss-Manin connections}}

\begin{definition}
    We will denote $\sMrpbo{m}{n}{r}$ the locus of smooth curves in $\Mrpbo{m}{n}{r}$.
    It is a $(\mu_r)^{n+1}$-gerbe over $\sM{n+m}$.
\end{definition}

\begin{notation}
    We will denote by $\epsilon$ the flat $\Z$ line bundle over $\mathrm{B}S_m$ corresponding to the signature.
    We shall also denote $\epsilon$ any of its pullbacks.
\end{notation}

\begin{theorem}\label{theoremgeometricconstruction}
    Let $r\geq 3$ be an odd integer, $n\geq 2$ and $b,a_1,\dotsc,a_n\in\{0,1,\dotsc,r-2\}$,
    such that $m=\frac{a_1+\dotsb+a_n-b}{2}$ is an integer.
    Consider the $\Q(\zeta_r)$-local system $\Line=\Nuab(b;a_1,\dotsc,a_n,(-2)^{\bullet m})\otimes\epsilon$ 
    over $\sMrpo{n+m}{r}/S_m$.
    Let us consider the forgetful map:
    \begin{equation*}
        p:\sMrpo{n+m}{r}/S_m\lra \sMrpo{n}{r}.
    \end{equation*}
    Then the framing of the $\SO$ modular functor local system associated to $b,a_1,\dotsc,a_n$ over $\sMrpo{n}{r}$ is described by:
    \begin{equation}\label{equationgeometricconstruction}
        \Nusl(b;a_1,\dotsc,a_n) \simeq \mathrm{im}\left(R^mp_!\Line\lra R^mp_*\Line\right)
    \end{equation}
    where $p_!$ is the restricted pushforward and the connection is the Gauss-Manin connection.
\end{theorem}

The spaces $R^mp_*\Line$ in \Cref{theoremgeometricconstruction} have been considered by Silvotti in \cite{silvottiLocalSystemsComplement1994}
and Schechtman-Varchenko in \cite{schechtmanArrangementsHyperplanesLie1991a}.
Silvotti conjectured that the space of conformal blocks was the degree $0$ part of the weight filtration of $R^mp_*\Line$ and proved it for $n=2$.

In \cite{feiginAlgebraicEquationsSatisfied1995}, Feigin-Schechtman-Varchenko identify
conformal blocks as a subspace of $R^mp_*\Line$.
Using the results in \cite{schechtmanArrangementsHyperplanesLie1991a}, in \cite{ramadasHarderNarasimhanTraceUnitarity2009}
Ramadas proves that $\sltwo$ conformal blocks give square integrable forms in $R^m p_*\Line$.
Belkale then extended these results to all simple Lie algebras in \cite{belkaleUnitarityKZHitchin2012}.
In \cite{looijengaUnitaritySLConformal2009}, Looijenga
then showed that $\sltwo$ conformal blocks correspond to all $L^2$-integrable forms.

After the completion of this paper, we learned that in the recent paper \cite{belkaleConformalBlocksGenus2023a},
Belkale and Fakhruddin identify spaces of conformal blocks with the image of $R^mp_!\Line$ in $R^mp_*\Line$ for any Lie algebra.
See \Cref{subsectionrelations} for more comments.

Our main contribution to \Cref{theoremgeometricconstruction} is to use Martel's homological models to identify the TQFT representations 
with the image of $R^mp_!\Line$ in $R^mp_*\Line$, and doing so over $\Q(\zeta_r)$ with explicit isomorphisms with the Skein module and quantum groups constructions
of these representations.
This then enables our identification of the intersection form
and the gluing morphisms in the geometric context with the TQFT hermitian form and gluing isomorphisms (\Cref{theoremintersectionform,theoremgluinggeometric}),
as well as the study of rational Hodge structures of TQFT representations conducted in the upcoming work \cite{godfardHodgeTheoryConformalInpreparation}.

The proof relies heavily on results of Martel \cite{martelHomologicalModelUq2022}, elucidating and improving work of Felder-Wieczerkowski
\cite{felderTopologicalRepresentationsQuantum1991}.
These results identify some semi-relative version of $R^mp_*\Line$ with tensor products of Verma modules.
See \Cref{subsubbases,subsubUaction} below for the statements we will need from Martel's work. The main reason for using Martel's
homological approach over the cohomological approach of Schechtman-Varchenko, Ramadas and Looijenga
is that we want explicit isomorphism over $\Q(\zeta_r)$ and not just $\C$. Indeed, in the cohomological approach
$(-\zeta_r)^2$ needs to be $e^{\pm\frac{(r-1)\pi i}{r}}$.

The idea of considering the image of $R^mp_!\Line$ in $R^mp_*\Line$ comes from Bigelow's formulation in
\cite{bigelowHomologicalRepresentationsIwahori2004} of
Lawrence's work \cite{lawrenceHomologicalRepresentationsHecke1990a} on geometrical models for representations of Temperley-Lieb algebras.
See \Cref{subsubBigelow} below for examples of such models.

\begin{remark}
    The local system $\Line$ is always of rank $1$ with finite monodromy.
\end{remark}

\begin{corollary}\label{corollaryHodge}
    For any level $r$ and $a_1,\dotsc,a_n\in\{0,1,\dotsc,r-2\}$, the corresponding flat bundle $\Nusl(a_1,\dotsc,a_n)$
    over $\Mrb{n}{r}$ for the genus $0$ $\SO$ modular functor is of geometric origin and supports
    an integral variation of Hodge structures.
\end{corollary}

\begin{remark}
    The map $p$ has fibers isomorphic to unordered configuration spaces.
    More precisely, if $x\in \sMrpo{n}{r}$ represents the punctured genus $0$ Riemann surface $C$,
    then the fiber over $x$ is:
    \begin{equation*}
        \Conf{m}{C}=\{\{z_1,\dotsc z_m\}\subset C\;\mid\; z_i\neq z_j\text{ for }i\neq j\}.
    \end{equation*}
    One can see that $p$ is a topological fibration. Hence:
    \begin{align*}
        (R^mp_!\Line)_x&\simeq H_c^m(\Conf{m}{C};\Line_{\mid \Conf{m}{C}}) \\
        (R^mp_*\Line)_x&\simeq H^m(\Conf{m}{C};\Line_{\mid \Conf{m}{C}}).
    \end{align*}
    By Poincaré duality, the image in \Cref{equationgeometricconstruction}, can be translated to homologies:
    \begin{equation*}
        \Nusl(b;a_1,\dotsc,a_n)_x \simeq \mathrm{im}\left(H_m(\Conf{m}{C};\Line_{\mid \Conf{m}{C}})
        \lra H_m^{BM}(\Conf{m}{C};\Line_{\mid \Conf{m}{C}})\right).
    \end{equation*}
    Here the connection is again the Gauss-Manin connection, and $H_*^{BM}$ is Borel-Moore homology, see 
    \Cref{subsubBorelMoore} below.
\end{remark}

\begin{notation}\label{notationLine}
    Let $b,a_1,\dotsc,a_n\in\{0,1,\dotsc,r-2\}$,
    such that $m=\frac{a_1+\dotsb+a_n-b}{2}$ is an integer.
    Let $\Dc=(D^n,b,\underline{a})$ be the corresponding colored disk.
    We will denote the restriction of the local system
    $\Nuab(b;a_1,\dotsc,a_n;(-2)^{\bullet m})\otimes\epsilon$ to $\Conf{m}{C}$
    by $\Line_{\Dc}$.
\end{notation}

From the general theory of Gauss-Manin connection with local systems described briefly in \Cref{subsublocalsystemGaussManin} below,
one gets the following homological description of the connection.

\begin{theorem}[Homological action formulation]\label{theoremhomologicalconstruction}
    Let $r\geq 3$ be an odd integer, $n\geq 2$ and $b,a_1,\dotsc,a_n\in\{0,1,\dotsc,r-2\}$,
    such that $m=\frac{a_1+\dotsb+a_n-b}{2}$ is an integer.
    Let $\Dc=(D^n,b,a_1,\dotsc,a_n)$ be the associated colored disk.
    Let $C=D_n$.
    The mapping class group $\PMod{\oD_n}\simeq PB_n$ acts on $\Conf{m}{C}$ by
    $f(\{z_1,\dotsc,z_m\})=\{f(z_1),\dotsc,f(z_m)\}$. This action, together with the action of $\PMod{\oD_n}$
    on the local system $\Line_{\Dc}$ of \Cref{notationLine} make $\PMod{\oD_n}$ act on
    $H_m^{BM}(\Conf{m}{C};\Line_{\Dc})$, $H_m(\Conf{m}{C};\Line_{\Dc})$ and thus also on:
    \begin{equation*}
        \Hmid(\Conf{m}{C};\Line_{\Dc}) = \mathrm{im}\left(H_m(\Conf{m}{C};\Line_{\Dc})
        \lra H_m^{BM}(\Conf{m}{C};\Line_{\Dc})\right).
    \end{equation*}
    Then $\Hmid(\Conf{m}{C};\Line_{\Dc})$ is isomorphic to $\Nusl(\Dc)$ as a representation of $\PMod{\oD_n}$,
    where the action on $\Nusl(\Dc)$ is given by the full framing of \Cref{definitionframing}.
\end{theorem}


\subsubsection{The Hermitian structure}\label{subsubhermitianstructure}

For $M$ a manifold of real dimension $2m$, and $\Line$ a Hermitian local system over $\Q(\zeta_r)$ with form $h$,
we have an intersection form on $H_c^m(M;\Line)$ given by:
\[\begin{array}{rcl}
    H_c^m(M;\Line)\times H_c^m(M;\Line) &\lra & \Q(\zeta_r) \\
    (\alpha, \beta)                  &\longmapsto & \int_M h(\alpha\wedge\beta).
\end{array}\]

This form is sequilinear, skew-symmetric if $m$ is even and skew-antisymmetric if $m$ is odd. Then the quotient of $H_c^m(M;\Line)$
by the kernel of the form is exactly:
\begin{equation*}
    \Hmid(M;\Line)=\myim{H_c^m(M;\Line)\lra H^m(M;\Line)}.
\end{equation*}
Hence $\Hmid(M;\Line)$ has a perfect sequilinear form.

We will show that the intersection form of the geometric construction recovers the Hermitian structure of the
$\SO$ modular functor.

\begin{theorem}\label{theoremintersectionform}
    Let $r\geq 3$ be an odd integer, $n\geq 2$ and $b,a_1,\dotsc,a_n\in\{0,1,\dotsc,r-2\}$,
    such that $m=\frac{a_1+\dotsb+a_n-b}{2}$ is an integer. Let $\Dc=(D^n,b,\underline{a})$.

    Let $\Line=\Nuab(b;a_1,\dotsc,a_n;(-2)^{\bullet m})\otimes\epsilon$ over $\sMrpo{n+m}{r}/S_m$
    and $p:\sMrpo{n+m}{r}/S_m\ra \sMrpo{n}{r}$ as in \Cref{theoremgeometricconstruction}.
    Then the Hermitian structure on $\Line$ induces a perfect sesquilinear form $s_{\Dc}$ on the sheaf:
    \begin{equation*}
        \sheafHmid = \mathrm{im}\left(R^mp_!\Line\lra R^mp_*\Line\right).
    \end{equation*}
    Now let $h$ be the Hermitian structure on $\Nusl$.
    Then under the isomorphism $\sheafHmid\simeq \Nusl(b;a_1,\dotsc,a_n)$, we have the equality:
    \begin{equation*}
        s_{\Dc}=(\zeta_r^2-\zeta_r^{-2})^mh_{\Dc}.
    \end{equation*}
\end{theorem}

The idea that the Jones polynomial can be extracted from the intersection form
was first studied by Bigelow in \cite{bigelowHomologicalDefinitionJones2002}.
In a way, our result extends his result to the whole $\SO$ TQFT.
Note that our proof relies on the same "noodle" homology classes used by Bigelow,
but is somewhat different: we first identify the representation at generic $q$
and then use unicity of the form to reduce the proof to a single simple
intersection computation. See \Cref{subsubBigelow} for details.

We will also show that the intersection form polarizes the Hodge structures.

\begin{definition}[Polarization]\label{definitionpolarization}
Let $(E,\nabla)\ra X$ be an integral variation of Hodge structures defined on $\OK$ and of weight $m$.
A polarization for $(E,\nabla)$ is a sesquilinear for $Q$ on $E$ defined on $\OK$ such that:
\begin{enumerate}
    \item it is flat ie $\nabla Q=0$;
    \item it is hermitian if $m$ is even and skew-hermitian if $m$ is odd;
    \item for each choice of embedding $K\hookrightarrow\C$, the Hodge decomposition is orthogonal for $Q$
    and for each $p$, $q$, $i^{p-q}Q$ is positive definite on $H^{p,q}$.
\end{enumerate}
We will call a polarized integral variation of Hodge structures, or PIVHS for short a triplet $(E,\nabla,Q)\ra X$
such that $Q$ is a polarization of the IVHS $(E,\nabla)\ra X$.
\end{definition}

\begin{remark}
Usually, when discussing polarizations one asks that $Q$ is bilinear, not sesquilinear.
That is possible when one considers real structures. We can not do this here.
If we denote by $Q'$ a bilinear polarization on a Hodge structure with real structure in the usual sense,
then here we are considering the sesquilinear form $Q(u,v)=Q'(u,\overline{v})$.
\end{remark}

\begin{theorem}\label{intersectionformpolarizes}
    With the notations of \Cref{theoremintersectionform},
    $(-1)^{\frac{m(m-1)}{2}}s_{\Dc}$ is a polarization for the Hodge structure on $\Nusl(b;a_1,\dotsc,a_n)$.
    In particular, if $\Nusl(b;a_1,\dotsc,a_n)=\bigoplus_{p+q=m}H^{p,q}$ is the Hodge decomposition,
    then it is orthogonal for $h_{\Dc}$
    and $(-1)^{\frac{m(m-1)}{2}}i^{p-q}(\zeta_r^2-\zeta_r^{-2})^mh_{\Dc}$ is positive definite on $H^{p,q}$.
\end{theorem}


\subsubsection{(Rooted) Integral structures}

Let $M$ be a manifold of real dimension $2m$, and $\Line$ a $\Q(\zeta_r)$-local system on $M$
with a $\Z[\zeta_r]$-integral structure $\Line^{\mathcal{O}}$.
Then the $\Z[\zeta_r]$-module:
\begin{equation*}
    \overline{H}_c^m(M;\Line^{\mathcal{O}})=H_c^m(M;\Line^{\mathcal{O}})/\mathrm{torsion}
\end{equation*}
is a lattice in $H_c^m(M;\Line)$. Similarly, the $\Z[\zeta_r]$-module:
\begin{equation*}
    \overline{H}^m(M;\Line^{\mathcal{O}})=H^m(M;\Line^{\mathcal{O}})/\mathrm{torsion}
\end{equation*}
is a lattice in $H^m(M;\Line)$.
These induce two Poincaré dual lattices in $\Hmid(M;\Line)$:
\begin{align*}
    \Hmid_1(M;\Line^{\mathcal{O}})&=\myim{ \overline{H}_c^m(M;\Line^{\mathcal{O}})\lra H^m(M;\Line)}/\mathrm{torsion} \\
    \Hmid_2(M;\Line^{\mathcal{O}})&=\overline{H}^m(M;\Line^{\mathcal{O}})\cap\Hmid(M;\Line).
\end{align*}
Notice that $\Hmid_1(M;\Line^{\mathcal{O}})\subset \Hmid_2(M;\Line^{\mathcal{O}})$.

\begin{theorem}\label{theoremintegralstructure}
    Let $r\geq 3$ be an odd integer, $n\geq 2$ and $b,a_1,\dotsc,a_n\in\{0,1,\dotsc,r-2\}$,
    such that $m=\frac{a_1+\dotsb+a_n-b}{2}$ is an integer. Set $\Dc=(D^n,b,\underline{a})$.

    Let $\Line=\Nuab(b;a_1,\dotsc,a_n;(-2)^{\bullet m})\otimes\epsilon$ over $\sMrpo{n+m}{r}/S_m$
    and $p:\sMrpo{n+m}{r}/S_m\ra \sMrpo{n}{r}$ as in \Cref{theoremgeometricconstruction}.
    Then the integral structure $\Line^{\mathcal{O}}$ on $\Line$ induces two sheaves of $\Z[\zeta_r]$-modules:
    \begin{align*}
        \sheafHmid_1^{\mathcal{O}} &= \mathrm{im}\left(R^mp_!\Line^{\mathcal{O}}\lra R^mp_*\Line\right)/\text{torsion}\\
        \sheafHmid_2^{\mathcal{O}} &= \mathrm{im}\left(R^mp_!\Line\lra R^mp_*\Line\right) \cap R^mp_*\Line^{\mathcal{O}}
    \end{align*}
    that are lattices in $\sheafHmid$, with $\sheafHmid_1^{\mathcal{O}}\subset \sheafHmid_2^{\mathcal{O}}$.
    If $r$ is prime, under the isomorphism $\sheafHmid\simeq \Nusl(b;a_1,\dotsc,a_n)$, we have the equalities:
    \begin{align*}
        \sheafHmid_1^{\mathcal{O}}&=(\zeta_r^2-\zeta_r^{-2})^{-m}\NuslO(b;a_1,\dotsc,a_n) \\
        \sheafHmid_2^{\mathcal{O}}&=\NuslO(b;a_1,\dotsc,a_n).
    \end{align*}
\end{theorem}

\begin{remark}
    Notice that when $r$ is prime $\sheafHmid_1^{\mathcal{O}}=(\zeta_r^2-\zeta_r^{-2})^{-m}\sheafHmid_2^{\mathcal{O}}$ for all $\Dc$.
    When $r$ is not prime, this relation still holds in the cases where $\Dc$ is of the form $(D^n,b,1,\dotsc,1)$,
    as can be seen by studying the formula for the map $\mathrm{red}_\alpha$ of \Cref{imageFtoEquantum} in the case $\alpha=1$.
    However, the precise relation between the $\sheafHmid_1^{\mathcal{O}}$ and $\sheafHmid_2^{\mathcal{O}}$ is not clear to the author
    for general $\Dc$ when $r$ is not prime.
\end{remark}

\subsection{The gluing property in the geometric context}


\subsubsection{With configuration spaces}\label{subsubgluinghomological}

We first describe the geometric counterpart to the gluing morphisms with configuration spaces as in \Cref{theoremhomologicalconstruction}.

Let $r\geq 3$ be an odd integer, $n\geq 2$ and $b,a_1,\dotsc,a_n\in\{0,1,\dotsc,r-2\}$,
such that $m=\frac{a_1+\dotsb+a_n-b}{2}$ is an integer.
Let $\Dc=(D^n,b,a_1,\dotsc,a_n)$ be the associated colored disk.

Let $\gamma$ be a simple closed curve in $\mathrm{int}(D^n)$.
Define $\Dc_{c}'$ and $\Dc_{c}''$ for each $0\leq c\leq r-2$ as in \Cref{subsubrootingtopological}.

Let $C=D^n\setminus\partial D^n$, $C'=D'\setminus\partial D'$ and $C''=D''\setminus\partial D''$.

Let $C_\gamma = C\setminus \gamma$.
One has $C_\gamma\simeq C'\sqcup C''$. We then have an inclusion:
\begin{equation*}
    \bigsqcup_{m_1+m_2=m}\Conf{m_1}{C'}\times\Conf{m_2}{C''}=\Conf{m}{C_\gamma}\subset \Conf{m}{C}.
\end{equation*}

By the gluing property of \Cref{gluingunrooted} for the abelian modular functor,
the local system $\Line_{\Dc}$ restricted to $\Conf{m}{C_\gamma}$ splits into:
\begin{equation*}
    \bigsqcup_{m_1+m_2=m} \Line_{\Dc_{c}'}\otimes \Line_{\Dc_{c}''}
\end{equation*}
where $c$ is such that $m_1=\frac{a_1+\dotsb+a_{n_1}-c}{2}$ or equivalently $m_2=\frac{a_{n_1+1}+\dotsb+a_{n}+c-b}{2}$.

The cellular complex of \Cref{subsubcellularcomplex} shows that $\Conf{m_1}{C'}$
has no cohomology in degrees $>m_1$. Similarly for $\Conf{m_2}{C''}$ in degrees $>m_2$.
From this we get the Künneth formula:
\begin{multline*}
    H_m(\Conf{m}{C_\gamma};\Line_{\Dc})\\
    =\bigoplus_{m_1+m_2=m}H_{m_1}(\Conf{m_1}{C'};\Line_{\Dc_{c}'})
    \otimes H_{m_2}(\Conf{m_2}{C''};\Line_{\Dc_{c}''}).
\end{multline*}

Now, the inclusion map:
\begin{equation}\label{inclusionmapgluingproperty}
    H_m(\Conf{m}{C_\gamma};\Line_{\Dc})\lra H_m(\Conf{m}{C};\Line_{\Dc})
\end{equation}
induces the gluing map as follows.

\begin{theorem}\label{theoremgluinghomological}
    The diagram of $\PMod{D'}\times\PMod{D''}$-representations:
    \[\begin{tikzcd}
        {\bigoplus\limits_{m_1+m_2=m}H_{m_1}(\Conf{m_1}{C'};\Line_{\Dc_{c}'})     \otimes H_{m_2}(\Conf{m_2}{C''};\Line_{\Dc_{c}''})} & {H_m(\Conf{m}{C};\Line_{\Dc})} \\
        {\bigoplus\limits_{m_1+m_2=m}\Hmid(\Conf{m_1}{C'};\Line_{\Dc_{c}'})     \otimes \Hmid(\Conf{m_2}{C''};\Line_{\Dc_{c}''})} & {\Hmid(\Conf{m}{C};\Line_{\Dc})} \\
        {\bigoplus_{c} \Nusl(\Dc_{c}')\otimes \Nusl(\Dc_{c}'')} & {\Nusl(\Dc)}
        \arrow["{\mathbf{(G)}}", tail reversed, from=3-1, to=3-2]
        \arrow[two heads, from=1-1, to=2-1]
        \arrow[two heads, from=1-2, to=2-2]
        \arrow[from=1-1, to=1-2]
        \arrow[tail reversed, from=2-1, to=3-1]
        \arrow[tail reversed, from=2-2, to=3-2]
    \end{tikzcd}\]
    commutes, where two-sided arrows are isomorphisms and
    $\mathbf{(G)}$ is the gluing map.
\end{theorem}

\begin{remark}
    The statement of the theorem contains some vanishing properties. Indeed, if $c<0$ or $c>r-2$,
    then the theorem says that $\Hmid(\Conf{m_1}{C'};\Line_{\Dc_{c}'})=0$ or
    $\Hmid(\Conf{m_2}{C''};\Line_{\Dc_{c}''})=0$.
\end{remark}

A geometric gluing morphism for $sl_2$ representations was formulated along the same lines by Silvotti
in \cite{silvottiFactorizationPropertyCohomology1998}. The two main differences is that Silvotti works
in a generic case (generic parameter $q$) whereas we work with roots of unity ($q=\zeta_r^2$),
and we work with finite-dimensional modules over $U_{\zeta_r}$ where Silvotti works with Verma modules over $sl_2$.
Hence we need to consider the quotient homology $\Hmid(\Conf{m}{C};\Line_{\Dc})$ where Silvotti
uses $H^m(\Conf{m}{C};\Line_{\Dc})$. Apart from these differences, the proof is similar and
quite straightforward once the isomorphism of \Cref{theoremgeometricconstruction} is described.
See \Cref{subsubmainproofs} below.


\subsubsection{With the local systems}

Let $r\geq 3$ be an odd integer, $n\geq 2$ and $b,a_1,\dotsc,a_n\in\{0,1,\dotsc,r-2\}$,
such that $m=\frac{a_1+\dotsb+a_n-b}{2}$ is an integer.

Let $\Mp$ be the locus in $\Mrpo{n}{r}$ of curves with at most one nodal singularity,
and $\Mp_m$ be the same locus in $\Mrpo{n+m}{r}/S_m$. Denote by $p$ the forgetful map $\sMrpo{n+m}{r}/S_m\ra \sMrpo{n}{r}$
and by $p'$ the map $\Mp_m\ra\Mp$.

Set $\Line=\Nuab(b;a_1,\dotsc,a_n;(-2)^{\bullet m})\otimes\epsilon$ on $\sMrpo{n+m}{r}/S_m$
and $\Line'$ the same representation as a local system on $\Mp_m$.

Set:
\begin{equation*}
    \sheafHmid = \mathrm{im}\left(R^mp_!\Line\lra R^mp_*\Line\right).
\end{equation*}
As a local system, $\sheafHmid$ extends to $\Mp$. We denote this extension by $\sheafHmid'$.

Let $n=n_1+n_2$ with $n_1\geq 2$ and $n_2\geq 1$. Let $\B\subset \Mp$ be the image 
of $\sMrpo{n_1}{r}\times\sMrpo{n_2+1}{r}$ under the map gluing the $0$-th section of a curve in $\sMrpo{n_1}{r}$
to the $(n_2+1)$-th section of a curve in $\sMrpo{n_2+1}{r}$.
Then we have a decomposition:
\begin{equation}
    p'^{-1}(\B) = \bigsqcup_{m_1+m_2=m} \sMrpo{n_1+m_1}{r}/S_{m_1}\times\sMrpo{n_2+1+m_2}{r}/S_{m_2}.\footnote{Actually, the right hand side is a $\mu_r$-gerbe over the left hand side. But this is unimportant here.}
\end{equation}
Under this decomposition, by the gluing property of \Cref{gluinggeometric}, $\Line'$ decomposes as:
\begin{equation*}
    \Line'_{\mid\B} = \bigsqcup_{m_1+m_2=m} \Line^1_{m_1}\otimes \Line^2_{m_2}
\end{equation*}
where:
\begin{align*}
    \Line^1_{m_1}&=\Nuab(c;a_1,\dotsc,a_{n_1};(-2)^{\bullet m_1})\otimes\epsilon \\
    \Line^2_{m_2}&=\Nuab(b;a_{n_1+1},\dotsc,a_n,c;(-2)^{\bullet m_2})\otimes\epsilon.
\end{align*}
Here $c$ is such that $m_1=\frac{a_1+\dotsb+a_{n_1}-c}{2}$ or equivalently $m_2=\frac{a_{n_1+1}+\dotsb+a_{n}+c-b}{2}$.

Let $\sheafHmid^1_{m_1}$, respectively $\sheafHmid^2_{m_2}$, be the sheaf over $\sMrpo{n_1}{r}$, respectively 
$\sMrpo{n_2}{r}$, given by \Cref{equationgeometricconstruction} from $\Line^1_{m_1}$, respectively $\Line^2_{m_2}$.

Then, as before, by the Künneth formula, one has:
\begin{equation*}
    \left(R^mp'_!\Line'\right)_{\mid \B}=\bigoplus_{m_1+m_2=m} R^{m_1}p'_!\Line^1_{m_1} \otimes R^{m_2}p'_!\Line^2_{m_2}
\end{equation*}
and thus we have a surjective map:
\begin{equation*}
    \left(R^mp'_!\Line'\right)_{\mid \B}\twoheadrightarrow \bigoplus_{m_1+m_2=m} \sheafHmid^1_{m_1}\otimes \sheafHmid^2_{m_2}.
\end{equation*}

Moreover, the specialization maps of the fibers of $p'$ over the locus of singular curves in $\mathcal{M}'$ induce a morphism of sheaves:
\begin{equation}\label{equationgeometricmap}
    R^mp'_!\Line'\lra \sheafHmid'.
\end{equation}

This is essentially the map of \Cref{inclusionmapgluingproperty} in families. However, here the target is $\sheafHmid'$,
because $\sheafHmid$ extends to $\mathcal{M}'$ but $R^mp_!\Line$ may not.
See \Cref{subproofgluinghodge} for a description of specialisation maps.
Now, notice that by \Cref{theoremgeometricconstruction}, we have isomorphisms:
\begin{align*}
    \sheafHmid' &\simeq \Nusl(b;a_1,\dotsc,a_n) \\
    \sheafHmid^1_{m_1}      &\simeq \Nusl(c;a_1,\dotsc,a_{n_1}) \\
    \sheafHmid^2_{m_2}      &\simeq \Nusl(b;a_{n_1+1},\dotsc,a_n,c).
\end{align*}

The gluing axiom tells us there is an isomorphism:
\begin{equation*}
    \Nusl(b;a_1,\dotsc,a_n)_{\mid\B}\simeq \bigoplus_c \Nusl(c;a_1,\dotsc,a_{n_1})\otimes \Nusl(b;a_{n_1+1},\dotsc,a_n,c).
\end{equation*}

\begin{theorem}\label{theoremgluinggeometric}
    The diagram:
    \[\begin{tikzcd}
        & {\left(R^mp'_!\Line'_{\mid\B}\right)} \\
        {\sheafHmid'_{\mid\B}} & {\bigoplus\limits_{m_1+m_2=m} \sheafHmid^1_{m_1}\otimes \sheafHmid^2_{m_2}} \\
        {\Nusl(b;a_1,\dotsc,a_n)_{\mid\B}} & {\bigoplus\limits_c \Nusl(c;a_1,\dotsc,a_{n_1})\otimes \Nusl(b;a_{n_1+1},\dotsc,a_n,c)}
        \arrow[two heads, from=1-2, to=2-2]
        \arrow["{\Cref{equationgeometricmap}}"', from=1-2, to=2-1]
        \arrow[tail reversed, from=2-1, to=3-1]
        \arrow[tail reversed, from=2-2, to=3-2]
        \arrow[tail reversed, from=3-1, to=3-2]
    \end{tikzcd}\]
    commutes, where the two-sided arrows are isomorphisms.
    In other terms, the map of \Cref{equationgeometricmap} gives a geometric realization of the
    gluing map.
\end{theorem}

From this formulation, we deduce that the gluing isomorphisms preserve the Hodge structures.

\begin{theorem}\label{gluingishodge}
    With the notations of \Cref{theoremgluinggeometric}, 
    the gluing map:
    \begin{equation*}
      \Nusl(b;a_1,\dotsc,a_n)_{\mid\B}\simeq \bigoplus_c \Nusl(c;a_1,\dotsc,a_{n_1})\otimes \Nusl(b;a_{n_1+1},\dotsc,a_n,c)
    \end{equation*}
    is an isomorphism of rational variations of Hodge structures.
\end{theorem}

When $r$ is prime, it is an isomorphism of integral variations of Hodge structures.
See \Cref{subsubgluinghermitianintegral} below for comments on the compatibility of lattices with gluing maps.

After the completion of this paper we learned that in the recent paper \cite{belkaleMotivicFactorisationKZ2023},
Belkale, Fakhruddin and Mukhopadhyay study geometric gluing isomorphisms of genus $0$ conformal blocks for any simple Lie algebra,
prove that they preserve Hodge
structures and give an algorithm to compute corresponding Hodge numbers in genus $0$ for $\mathfrak{sl}_n$.
See \Cref{subsectionrelations} for more comments.


\subsubsection{Compatibility with Hermitian and integral structures}\label{subsubgluinghermitianintegral}

We take the notations of \Cref{subsubgluinghomological}

\begin{remark}[Gluing maps and intersection form]
    The morphism:
    \begin{equation*}
        \bigoplus\limits_{m_1+m_2=m}H_{m_1}(\Conf{m_1}{C'};\Line_{\Dc_{c}'})
        \otimes H_{m_2}(\Conf{m_2}{C''};\Line_{\Dc_{c}''})\lra H_m(\Conf{m}{C};\Line_{\Dc})
    \end{equation*}
    as in \Cref{theoremgluinghomological} is compatible with intersection forms.
    This identifies intersection forms on either side of the induced isomorphism:
    \begin{equation}\label{equationisogluinggeometric}
        \Hmid(\Conf{m}{C};\Line_{\Dc})\simeq \bigoplus\limits_{m_1+m_2=m}\Hmid(\Conf{m_1}{C'};\Line_{\Dc_{c}'})
        \otimes \Hmid(\Conf{m_2}{C''};\Line_{\Dc_{c}''}).
    \end{equation}
    So that:
    \begin{equation*}
        s_{\Dc}=\bigoplus_{0\leq c\leq r-2}s_{\Dc_{c}'}\otimes s_{\Dc_{c}''}
    \end{equation*}
    and hence, by \Cref{theoremintersectionform}:
    \begin{equation*}
        h_{\Dc}=\bigoplus_{0\leq c\leq r-2}h_{\Dc_{c}'}\otimes h_{\Dc_{c}''}.
    \end{equation*}
    This recovers geometrically the compatibility of gluing maps with the
    Hermitian structure.
\end{remark}

As for the $\Z[\zeta_r]$-lattices, the diagram of \Cref{theoremgluinghomological} and its Poincaré dual induce inclusions:
\begin{align*}
    \Hmid_1(\Conf{m}{C};\Line_{\Dc}^\mathcal{O})&\supseteq \bigoplus\limits_{m_1+m_2=m}\Hmid_1(\Conf{m_1}{C'};\Line_{\Dc_{c}'}^\mathcal{O})
    \otimes \Hmid_1(\Conf{m_2}{C''};\Line_{\Dc_{c}''}^\mathcal{O}), \\
    \Hmid_2(\Conf{m}{C};\Line_{\Dc}^\mathcal{O})&\subseteq \bigoplus\limits_{m_1+m_2=m}\Hmid_2(\Conf{m_1}{C'};\Line_{\Dc_{c}'}^\mathcal{O})
    \otimes \Hmid_2(\Conf{m_2}{C''};\Line_{\Dc_{c}''}^\mathcal{O}).
\end{align*}
If $r$ is prime, these inclusions are equalities by \Cref{theoremgluinghomological,theoremintegralstructure}. 


\section{\texorpdfstring{Constructions of the $\SO$ modular functors}
{Constructions of the SO(3) modular functors}}\label{sectionmodularfunctorconstruction}


\subsection{The Skein module construction}\label{subsectionskein}


\begin{definition}
    Let $M$ be a $3$-manifold and $P\subset \partial M$ a set of disjoint intervals $[0,1]$ on the boundary of $M$,
    called a marking of $M$. A tangle in $(M,P)$ is an isotopy class of embedded surface $T\hookrightarrow M$
    such that each component of $T$ is the image of:
    \begin{description}
        \item[(i)] either a smooth embedding of $S^1\times [0,1]$ in $\mathrm{int}(M)$ ;
        \item[(ii)] or a smooth embedding of $([0,1]\times [0,1], \{0,1\}\times[0,1])$ in $(M,P)$ that maps $\{0,1\}\times[0,1]$
        isomorphically to two components of $P$.
    \end{description}
    The Skein module of $(M,P)$ is the quotient:
    \begin{equation*}
        S(M,P) = \left(\bigoplus_T \Z[q^{\frac{1}{2}},q^{-\frac{1}{2}}]\cdot [T]\right) / R
    \end{equation*}
    where the sum is over every tangle in $(M,P)$ and the subspace $R$ is generated by the Skein relations in
    \Cref{skein_relations}.
    These relations are to be understood as follows:
    \begin{description}
        \item[(i)] If $T_\times$, $T_=$ and $T_{||}$ are tangles that differ only in a ball
    $B^3\subset \mathrm{int}(M)$ and in this ball are as in \Cref{skein_relations}, then $T_\times= q^{\frac{1}{2}}T_{||} + q^{-\frac{1}{2}}T_=$ ;
        \item[(ii)] If $T_\circ$ and $T_\varnothing$ are tangles that differ only in a ball
        $B^3\subset \mathrm{int}(M)$ and in this ball are as in \Cref{skein_relations}, then $T_\circ = -(q +q^{-1}) T_\varnothing$.
    \end{description}
\end{definition}

\begin{figure}
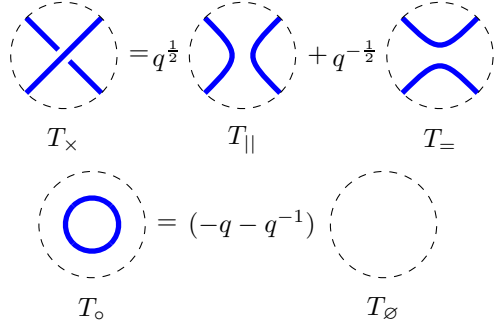

    \ctikzfig{skein_relations}
    \caption{Skein relations.}
    \label[figure]{skein_relations}
\end{figure}

\begin{remark}
    The Skein module of $(B,\varnothing)$ is generated by the empty tangle,
    ie $S(B,\varnothing) = \Z[q^{\frac{1}{2}},q^{-\frac{1}{2}}]\cdot [\varnothing]$.
\end{remark}

\begin{definition}
    Let $n\geq 0$. The Temperley-Lieb algebra $TL_n$ is the Skein module of $([0,1]^3,P)$ where $P$ is made of $2n$ marks,
    $n$ of which are evenly spaced in $\{1\}\times [0,1]\times \bigl\{\frac{1}{2}\bigr\}$ on the top face of the cube $[0,1]^3$
    and the $n$ others are evenly spaced in $\{0\}\times [0,1]\times \bigl\{\frac{1}{2}\bigr\}$ on the bottom face of the cube $[0,1]^3$,
    as in \Cref{Temperley_Lieb}.

    It is an algebra for the product given by the stacking of cubes.

    For $n,m\geq 0$, the gluing of cubes on a side face induces a map of algebras:
    \begin{equation*}
        \times: T_n\otimes T_m \lra T_{n+m}.
    \end{equation*}
\end{definition}

\begin{figure}
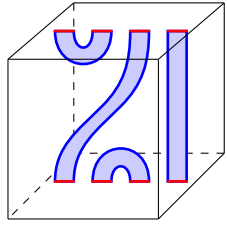

    \ctikzfig{Temperley_Lieb}
    \caption{The marked cube $([0,1]^3,P)$ for the Temperley Lieb algebra $T_4$ and a tangle in it.
    Only the boundary of the tangle is represented.}
    \label[figure]{Temperley_Lieb}
\end{figure}

\begin{notation}
    We will use the notations $[n]=(q^{n}-q^{-n})/(q-q^{-1})$ for $n\in\Z$, $[n]!=[n][n-1]\dotsb [1]$ for $n\in\N$, and
    $\begin{bmatrix} m\\n \end{bmatrix}=\frac{[m][m-1]\dotsb[m-n+1]}{[n]!}$ for $m\in\Z$ and $n\in\N$. All of these are elements of $\Z[q,q^{-1}]$.
    Moreover, if we specialize $q^{\frac{1}{2}}$ to $-\zeta_r$ where $\zeta_r$ is a root of unity of \textit{prime} order $r$,
    then for all $n$ not $0$ modulo $r$, $[n]$ is invertible in $\Z[\zeta_r]$.

    We will denote by $A_k$ the ring $\Z[q^{\frac{1}{2}},q^{-\frac{1}{2}},[1]^{-1},\dotsc,[k]^{-1}]$.
\end{notation}

\begin{definition}
    Let $n\geq 0$, for $k\in\{1,\dotsc,n-1\}$, let $e_k\in TL_n\otimes A_k$ be the element described on \Cref{e_k}.
    We define elements $f_n\in TL_n$ by induction:
    \begin{description}
        \item[(i)] $f_0$ is the class of the empty tangle ;
        \item[(ii)] $f_{n+1} =  f_n\times \id + \frac{[n]}{[n+1]}(f_n\times 1)e_n(f_n\times 1)$.
    \end{description}
\end{definition}

\begin{figure}
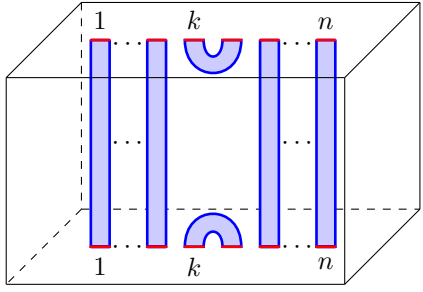

    \ctikzfig{e_k}
    \caption{The element $e_k$ of $T_n$.}
    \label[figure]{e_k}
\end{figure}

\begin{proposition}
    Let $n\geq 0$. Then the $e_k$ generate $TL_n$ as an algebra.

    The element $f_n\in TL_n\otimes A_n$ is idempotent and a projector of rank $1$. Moreover for each $k$, $e_kf_n=0$.
\end{proposition}

See \cite[3.]{blanchetTopologicalQuantumField1995} and \cite[8.9]{ohtsuki2001quantum} for a proof.

\begin{definition}\label{definitionSD}
    Let $\Dc=(D^n,\nu,\underline{\lambda})$ be a colored disk.
    Let $(B,P)$ be a cube $[0,1]^3$ with $\nu+\sum_i\lambda_i$ markings, $\nu$ of which are evenly spaced
    in $\{0\}\times [0,1]\times \bigl\{\frac{1}{2}\bigr\}$ on the bottom face
    and all the others are evenly spaced in $\{1\}\times [0,1]\times \bigl\{\frac{1}{2}\bigr\}$ on the top face.

    Let us partition the markings on the top of $B$ according to the $\lambda_i$, as in \Cref{colored_skein_module}.
    Then the Skein module $S(B,P)$ is acted upon by $TL_{\nu}\otimes \bigotimes_i TL_{\lambda_i}$.

    We define $S(\Dc)$ to be the direct summand of $S(B,P)\otimes A_{r-1}$ corresponding to the action of the idempotent 
    $f_{\nu}\otimes\bigotimes_i f_{\lambda_i}$.

    Let $S_{\zeta_r}(\Dc)$ be $S(\Dc)\otimes_{\Z[q^{\frac{1}{2}},q^{-\frac{1}{2}}]}\Q[\zeta_r]$ where $\Q[\zeta_r]$ the cyclotomic
    field of $r$-th roots of unity. Here we set $q^{\frac{1}{2}}$ to $-\zeta_r$.
\end{definition}

\begin{figure}
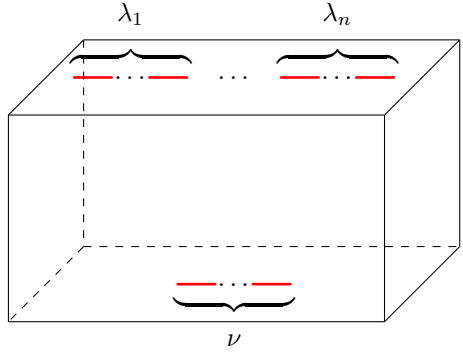

    \ctikzfig{colored_skein_module}
    \caption{The marked cube associated to $\Dc=(D^n,\nu,\underline{\lambda})$.}
    \label[figure]{colored_skein_module}
\end{figure}

\begin{remark}\label{remarklattice}
    The module $S_{\zeta_r}(\Dc)$ is well defined as we have a map $A_{r-1}\ra\Q(\zeta_r)$.
    This would not be the case if we used $A_k$, $k\geq r$.
    Notice also that if $r$ is prime, we have a map $A_{r-1}\ra\Z[\zeta_r]$.
    We can thus define in this case a lattice $S^\mathcal{O}_{\zeta_r}(\Dc)\subset S_{\zeta_r}(\Dc)$.
\end{remark}

\begin{proposition}[Mapping class group action]\label{braidgroupaction}
    Let $\Dc=(D^n,\nu,\underline{\lambda})$ be a colored disk. Denote by $B_1,\dotsc, B_n$ the inner boundary components and
    $\lambda_1,\dotsc,\lambda_n$ their respective colors.
    
    Then there are natural group maps:
    \begin{equation*}
        UPB_n\lra UPB_{\sum_{1\leq i\leq n}\lambda_i}\lra \left(TL_{\sum_{1\leq i\leq n}\lambda_i}\right)^\times
    \end{equation*}
    where the first map splits the $i$-th strand of a framed braid into $\lambda_i$ strands,
    and the second map associates to each framed braid the tangle it represents.

    As $TL_{\sum_{1\leq i\leq n}\lambda_i}$ acts on $S_{\zeta_r}(\Dc)$,
    we have an action of $ UPB_n\simeq \Mod{D^n}$ on $S_{\zeta_r}(\Dc)$.
\end{proposition}

\begin{proposition}[Bilinear form]\label{hermitianform}
    Let $\Dc=(D^n,\nu,\underline{\lambda})$ be a colored disk.
    Then symmetry about the top face of the cube and stacking
    induce a bilinear form:
    \begin{equation*}
        S_{\zeta_r}(\Dc)\otimes \overline{S_{\zeta_r}(\Dc)}\lra f_{\nu}TL_{\nu}.
    \end{equation*}
    As $f_{\nu}TL_{\nu}\otimes A_{r-1}\simeq A_{r-1}$, this is a sesquilinear form:
    \begin{equation*}
        h_D:\;S_{\zeta_r}(\Dc)\otimes\overline{S_{\zeta_r}(\Dc)}\lra \Q(\zeta_r).
    \end{equation*}
    Then $h_D$ is $\Mod{D^n}$-invariant.
\end{proposition}

\begin{definition}[Modular functor]\label{modularfunctorskeindefinition}
    Let $\Dc=(D^n,\nu,\underline{\lambda})$ be a colored disk. We define $\Nu_{\zeta_r}(\Dc)$ to be the quotient of $S_{\zeta_r}(\Dc)$
    by the kernel of $h_D$:
    \begin{equation*}
        \Nu_{\zeta_r}(\Dc) = S_{\zeta_r}(\Dc)/\ker h_D.
    \end{equation*}
    The action of $\Mod{D^n}$ on $S_{\zeta_r}(\Dc)$ induces a representation:
    $$\rho_D: \Mod{D^n} \lra \GL{\Nu_{\zeta_r}(\Dc)}.$$
\end{definition}

\begin{remark}\label{remarkintegral}
    If $r$ is prime, from \Cref{remarklattice}, we get a $\Mod{D^n}$-invariant lattice $\Nu^\mathcal{O}_{\zeta_r}(\Dc)\subset \Nu_{\zeta_r}(\Dc)$.
    From explicit computations, it can be shown that the Hermitian form $h_D$ is perfect on $\Nu^\mathcal{O}_{\zeta_r}(\Dc)$.
    These calculations of the norms on a basis are done, for example, in \cite[3.11,3.12]{marcheIntroductionQuantumRepresentations2021}.
\end{remark}

\begin{proposition}\label{dimensionproperties}
    The modules $\Nu_{\zeta_r}(\Dc)$ satisfy \Cref{colorconditions} and \Cref{existenceSOmodularfunctors}.
\end{proposition}

For a proof, see \cite[2.5 and 3.7]{marcheIntroductionQuantumRepresentations2021}.

We now describe the gluing property of $\Nu_{\zeta_r}$.

Let $\Dc=(D^n,\nu,\underline{\lambda})$ be a colored disk together with a simple closed curve $\gamma$ in $\mathrm{int}(D^n)$.
Cut $D^n$ along $\gamma$ to obtain two disks $D'$ and $D''$. For every color $\mu$, we denote by $\Dc_{\mu}'$, respectively $\Dc_{\mu}''$
the disk $D'$ (respectively $D''$) colored by $\nu,\underline{\lambda}$ at the boundary components it has in common with $D^n$
and by $\mu$ at the boundary component corresponding to $\gamma$.

Let $(B,P)$ (respectively $(B',P'_{\mu})$, $(B'',P''_{\mu})$) be the marked balls associated to $\Dc$
(respectively $\Dc_{\mu}'$, $\Dc_{\mu}''$) as in \Cref{definitionSD}. Then stacking at the boundaries corresponding to $\mu$
using $\gamma$ as in \Cref{gluing_skein_modules} induces a map:
\begin{equation*}
    S_{\zeta_r}(\Dc_{\mu}')\otimes S_{\zeta_r}(\Dc_{\mu}'') \lra S_{\zeta_r}(\Dc).
\end{equation*}
This map is compatible with the sesquilinear forms and the action of $\Mod{D'}\times\Mod{D''}$.
Thus, it induces a $\Mod{D'}\times\Mod{D''}$-equivariant map:
\begin{equation*}
    \Nu_{\zeta_r}(\Dc_{\mu}')\otimes\Nu_{\zeta_r}(\Dc_{\mu}'') \lra \Nu_{\zeta_r}(\Dc).
\end{equation*}

\begin{figure}
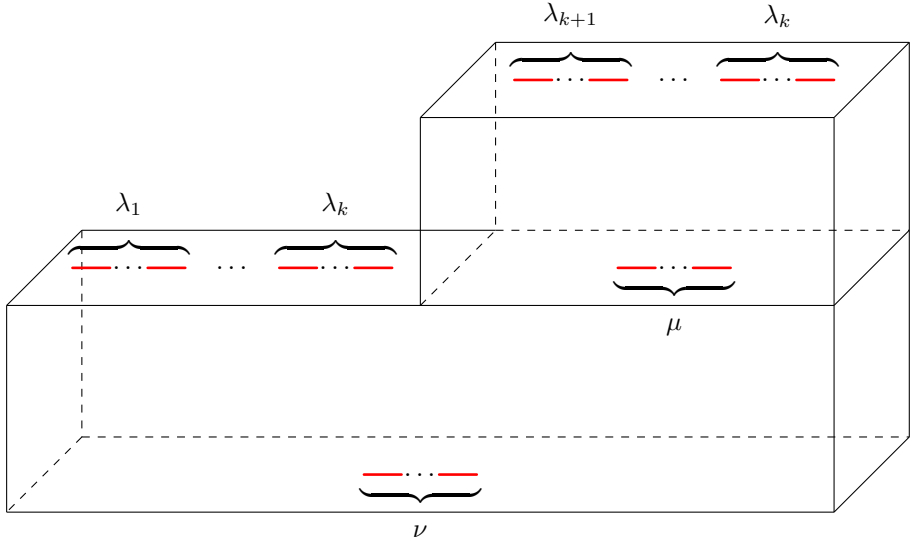

    \ctikzfig{gluing_skein_modules}
    \caption{The stacking of the two cubes $(B',P'_{\mu})$ and $(B'',P''_{\mu})$) along $\gamma$ is homeomorphic to $(B,P)$.}
    \label[figure]{gluing_skein_modules}
\end{figure}

We have the following result.

\begin{proposition}\label{skeingluingaxiom}
    The map:
    \begin{equation*}
        \bigoplus_{0\leq \mu\leq r-2}\Nu_{\zeta_r}(\Dc_{\mu}')\otimes\Nu_{\zeta_r}(\Dc_{\mu}'')\lra \Nu_{\zeta_r}(\Dc).
    \end{equation*}
    is a $\Mod{D'}\times\Mod{D''}$-equivariant isomorphism compatible with Hermitian forms.
    Moreover, if $r$ is prime, then this map induces an isomorphism at the level of lattices:
    \begin{equation*}
        \bigoplus_{0\leq \mu\leq r-2}\Nu_{\zeta_r}^\mathcal{O}(\Dc_{\mu}')\otimes\Nu^\mathcal{O}_{\zeta_r}(\Dc_{\mu}'')\simeq \Nu^\mathcal{O}_{\zeta_r}(\Dc).
    \end{equation*}
\end{proposition}

This result is the genus $0$ case of \cite[1.14]{blanchetTopologicalQuantumField1995}.
The statement about integrality results from the fact that the explicit bases in \cite[4.11]{blanchetTopologicalQuantumField1995}
are in the lattice.
We have thus outlined the construction of a the modular functor $\Nusl$.


\subsection{The special case of the Temperley-Lieb representations}

The functor $\Nu_{\zeta_r}$ contains as special cases some finite irreducible representations of the Temperley-Lieb
algebra at $q=\zeta_r$.

\begin{proposition}\label{TLrepresentations}
    Let $n\geq 0$ and $\nu\leq r-2$. Let $\Dc=(D^n,\nu,(1,\dotsc,1))$. Then the Temperley-Lieb algebra $TL_n$
    acts on $S_{\zeta_r}(\Dc)$ by stacking of cubes. This action preserves $h_D$, and $TL_n$ acts on $\Nu_{\zeta_r}(\Dc)$.

    Then $\Nu_{\zeta_r}(\Dc)$ is the irreducible representation of $TL_n\otimes \Q(\zeta_r)$
    corresponding to the Young diagram $(\frac{n+\nu}{2},\frac{n-\nu}{2})$.
\end{proposition}

See \cite{westburyRepresentationTheoryTemperleyLieb1995} for representations of Temperley-Lieb algebras.
These representations will be useful in \Cref{subsubBigelow,subsubproofintersection} to identify the Hermitian form in the geometric models.

\begin{proposition}
    Let $n$, $\mu$ and $\Dc$ be as in \Cref{TLrepresentations}. Then a basis of $S_{\zeta_r}(\Dc)$ is given by flat tangles, ie. tangles in
    $[0,1]^2\times \bigl\{\frac{1}{2}\bigr\}\subset [0,1]^3$. These tangles are in bijection with length $n$ balanced sequences
    of $"\bm{\cdot}"$, $"\bm{(}"$ and $"\bm{)}"$ with $\nu$ characters $"\bm{\cdot}"$ such that none is inside a matching pair of parentheses.
    For such a sequence $w$, we will denote $e_w$ the corresponding basis element.

    In particular, flat tangles give a generating set of $\Nu_{\zeta_r}(\Dc)$.
\end{proposition}

See \cite[2.3]{marcheIntroductionQuantumRepresentations2021} for a proof of the first fact.
The second one is illustrated in \Cref{flat_basis}.

\begin{figure}
    \ctikzfig{flat_basis}
    \caption{A tangle in the flat basis of $S_{\zeta_r}(D^8,2,(1,\dotsc,1))$ and the associated expression in
    $"\bm{\cdot}"$, $"\bm{(}"$ and $"\bm{)}"$.}
    \label[figure]{flat_basis}
\end{figure}


\subsection{The Quantum group construction}\label{subquantumconstruction}

In this section, we describe the modular functors $\Nu_{\zeta_r}$ with representations of the quantum group associated to $sl_2$.
The main reference for this section is \cite[9.3 and 11]{chariGuideQuantumGroups1995}.

\begin{definition}{\cite[9.1.1]{chariGuideQuantumGroups1995}}
    The quantum group $U_{\Q(q)}=U_{\Q(q)}(sl_2)$ is the Hopf algebra over $\Q(q)$ generated by $K$, $K^{-1}$, $E$ and $F$ subject to the relations:
    \begin{equation*}
        KK^{-1}=K^{-1}K=1
    \end{equation*}
    \begin{equation*}
        KEK^{-1} = q^2E,\; KFK^{-1} = q^{-2}F
    \end{equation*}
    \begin{equation*}
        [E,F] = \frac{K-K^{-1}}{q-q^{-1}}.
    \end{equation*}
    The counit $\epsilon$, antipode $S$ and comultiplication $\Delta$ are defined by:
    \[\begin{array}{c}
        \epsilon(K)=1,\;\epsilon(E)=\epsilon(F)=0,\;\Delta(K^{\pm 1})= K^{\pm 1}\otimes K^{\pm 1} \\
        \Delta(E) = E\otimes K + 1\otimes E,\; \Delta(F) = F\otimes 1 + K^{-1}\otimes F \\
        S(K) = K^{-1},\; S(E) = -EK^{-1},\; S(F) = -KF.
    \end{array}\]
\end{definition}

We now define the restricted integral form of $U_q$.

\begin{definition}{\cite[9.3.1]{chariGuideQuantumGroups1995}}
    For $l\geq 0$, define the divided powers of $E$ and $F$ as:
    \begin{equation*}
        E^{(l)} =\frac{1}{[l]!}E^l,\; F^{(l)} =\frac{1}{[l]!}F^l.
    \end{equation*}
    Let $U_{\Z[q^{\pm1}]}=U_{\Z[q^{\pm1}]}(sl_2)$ be the sub Hopf algebra of $U_{\Q(q)}$ over $\Z[q,q^{-1}]$
    generated by $K$, $K^{-1}$ and the divided powers $E^{(l)}$ and $F^{(l)}$ for $l\geq 0$.
\end{definition}

\begin{proposition}{\cite[9.3.A]{chariGuideQuantumGroups1995}}
    For $l\in \Z$, the following elements of $U_{\Q(q)}$ are in $U_{\Z[q^{\pm1}]}$:
    \begin{equation*}
        \begin{bmatrix} K\\l \end{bmatrix}  = \prod_{s=1}^l\frac{Kq^{1-s}-K^{-1}q^{s-1}}{q^s-q^{-s}}.
    \end{equation*}
    Moreover $U_{\Z[q^{\pm1}]}$ is generated as an algebra by $K$, $K^{-1}$, $E^{(l)}$, $F^{(l)}$
    and $\begin{bmatrix} K\\l \end{bmatrix}$ for $l\geq 0$.
\end{proposition}

For the relations amongst these generators of $U_{\Z[q^{\pm1}]}$, see \cite[9.3.4]{chariGuideQuantumGroups1995}.

We are now able to give a definition of Lusztig's quantum group for $sl_2$.

\begin{definition}[Restricted quantum group]
    Let $r\geq 1$ be odd. Then define the Lusztig restricted quantum group at $\zeta_r$ as:
    $$U_{\zeta_r} = (U_{\Z[q^{\pm1}]}\otimes_{\Z[q^{\pm1}]}\Z[\zeta_r])/(K^r-1)$$
    where $\Z[\zeta_r]$ is the ring of integers of the cyclotomic field of $r$-th roots.
    Here we set $q$ to $\zeta_r^2$.
\end{definition}

\begin{remark}
    In $U_{\zeta_r}$, for $0\leq a\leq r-1$, one has $E^{(a+rb)}=\frac{1}{[a]!}E^a\frac{(E^{(r)})^b}{b!}$
    and $F^{(a+rb)}=\frac{1}{[a]!}F^a\frac{(F^{(r)})^b}{b!}$. Similar formulas show that for $l\geq 0$, $\begin{bmatrix} K\\l \end{bmatrix}$
    is in the algebra generated by $K$, $K^{-1}$ and $\begin{bmatrix} K\\r \end{bmatrix}$.
    
    Hence $U_{\zeta_r}$ is generated by $K$, $K^{-1}$, $E$, $F$, $E^{(r)}$, $F^{(r)}$
    and $\begin{bmatrix} K\\r \end{bmatrix}$.
\end{remark}

Let us now describe a bit of representation theory of $U_{\zeta_r}$. 

\begin{proposition}{\cite[11.2.3]{chariGuideQuantumGroups1995}}\label{weightspaces}
    Let $W$ be a torsion-free finite-dimensional representation of $U_{\zeta_r}$. For $\lambda\in \Z$, we define the weight space:
    \begin{equation*}
        W_\lambda  = \{w\in W\;\mid\; Kw=q^\lambda w\text { and }\begin{bmatrix} K\\r \end{bmatrix}w=\begin{bmatrix} \lambda\\r \end{bmatrix}w \}.
    \end{equation*}
    Then $W$ is the direct sum of its weight spaces:
    \begin{equation*}
        W = \bigoplus_\lambda W_\lambda.
    \end{equation*}
\end{proposition}

From the defining relations, one notices that $EW_\lambda\subset W_{\lambda+2}$, $E^{(r)}W_\lambda\subset W_{\lambda+2r}$,
$FW_\lambda\subset W_{\lambda-2}$ and $F^{(r)}W_\lambda\subset W_{\lambda-2r}$.

\begin{definition}[$\mathcal{R}$-matrix]\label{definitionRmatrix}
    Let $W$ be a torsion-free finite dimensional representation of $U_{\zeta_r}$. Define the operator $H$ on $W$ by:
    \begin{equation*}
        Hw=\lambda w\text{ for }w\in W_\lambda.
    \end{equation*}
    The operator $v$ on $W$ is defined by:
    \begin{equation*}
        v = (-1)^{\mathrm{weight}}q^{-H^2/2}\sum_{n\geq 0}q^{n(3n+1)/2}(q-q^{-1})^nF^{(n)}K^{-n-1}E^n.
    \end{equation*}
    Now let $W_1$, $W_2$ be two torsion-free finite dimensional representations of $U_{\zeta_r}$. Define $\mathcal{R}$ to be the operator:
    \begin{equation*}
        \mathcal{R} = q^{H\otimes H/2}\sum_{n\geq 0}q^{\frac{n(n-1)}{2}}(q-q^{-1})^nE^n\otimes F^{(n)}.
    \end{equation*}
\end{definition}

We will see below that in the $UPB_n$-action $v$ corresponds to a twist while $\mathcal{R}$ corresponds to an elementary braid.
Note that $v$ and $\mathcal{R}$ commute to the action of $U_{\zeta_r}$ \cite[(4.23) and 4.14]{ohtsuki2001quantum}.
Let us now define the representations of interest.

\begin{definition}\label{definitionfinitedimensionalmodules}
    For $\alpha\geq 0$, let $V_\alpha=\Z[\zeta_r]e_0\oplus\dotsb\oplus\Z[\zeta_r]e_\alpha$ be the representation of $U_{\zeta_r}$
    with the action given by:
    \begin{equation*}
        Ee_i=[r-i+1]e_{i-1},\; Fe_i=[i+1]e_{i+1},\; Ke_i=q^{\alpha-2i}e_i,\;
        \begin{bmatrix} K\\r \end{bmatrix}e_i=\begin{bmatrix} \alpha-2i\\r \end{bmatrix}e_i
    \end{equation*}
    \begin{equation*}
        E^{(r)}=((r-i)_1+1)e_{i-r},\; F^{(r)}e_i=(i_1+1)e_{i+r}.
    \end{equation*}
    Here for $a\in\Z$, $a_1$ is such that $a=a_0+ra_1$ with $0\leq a_0\leq r-1$. Also, we use the convention $e_i=0$ for $i<0$ or $i>\alpha$.

    For $\alpha<0$, we set $V_\alpha=\{0\}$.
\end{definition}

\begin{proposition}{\cite[11.2.7]{chariGuideQuantumGroups1995}}
    For $0\leq \alpha\leq r-1$, $V_\alpha$ is irreducible.
\end{proposition}

Let us first recall the following fact about framed braid groups.

\begin{proposition}
    Let $n\geq 0$. There is an isomorphism:
    \begin{equation*}
        PB_n\times\Z^n\simeq UPB_n
    \end{equation*}
    where the map $PB_n\ra UPB_n$ is the blackboard framing of framed braids and the generator of the $i$-th factor of $\Z^n$
    is mapped to the full twist $tw_i$ of the $i$-th strand.
\end{proposition}

We will denote by $P$ the swap operator $V\otimes W\ra W\otimes V$, $v\otimes w\mapsto w\otimes v$.

\begin{proposition}\label{braidgroupactionquantumgroup}
    Let $0\leq \alpha_1,\dotsc,\alpha_n\leq r-2$. Then for $1\leq i\leq n-1$ and $\sigma\in \mathfrak{S}_n$, let $PR_i$ be
    the application of $P\circ\mathcal{R}$ on the $i$-th and $(i+1)$-th factors of
    $V_{\alpha_{\sigma(1)}}\otimes\dotsb\otimes V_{\alpha_{\sigma(n)}}$.
    For $1\leq i\leq n$ and $\sigma\in \mathfrak{S}_n$, let $v_i$ be the action of $v$ on the $i$-th factor of 
    $V_{\alpha_{\sigma(1)}}\otimes\dotsb\otimes V_{\alpha_{\sigma(n)}}$.
    Then the assignments $\sigma_i\mapsto PR_i$ and $tw_i\mapsto v_i$
    induces a representation of the framed pure braid group which commutes to the $U_{\zeta_r}$-action:
    \begin{equation*}
        \rho:UPB_n\lra \GLn{U_{\zeta_r}}{V_{\alpha_{1}}\otimes\dotsb\otimes V_{\alpha_{n}}}.
    \end{equation*}
\end{proposition}

The proof is essentially the fact that $\mathcal{R}$ satisfies the Yang-Baxter equation.
This is checked, for example, in \cite[4.2 and Appendix A.1]{ohtsuki2001quantum}.

\begin{definition}
    Let $W$ be a torsion-free finite dimensional representation of $U_{\zeta_r}$ and $f$ be an endomorphism of this representation.
    Its quantum trace $\qtr{f}$ is defined as $\qtr{f}=\tr{Kf}=\tr{fK}$.
    Now let $W_1$, $W_2$ be two torsion-free finite dimensional representations of $U_{\zeta_r}$.
    A $U_{\zeta_r}$-morphism $g:W_1\ra W_2$ is said to be negligible if for all $U_{\zeta_r}$-morphism $h:W_2\ra W_1$,
    $\qtr{f\circ g}=0$. We will denote $\bhom{W_1}{W_2}$ the set of $U_{\zeta_r}$-morphisms from $W_1$ to $W_2$ modulo
    the negligible ones.
\end{definition}

We will use the following more practical definition of $\bhom{V_\mu}{W}$.

\begin{proposition}\label{qtracedefinitionisevaluationdefinition}
    For $0\leq\mu\leq r-2$, $\qtr{\id_{V_\mu}}=[\mu+1]\neq 0$.
    So, for $W$ a torsion-free finite-dimension representations of $U_{\zeta_r}$, we have a bilinear form:
    \begin{equation*}
        ev:\myhom{V_\mu}{W}\otimes \myhom{W}{V_\mu}
        \lra \Z[\zeta_r]
    \end{equation*}
    defined by $ev(f\otimes g)\id_{V_\mu}=g\circ f$. Then $\bhom{V_\mu}{W}$
    is the quotient of $\myhom{V_\mu}{W}$ by the left-kernel of $ev$.
\end{proposition}

\begin{proof}
    One has $[\mu+1]ev(f\otimes g)=\qtr{g\circ f}$.
\end{proof}

\begin{remark}
    Another way to phrase \Cref{qtracedefinitionisevaluationdefinition}
    is that $\bhom{V_\mu}{W}$ is the image of the map:
    \[\begin{array}{rcl}
        \hom(V_\mu,W)&\lra    & \hom(W,V_\mu)^\vee \\
        f &\mapsto& (g\mapsto c\text{ such that } g\circ f=c\cdot \id).
    \end{array}\]
\end{remark}

We can now give the quantum group description of the modular functors $\Nu_{\zeta_r}$.

Let $\Dc=(D^n,\nu,\underline{\lambda})$ be a colored disk. We describe how operator invariants of tangles lead to a morphism:
\begin{equation}\label{skeintoquantummap}
    \mathrm{Op}_{\Dc}:S_r(\Dc)\lra \myhom{V_\mu}{V_{\lambda_1}\otimes\dotsb\otimes V_{\lambda_n}}\otimes \Q.
\end{equation}

The Kauffman bracket operator invariant of \cite[chp. 3]{ohtsuki2001quantum} can be lifted from $\C$ to $\Q[\zeta_r]$
by specifying $A=q=\zeta_r$ in the formulas for $R$ and $n$. The matrix $R$ given by Ohtsuki is then equal to the action of
$P\circ \mathcal{R}$ on $V_1\otimes V_1$ from \Cref{braidgroupactionquantumgroup} (see \cite[(4.42)]{ohtsuki2001quantum}).

This operator invariant associates to any tangle $T$ in $S_r(\Dc)$ a map:
\begin{equation*}
    \Phi(T): \myhom{V_1^{\otimes\mu}\otimes \Q}{V_1^{\otimes{\lambda_1}}\otimes \Q\otimes\dotsb\otimes V_1^{\otimes{\lambda_n}}\otimes \Q}.
\end{equation*}

Now, for each $n\geq 0$, we have maps:
\[\begin{tikzcd}
	{V_n} && {V_1^{\otimes n}} && {V_n}
	\arrow["{i_n}", hook, from=1-1, to=1-3]
	\arrow["{p_n}", two heads, from=1-3, to=1-5]
\end{tikzcd}\]
where $i_n$ is characterized by $i_n(e_0)=e_0^{\otimes n}$ and $p_n$ by $e_0^\vee(p_n(\otimes_ke_{i_k}))=\delta_{0i_1}\dotsb\delta_{0i_n}$
(see \Cref{vermaproperty} below for how this defines $i_n$ and $p_n$).

Then the map \Cref{skeintoquantummap} is given by:
\begin{equation*}
    \mathrm{Op}_{\Dc}(T) = \otimes_kp_{\lambda_k}\circ \Phi(T)\circ i_\nu.
\end{equation*}

\begin{theorem}\label{theoremquantumgroupconstruction}
    Let $\Dc=(D^n,\nu,\underline{\lambda})$ be a colored disk.
    Then the map $\mathrm{Op}_{\Dc}$ induces an isomorphism:
    \begin{equation*}
        \Nu_{\zeta_r}(\Dc)\simeq \bhom{V_\mu}{V_{\lambda_1}\otimes\dotsb\otimes V_{\lambda_n}}\otimes\Q.
    \end{equation*}
    This is an isomorphism of $\Mod{D^n}\simeq UPB_n$-modules, where it acts on the right hand side
    by post-composition via $\rho$ of \Cref{braidgroupactionquantumgroup}.

    Moreover, if $r$ is prime, we have an isomorphism between the $\Z[\zeta_r]$-lattices:
    \begin{equation*}
        \Nu_{\zeta_r}^\mathcal{O}(\Dc)\simeq \bhom{V_\mu}{V_{\lambda_1}\otimes\dotsb\otimes V_{\lambda_n}}.
    \end{equation*}
\end{theorem}

\begin{proof}[Outline of the proof]
    Using \Cref{qtracedefinitionisevaluationdefinition} below, and the definition of $\Nu_{\zeta_r}(\Dc)$ one sees that $\mathrm{Op}_{\Dc}(T)$ induces
    an injection:
    \begin{equation*}
        \Nu_{\zeta_r}(\Dc)\hookrightarrow\bhom{V_\mu}{V_{\lambda_1}\otimes\dotsb\otimes V_{\lambda_n}}.
    \end{equation*}
    Now, using the gluing properties on the skein modules side and the quantum representations side, one sees
    that we need only prove that it is a surjection for $n=2$. From \Cref{dimensionproperties}
    and \cite[11.3.16]{chariGuideQuantumGroups1995}, one sees that for $n=2$, both sides have the same dimension,
    which is always $0$ or $1$. This concludes the first statement.
    
    For the isomorphism at the level of lattices, one has that $h_D$ and $ev$ of \Cref{qtracedefinitionisevaluationdefinition}
    are compatible under the injections $\Nu_{\zeta_r}(\Dc)\hookrightarrow\bhom{V_\mu}{V_{\lambda_1}\otimes\dotsb\otimes V_{\lambda_n}}$ and 
    $\overline{\Nu_{\zeta_r}(\Dc)}\hookrightarrow\bhom{V_{\lambda_1}\otimes\dotsb\otimes V_{\lambda_n}}{V_\mu}$.
    Hence, to check that the generator of $e\in\Nu_{\zeta_r}^\mathcal{O}(\Dc)$ is mapped to a generator
    $\bhom{V_\mu}{V_{\lambda_1}\otimes\dotsb\otimes V_{\lambda_n}}$, one needs only check that $h_D(e,e)$ is invertible in $\Z[\zeta_r]$.

    The calculation of $(-1)^\nu[\nu+1]h_D(e,e)$ is done in \cite[3.12]{marcheIntroductionQuantumRepresentations2021}, where
    it is written as a quotient of products of quantum integers, and hence is invertible.
\end{proof}

\begin{remark}
    This shows, amongst other things, that when $r$ is prime, the map:
    \begin{equation*}
        \bhom{V_\mu}{V_{\lambda_1}\otimes\dotsb\otimes V_{\lambda_n}}\otimes \bhom{V_{\lambda_1}\otimes\dotsb\otimes V_{\lambda_n}}{V_\mu}
        \lra \Z[\zeta_r]
    \end{equation*}
    is a perfect pairing.
\end{remark}

Let us now discuss the gluing property as in \Cref{definitionrooting}. The definition of $\mathrm{Op}_{\Dc}$ in \Cref{skeintoquantummap}
via tangle operators yields the following Proposition.

\begin{proposition}
    Let $\Dc=(D^n,\nu,\underline{\lambda})$ be a colored disk together with a simple closed curve $\gamma$ in $\mathrm{int}(D^n)$.
    Denote by $\lambda_1,\dotsc,\lambda_n$ the inner colors of $\Dc$, numbered so that $\lambda_1,\dotsc,\lambda_k$ correspond
    to the colors of the boundary components inside of the region delimited by $\gamma$.
    Let $D'$ and $D''$ be the disks obtained by cutting $D^n$ along $\gamma$, $D'$ having $\gamma$ as outer boundary.
    For every color $\mu$, we denote by $\Dc_{\mu}'$, respectively $\Dc_{\mu}''$
    the disk $D'$ (respectively $D''$) colored by $\nu,\underline{\lambda}$ at the boundary components it has in common with $D^n$
    and by $\mu$ at the boundary component corresponding to $\gamma$.
    
    Then the gluing isomorphism:
    \begin{equation*}
        \bigoplus_{0\leq \mu\leq r-2}\Nu_{\zeta_r}(\Dc_{\mu}')\otimes\Nu_{\zeta_r}(\Dc_{\mu}'')\simeq \Nu_{\zeta_r}(\Dc)
    \end{equation*}
    is given by composition as follows:
    \begin{equation*}
        \bigoplus\limits_{0\leq \mu\leq r-2}\bhom{V_\nu}{V_\mu\otimes V_{\lambda_{k+1}}\otimes\dotsb\otimes V_{\lambda_n}}
        \otimes\bhom{V_\mu}{V_{\lambda_1}\otimes\dotsb\otimes V_{\lambda_k}}\otimes \Q
    \end{equation*}
    \begin{equation*}
        \simeq \bhom{V_\nu}{V_{\lambda_1}\otimes\dotsb\otimes V_{\lambda_n}}\otimes \Q
    \end{equation*}
    \begin{equation*}
        \sum_\mu f_\mu\otimes g_\mu \mapsto  (g_\mu\otimes \id_{V_{\lambda_{k+1}}}\otimes\dotsb\otimes \id_{V_{\lambda_n}})\circ f_\mu.
    \end{equation*}
\end{proposition}

The statement is a direct consequence of the definition of operator invariants of tangles, where gluing tangles corresponds to composition
of maps.

We introduce two sub-algebras of $U_{\zeta_r}$ that will appear in the geometric construction.

\begin{definition}
    For $l\geq 0$, let $\overline{F}^{(l)}=(q-q^{-1})^lF^{(l)}$.
    Now let $U_{\zeta_r}^{\overline{F}}$ be the sub-algebra of $U_{\zeta_r}$ generated by $K$, $K^{-1}$, $\begin{bmatrix} K\\r \end{bmatrix}$,
     $\overline{F}^{(1)}$ and $E^{(l)}$ for $l\geq 0$,
    and $U_{\zeta_r}^{E}$ be the sub-algebra of $U_{\zeta_r}$ generated by $K$, $K^{-1}$, $\begin{bmatrix} K\\r \end{bmatrix}$,
    $E$ and $\overline{F}^{(l)}$ for $l\geq 0$.

    Let $U_{\zeta_r}^{\mathrm{fin}}$ be the sub-algebra of $U_{\zeta_r}$ generated by $K$, $K^{-1}$, $E$ and $\overline{F}^{(1)}$.
\end{definition}

\begin{proposition}
    $U_{\zeta_r}^{E}$, $U_{\zeta_r}^{\overline{F}}$ and $U_{\zeta_r}^{\mathrm{fin}}$ are sub-Hopf-algebras of $U_{\zeta_r}$.
\end{proposition}

The proof consists in explicit computations of the comultiplication $\Delta$ on the generators.

\begin{remark}
    Notice that the operator $\mathcal{R}$ of \Cref{definitionRmatrix} satisfies:
    \begin{align*}
        \mathcal{R} &= q^{H\otimes H/2}\sum_{n\geq 0}q^{\frac{n(n-1)}{2}}E^n\otimes \overline{F}^{(n)}\\
                    &= q^{H\otimes H/2}\sum_{n\geq 0}q^{\frac{n(n-1)}{2}}E^{(n)}\otimes (\overline{F}^{(1)})^{n}.
    \end{align*}
    Hence it is well defined for $U_{\zeta_r}^{E}$ and $U_{\zeta_r}^{\overline{F}}$ modules as well. The same goes for the operator $v$.
\end{remark}

\begin{definition}
    Let $\alpha\in\Z$. Define $\hat{V}^{E}_\alpha=\Z[\zeta_r]e_0\oplus\dotsb\oplus \Z[\zeta_r]e_n \oplus\dotsb$ to be the
    $U_{\zeta_r}^{E}$-module with action given by:
    \begin{equation*}
        Ke_n=q^{\alpha-2n}e_n,\;\begin{bmatrix} K\\r \end{bmatrix}e_n=\begin{bmatrix} \alpha-2n\\r \end{bmatrix}e_n,
    \end{equation*}
    \begin{equation*}
        Ee_n=e_{n-1},\; \overline{F}^{(l)}e_n=\begin{bmatrix} n+l\\l \end{bmatrix}\prod_{k=0}^{l-1}(q^{\alpha-k-n}-q^{-\alpha+k+n})e_{n+l}.
    \end{equation*}
    Dually, define $\hat{V}^{\overline{F}}_\alpha=\Z[\zeta_r]e_0\oplus\dotsb\oplus \Z[\zeta_r]e_n \oplus\dotsb$ to be the
    $U_{\zeta_r}^{\overline{F}}$-module with action given by:
    \begin{equation*}
        Ke_n=q^{\alpha-2n}e_n,\;\begin{bmatrix} K\\r \end{bmatrix}e_n=\begin{bmatrix} \alpha-2n\\r \end{bmatrix}e_n,
    \end{equation*}
    \begin{equation*}
        \overline{F}^{(1)}e_n=e_{n+1},\; E^{(l)}e_n=\begin{bmatrix} n\\l \end{bmatrix}\prod_{k=1}^{l}(q^{\alpha+k-n}-q^{-\alpha-k+n})e_{n-l}.
    \end{equation*}
    The modules $\hat{V}^{\overline{F}}_\alpha$ and $\hat{V}^{E}_\alpha$ are called Verma modules.
\end{definition}

The proof is again some simple explicit computations.

In the geometric construction, the map $\mathrm{red}_\alpha:\hat{V}^{\overline{F}}_\alpha\ra \hat{V}^{E}_\alpha$
described below will appear.

\begin{proposition}\label{imageFtoEquantum}
    Let $\alpha\in\Z$, then the map:
    \begin{equation*}
        \mathrm{red}_\alpha:\hat{V}^{\overline{F}}_\alpha\lra \hat{V}^{E}_\alpha,\;e_i\mapsto (\overline{F}^{(1)})^ie_0=
        (q-q^{-1})^i[i]!\prod_{k=0}^{i-1}[\alpha-k]e_i
    \end{equation*}
    is a $U_{\zeta_r}^{\mathrm{fin}}$-map. If $\alpha=\alpha_0+r\alpha_1$ with $0\leq \alpha_0\leq r-1$,
    then the image of $\mathrm{red}_\alpha$ is isomorphic to $V_{\alpha_0}$
    as a $U_{\zeta_r}^{\mathrm{fin}}$-module. This isomorphism is realized via the following diagram:
    \[\begin{tikzcd}
        & {\hat{V}^{\overline{F}}_\alpha} && {\hat{V}^{E}_\alpha} \\
        {} && {V_{\alpha_0}}
        \arrow["{\pi_\alpha}"', two heads, from=1-2, to=2-3]
        \arrow["{\iota_\alpha}"', hook, from=2-3, to=1-4]
        \arrow["{\mathrm{red}_\alpha}", from=1-2, to=1-4]
    \end{tikzcd}\]
    where:
    \begin{equation*}
        \pi_\alpha:e_i\mapsto (q-q^{-1})^i[i]!e_i\text{ and }\iota_\alpha:e_i\mapsto \prod_{k=0}^{i-1}[\alpha-k]e_i.
    \end{equation*}
    If $0\leq\alpha\leq r-1$, $\pi_\alpha$ is a $U_{\zeta_r}^{\overline{F}}$-map and $\iota_\alpha$ is a $U_{\zeta_r}^{E}$-map.
\end{proposition}

\begin{proof}
    The quantum integer $[n]$ is $0$ if and only if $r\mid n$, when it is not $0$, it is invertible in $\Z[\zeta_r]$.
    Hence $(q-q^{-1})^i[i]!\prod_{k=0}^{i-1}[\alpha-k]$ is $0$ exactly for $i>\alpha_0$.
    And for $i\leq \alpha_0$, it is invertible. This identifies the image. The other claims are deduced by explicit computations
    of the actions.
\end{proof}

\begin{proposition}[Verma modules]\label{vermaproperty}
    Let $W$ be a torsion-free finite-dimensional $U_{\zeta_r}$-module and $0\leq\alpha\leq r-1$. Then we have $\Z[\zeta_r]$-linear bijections:
    \[\begin{array}{rcl}
        \hom_{U_{\zeta_r}}(V_\alpha,W) &\simeq  & W_\alpha\cap \ker E\cap \ker E^{(r)} \\
        f                              &\mapsto & f(e_0) \\
        \hom_{U_{\zeta_r}}(W,V_\alpha) &\simeq  & (W_\alpha/(\mathrm{im}\: F^{(1)}+\mathrm{im}\: F^{(r)}))^\vee\\
        g                              &\mapsto & (w\mapsto e_0^\vee(g(w)))
    \end{array}\]
    where $(e_i^\vee)_i$ is the dual basis of $(e_i)_i$ in $V_\alpha$.
    
    Let $W^{\overline{F}}$ be any $U_{\zeta_r}^{\overline{F}}$-module and $\alpha\in\Z$. Then we have a $\Z[\zeta_r]$-linear bijection:
    \[\begin{array}{rcl}
        \hom_{U_{\zeta_r}^{\overline{F}}}(\hat{V}^{\overline{F}}_\alpha,W^{\overline{F}})&\simeq 
        & W^{\overline{F}}_\alpha\cap \ker E\cap \ker E^{(r)}\\
        f                                                                 &\mapsto & f(e_0).
    \end{array}\]
    Let $W^{E}$ be any $U_{\zeta_r}^{E}$-module and $\alpha\in\Z$. Then we have a $\Z[\zeta_r]$-linear bijection:
    \[\begin{array}{rcl}
        \hom_{U_{\zeta_r}^{E}}(W^{E},\hat{V}^{E}_\alpha) &\simeq  &
        (W^{E}_\alpha/(\mathrm{im}\: \overline{F}^{(1)}+\mathrm{im}\: \overline{F}^{(r)}))^\vee \\
        g                                                &\mapsto & (w\mapsto e_0^\vee(g(w)))
    \end{array}\]
    where $(e_i^\vee)_i$ is the dual basis of $(e_i)_i$ in $\hat{V}^{E}_\alpha$.
\end{proposition}

\begin{proof}
    Let us begin with the statement about $W^{\overline{F}}$. Let $w\in W^{\overline{F}}_\alpha\cap \ker E\cap \ker E^{(r)}$.
    Define a map $\hat{V}^{\overline{F}}_\alpha\ra W^{\overline{F}}$ by sending $e_n=\overline{F}^{(n)}e_0$ to $\overline{F}^{(n)}w$.
    An explicit computation shows that it is compatible to the $U_{\zeta_r}^{\overline{F}}$-action. Hence we have a map from
    $\hat{V}^{\overline{F}}_\alpha\ra W^{\overline{F}}$ to $\hom_{U_{\zeta_r}^{\overline{F}}}(\hat{V}^{\overline{F}}_\alpha,W^{\overline{F}})$
    that is inverse to the one given above.

    The statement about $W^{E}$ is dual. For $\phi$ in $(W^{E}_\alpha/(\mathrm{im}\: \overline{F}^{(1)}+\mathrm{im}\: \overline{F}^{(r)}))^\vee$,
    define a map $W^{E}\ra\hat{V}^{E}_\alpha$ by sending $w\in W^{E}_\beta$ to $\phi(E^{\alpha-\beta}w)e_{\alpha-\beta}$
    if $\beta\leq\alpha$ and $0$ otherwise. Again, this is a map of $U_{\zeta_r}^{E}$-module. So we have a linear map from
    $(W^{E}_\alpha/(\mathrm{im}\: \overline{F}^{(1)}+\mathrm{im}\: \overline{F}^{(r)}))^\vee$ to $\hom_{U_{\zeta_r}^{E}}(W^{E},\hat{V}^{E}_\alpha)$,
    inverse to the one of the Proposition.

    Let us now turn to the statements about $W$. We will prove the first one,
    the second one is proved similarly. We will need the following two formulas.
    \begin{equation}\label{rel1}
        [E,F^{(n+1)}]=F^{(n)}(q^{-n}K-q^nK^{-1})
    \end{equation}
    \begin{equation}\label{rel2}
        \text{If }Eu=0\text{ and }E^{(r)}u=0\text{ then } E^{(l)}F^{(l)}u=\begin{bmatrix}K\\ l\end{bmatrix}u.
    \end{equation}
    These are both special cases of the relations in \cite[9.3.4]{chariGuideQuantumGroups1995}.
 
    Let $w\in W_\alpha\cap \ker E\cap \ker E^{(r)}$, then, as before, we have a map
    $V_\alpha\ra W$ by sending $e_n=F^{(n)}e_0$ to $F^{(n)}w$ for $0\leq n\leq \alpha$.
    One can see that to show that this is a map of $U_{\zeta_r}$-modules, we need only prove that $F^{(\alpha+1)}w=0$.

    Set $v=F^{(\alpha+1)}w$. Let us first show that $Ev=0$, using \Cref{rel1}:
    \begin{equation*}
        Ev=EF^{(\alpha+1)}w=[E,F^{(\alpha+1)}]w=F^{(\alpha)}(q^{-\alpha}K-q^nK^{-\alpha})w=0.
    \end{equation*}
    Now, if $\alpha\leq r-2$, it is clear that $E^{(r)}v=0$. If $\alpha=r-1$, we use \Cref{rel2}:
    \begin{equation*}
        E^{(r)}v=E^{(r)}F^{(r)}w=\begin{bmatrix}K\\ r\end{bmatrix}w=\begin{bmatrix}r-1\\ r\end{bmatrix}w=0.
    \end{equation*}
    Now, let $l\geq 0$. We have, by \Cref{rel2}:
    \begin{equation*}
        E^{(rl)}F^{(rl)}v=\begin{bmatrix}K\\ rl\end{bmatrix}v=\begin{bmatrix}-\alpha-1\\ rl\end{bmatrix}v.
    \end{equation*}
    But, one has:
    \begin{equation*}
        \begin{bmatrix}-\alpha-1\\ rl\end{bmatrix}=(-1)^{rl}\begin{bmatrix}\alpha+rl\\ rl\end{bmatrix}=
        (-1)^{rl}\begin{bmatrix}\alpha\\ 0\end{bmatrix}\begin{pmatrix}l\\ l\end{pmatrix}=(-1)^{rl}.
    \end{equation*}
    Hence if $v\neq 0$, then $F^{(rl)}v\neq 0$ for all $l$. But $W$ is finite dimensional, so $v=0$.
\end{proof}

\begin{corollary}\label{corollaryverma}
    Let $W$ be a torsion-free finite-dimensional $U_{\zeta_r}$-module and $0\leq\alpha\leq r-1$,
    then the maps $\pi_\alpha:\hat{V}^{\overline{F}}_\alpha\ra V_\alpha$ and $\pi_\alpha:V_\alpha\ra \hat{V}^E_\alpha$
    induce isomorphisms:
    \[\begin{array}{lcl}
        \hom_{U_{\zeta_r}}(V_\alpha,W)&\simeq & \hom_{U_{\zeta_r}^{\overline{F}}}(\hat{V}^{\overline{F}}_\alpha,W) \\
        \hom_{U_{\zeta_r}}(W,V_\alpha)&\simeq & \hom_{U_{\zeta_r}^{E}}(W,\hat{V}^{E}_\alpha).
    \end{array}\]
\end{corollary}

\begin{proof}
    For the first statement, we see by \Cref{vermaproperty} that both sides equal $W_\alpha\cap \ker E\cap \ker E^{(r)}$.
    As for the second statement, the left hand side is $(W_\alpha/(\mathrm{im}\: F^{(1)}+\mathrm{im}\: F^{(r)}))^\vee$ while the
    right hand side is $(W_\alpha/(\mathrm{im}\: \overline{F}^{(1)}+\mathrm{im}\: \overline{F}^{(r)}))^\vee$.
    As $\overline{F}^{(1)}=(q-q^{-1})F^{(1)}$ and $\overline{F}^{(r)}=(q-q^{-1})^rF^{(r)}$,
    $W_\alpha/(\mathrm{im}\: F^{(1)}+\mathrm{im}\: F^{(r)})$ and $W_\alpha/(\mathrm{im}\: \overline{F}^{(1)}+\mathrm{im}\: \overline{F}^{(r)})$
    are the same modulo torsion. Hence they have the same dual.
\end{proof}

\begin{proposition}\label{propositionimagequantumgroups}
    Let $1\leq \alpha_1,\dotsc,\alpha_n\leq r-1$ and $0\leq\beta\leq r-2$. Then the two maps:
    \[\begin{array}{rcl}
        \hom_{U_{\zeta_r}}(V_\beta,\bigotimes_iV_{\alpha_i})&\lra & \hom_{U_{\zeta_r}}(\bigotimes_iV_{\alpha_i},V_\beta)^\vee \\
        f &\mapsto& (g\mapsto c\text{ such that } g\circ f=c\cdot \id) \\
        \hom_{U_{\zeta_r}^{\overline{F}}}(\hat{V}^{\overline{F}}_\beta,\bigotimes_iV^{\overline{F}}_{\alpha_i})
        &\lra & \hom_{U_{\zeta_r}^{\overline{F}}}(\bigotimes_iV^{E}_{\alpha_i},\hat{V}^{E}_\beta)^\vee \\
        f &\mapsto& (g\mapsto c\text{ such that } g\circ f=c\cdot \id)
    \end{array}\]
    have isomorphic images.
    More precisely, consider the commutative diagram:
    \[\begin{tikzcd}
        {\hom_{U_{\zeta_r}^{\overline{F}}}(\hat{V}^{\overline{F}}_\beta,\bigotimes_i\hat{V}^{\overline{F}}_{\alpha_i})} & {\hom_{U_{\zeta_r}^{E}}(\bigotimes_i\hat{V}^{E}_{\alpha_i},\hat{V}^{E}_\beta)^\vee} \\
        {\hom_{U_{\zeta_r}^{\overline{F}}}(\hat{V}^{\overline{F}}_\beta,\bigotimes_iV_{\alpha_i})} & {\hom_{U_{\zeta_r}^{E}}(\bigotimes_iV_{\alpha_i},\hat{V}^{E}_\beta)^\vee} \\
        {\hom_{U_{\zeta_r}}(V_\beta,\bigotimes_iV_{\alpha_i})} & {\hom_{U_{\zeta_r}}(\bigotimes_iV_{\alpha_i},V_\beta)^\vee}
        \arrow["{\mathrm{red}}", from=3-1, to=3-2]
        \arrow[tail reversed, from=3-1, to=2-1]
        \arrow[tail reversed, from=2-2, to=3-2]
        \arrow[from=2-1, to=2-2]
        \arrow[from=1-1, to=1-2]
        \arrow[from=1-1, to=2-1]
        \arrow[from=1-2, to=2-2]
    \end{tikzcd}\]
    where the two-sided arrows are isomorphisms by \Cref{corollaryverma}.
    We have induced maps:
    \[\begin{tikzcd}
        {\hom_{U_{\zeta_r}^{\overline{F}}}(\hat{V}^{\overline{F}}_\beta,\bigotimes_i\hat{V}^{\overline{F}}_{\alpha_i})} & {(q-q^{-1})^m\cdot \mathrm{im}(\mathrm{red})} & {\hom_{U_{\zeta_r}^{E}}(\bigotimes_i\hat{V}^{E}_{\alpha_i},\hat{V}^{E}_\beta)^\vee}
        \arrow["\pi", two heads, from=1-1, to=1-2]
        \arrow["\iota", hook, from=1-2, to=1-3]
    \end{tikzcd}\]
    where $m$ is $(\alpha_1+\dotsb+\alpha_n-\beta)/2$.
    Then $\pi$ is surjective and $\iota$ is injective.
\end{proposition}

Before proving the Proposition, we introduce a duality between $\hat{V}^{\overline{F}}_\beta$ and $\hat{V}^{E}_\beta$ 
described by the following Proposition.

\begin{proposition}\label{bilinearformquantum}
    Let $1\leq \alpha_1,\dotsc,\alpha_n\leq r-1$. Define the bilinear form:
    \begin{equation*}
        (\cdot,\cdot):\bigotimes_i\hat{V}^{\overline{F}}_{\alpha_i}\times \bigotimes_i\hat{V}^{E}_{\alpha_i}
        \lra \Z[\zeta_r],\;\left(\otimes_ie_{k_i},\otimes_ie_{l_i}\right)\mapsto q^{f(\underline{k})}\delta_{k_1l_1}\dotsb\delta_{k_nl_n}
    \end{equation*}
    with:
    \begin{equation*}
        f(\underline{k})=\sum_{i\neq j} \alpha_ik_i - 2\sum_{i<j}k_ik_j.
    \end{equation*}
    Then for this form, $K$ and $\begin{bmatrix} K\\r \end{bmatrix}$ are auto-adjoint, $\overline{F}^{(1)}$ is left-adjoint to $E$ and
    for $l\geq 0$, $E^{(l)}$ is left-adjoint to $\overline{F}^{(l)}$.
    
    This form induces a symmetric bilinear form:
    \begin{equation*}
        \langle\cdot,\cdot\rangle:\bigotimes_i V_{\alpha_i}\times \bigotimes_i V_{\alpha_i}
        \lra \Z[\zeta_r]
    \end{equation*}
    by $\langle \pi_{\underline{\alpha}}(a),b\rangle= (a,\iota_{\underline{\alpha}}(b))$, 
    where $\pi_{\underline{\alpha}}=\otimes_i\pi_{\alpha_i}$ and $\iota_{\underline{\alpha}}=\otimes_i\iota_{\alpha_i}$. Equivalently:
    \begin{equation*}
        \langle \otimes_ie_{k_i},\otimes_ie_{l_i}\rangle=f(\underline{k})\prod_i(q-q^{-1})^{-k_i}\begin{bmatrix}\alpha_i\\k_i\end{bmatrix}\delta_{k_il_i}.
    \end{equation*}
    Then again $K$ and $\begin{bmatrix} K\\r \end{bmatrix}$ are auto-adjoint and $E^{(l)}$ and $\overline{F}^{(l)}$ are adjoint
    $\langle\cdot,\cdot\rangle$.

    The maps $\pi_\alpha$ and $\iota_\alpha$ are adjoint for the forms $\langle\cdot,\cdot\rangle$ and $(\cdot,\cdot)$.
\end{proposition}

\begin{proof}
    The statement about $K$ and $\begin{bmatrix} K\\r \end{bmatrix}$ is immediate since the form respects the weight gradings.
    The other statements are calculations that are easily reduced to the case $n=2$. For $n=2$, we show that 
    $\overline{F}^{(1)}$ is left-adjoint to $E$. The other statement is similar.
    The actions are given by:
    \begin{equation*}
        \Delta(E) = E\otimes K + 1\otimes E,\; \Delta(\overline{F}^{(1)}) = \overline{F}^{(1)}\otimes 1 + K^{-1}\otimes \overline{F}^{(1)}.
    \end{equation*}
    We see that $\overline{F}^{(1)}\otimes 1$ is left-adjoint to $E\otimes K$. Indeed, we have:
    \begin{align*}
        ((\overline{F}^{(1)}\otimes 1)(e_{k_1}\otimes e_{k_2}),e_{k_1+1}\otimes e_{k_2})&=(e_{k_1+1}\otimes e_{k_2},e_{k_1+1}\otimes e_{k_2}) \\
        (e_{k_1}\otimes e_{k_2},(E\otimes K)(e_{k_1+1}\otimes e_{k_2}))&=q^{\alpha_2-2k_2}(e_{k_1}\otimes e_{k_2},e_{k_1}\otimes e_{k_2}) \\
        (e_{k_1+1}\otimes e_{k_2},e_{k_1+1}\otimes e_{k_2})&=q^{\alpha_2-2k_2}(e_{k_1}\otimes e_{k_2},e_{k_1}\otimes e_{k_2}).
    \end{align*}
    Similarly, $1\otimes E$ is left-adjoint to $K^{-1}\otimes \overline{F}^{(1)}$. Hence $\overline{F}^{(1)}$ is left-adjoint to $E$.
\end{proof}

\begin{proof}[Proof of \Cref{propositionimagequantumgroups}]
    We will use the following notations:
    \begin{align*}
        W^{\overline{F}}&=\bigotimes_i\hat{V}^{\overline{F}}_{\alpha_i}, & W^{E}&=\bigotimes_i\hat{V}^{E}_{\alpha_i}, \\
        W &=\bigotimes_iV_{\alpha_i},  & W' &=\bigoplus\limits_\lambda(q-q^{-1})^{m(\lambda)}W_\lambda
    \end{align*}
    with $m(\lambda)=\frac{\alpha_1+\dotsb+\alpha_n-\lambda}{2}$ and $W_\lambda$ are the vectors of weight $\lambda$.
    Notice that when $m(\lambda)$ is not an integer, $W_\lambda=0$. Also, $W'$ is a sub-$\UF$-module of $W$.

    The maps $\pi_{\underline{\alpha}}$ and $\iota_{\underline{\alpha}}$ are adjoint for $\langle\cdot,\cdot\rangle$ and $(\cdot,\cdot)$.
    Moreover, as $E^{(l)}$ is left-adjoint to $\overline{F}^{(l)}$ for $(\cdot,\cdot)$,
    $W^{\overline{F}}\cap \ker E\cap \ker E^{(r)}$ is dual to
    $(W^{E}/(\mathrm{im}\: \overline{F}^{(1)}+\mathrm{im}\: \overline{F}^{(r)}))/\mathrm{torsion}$.
    Similarly for $\langle\cdot,\cdot\rangle$,
    $W\cap \ker E\cap \ker E^{(r)}$ and
    $(W/(\mathrm{im}\: \overline{F}^{(1)}+\mathrm{im}\: \overline{F}^{(r)}))/\mathrm{torsion}$
    are dual.

    Hence $\pi_{\underline{\alpha}}$ and $\iota_{\underline{\alpha}}$ are dual and we need only prove that $\pi$ is surjective.
    Let us prove that the map
    \begin{equation*}
        \pi^E_{\underline{\alpha}}:W^{\overline{F}}\cap \ker E\cap \ker E^{(r)}\lra W'\cap \ker E\cap \ker E^{(r)}
    \end{equation*}
    is surjective. To see this, note that $\pi_{\underline{\alpha}}$ has the following section:
    \begin{align*}
        W'                                                           &\lra               W^{\overline{F}} \\
        \otimes_i\left((q-q^{-1})^{k_i}\cdot e_{k_i}\right)          &\longmapsto        \otimes_i\frac{1}{[k_i]!}e_{k_i}.
    \end{align*}
    If $r$ is not prime, this is only a section after taking the tensor product with $\Q$.
    This section does not commute with the action of the whole of $W^{\overline{F}}$, but it does commute to the action
    of $E^{(l)}$ for $l\geq 0$. Hence it induces a section of $\pi^E_{\underline{\alpha}}$. So $\pi^E_{\underline{\alpha}}$
    is surjective.
\end{proof}

\begin{corollary}\label{imageofred}
    Let $1\leq \alpha_1,\dotsc,\alpha_n\leq r-1$ and $0\leq\beta\leq r-2$. Let $\Dc=(D^n,\beta,\underline{\alpha})$ be the corresponding
    colored disk. Then:
    \begin{equation*}
        \Nu_{\zeta_r}(\Dc) \simeq \mathrm{im}\left(\hom_{U_{\zeta_r}^{\overline{F}}}(\hat{V}^{\overline{F}}_\beta,\bigotimes_i
        \hat{V}^{\overline{F}}_{\alpha_i})
        \lra  \hom_{U_{\zeta_r}^{E}}(\bigotimes_i\hat{V}^{E}_{\alpha_i},\hat{V}^{E}_\beta)^\vee\right)\otimes \Q.
    \end{equation*}
    Equivalently:
    \begin{equation*}
        \Nu_{\zeta_r}(\Dc) \simeq \mathrm{im}\left(W^{\overline{F}}_\alpha\cap \ker E\cap \ker E^{(r)}\lra
        (W^{E}_\alpha/(\mathrm{im}\: \overline{F}^{(1)}+\mathrm{im}\: \overline{F}^{(r)}))/\mathrm{torsion}\right)\otimes \Q.
    \end{equation*}

    Moreover, if $r$ is prime, then we have:
    \begin{align*}
        \Nu_{\zeta_r}^\mathcal{O}(\Dc) &\simeq (q-q^{-1})^m\cdot\mathrm{im}\left(\hom_{U_{\zeta_r}^{\overline{F}}}(\hat{V}^{\overline{F}}_\beta,\bigotimes_i
        \hat{V}^{\overline{F}}_{\alpha_i})
        \lra  \hom_{U_{\zeta_r}^{E}}(\bigotimes_i\hat{V}^{E}_{\alpha_i},\hat{V}^{E}_\beta)^\vee\right) \\
        & \simeq (q-q^{-1})^m\cdot\mathrm{im}\left(W^{\overline{F}}_\alpha\cap \ker E\cap \ker E^{(r)}\lra
        (W^{E}_\alpha/(\mathrm{im}\: \overline{F}^{(1)}+\mathrm{im}\: \overline{F}^{(r)}))/\mathrm{torsion}\right)
    \end{align*}
    where $m$ is $(\alpha_1+\dotsb+\alpha_n-\beta)/2$.
\end{corollary}


\section{\texorpdfstring{The geometric construction of the $\SO$ modular functors}
{The geometric constructions of the SO(3) modular categories}}\label{sectiongeometricproofs}


\subsection{On homologies and Gauss-Manin Connections}


\subsubsection{Borel-Moore homology}\label{subsubBorelMoore}

We briefly recap some of the few properties of Borel-Moore homology that we will use.
For a definition of Borel-Moore homology, see \cite[IX]{iversenCohomologySheaves1986}.

We will be only concerned with the homology of differentiable manifolds with corners (ie. locally of the form 
$(\R_+)^n\times\R^m$). Hence all spaces considered from now on will be of this type.

\begin{proposition}
    If $X\ra Y$ is a proper map, then we have an induced map $H_*^{BM}(X)\ra H_*^{BM}(Y)$
    between Borel-Moore homologies.
\end{proposition}

\begin{proposition}\label{BMpairs}
    Let $F\subset X$ be a closed subset and $U=X\setminus F$. Then there is a canonical isomorphism:
    \begin{equation*}
        H_*^{BM}(X,F)\simeq H^{BM}_*(U).
    \end{equation*}
\end{proposition}

\begin{proposition}\label{symplexBMhomology}
    Let $m\geq 0$. Consider the unordered configuration space:
    \begin{equation*}
        \Conf{m}{(0,1)}=\{0<x_1<\dotsb<x_m<1\}.
    \end{equation*}
    Then it is homeomorphic to an open $m$-simplex and:
    \begin{equation*}
        H_k^{BM}(\Conf{m}{(0,1)}) = \begin{cases}
            \Z &\text{ if }k=m \\
            0 &\text{ otherwise.}
        \end{cases}
    \end{equation*}
\end{proposition}


\subsubsection{Homology with coefficients in a local system}\label{fundamentalclasslocalsystem}

Let $X$ be a connected manifold with corners and $\Line$ a $\C$-local system of rank $1$ on $X$.
Let $f:Z\ra X$ be a map such that $f^*\Line$ is trivial. Assume $Z$ is connected. Then we have a map:
\begin{equation*}
    H_*(Z;\Z)\subset H_*(Z;\C)\simeq H_*(Z;f^*\Line)\lra H_*(X;\Line).
\end{equation*}
depending on the choice of an isomorphism $H_*(Z;\C)\simeq H_*(Z;f^*\Line)$.
Such an isomorphism is given, for example, by the choice of a point in $z$ and an identification $(f^*\Line)_z\simeq \C$.

Let $x\in X$ be a base point, together with an identification $\Line_x\simeq \C$. 
Then a point $z\in Z$ together with a path $\gamma$ from $x$ to $f(z)$ yields such an isomorphism $H_*(Z;\C)\simeq H_*(Z;f^*\Line)$.
Hence, for $c\in H_*(Z;\Z)$ we can define a class:
\begin{equation*}
    f_*(c,z,\gamma)\in H_*(X;\Line).
\end{equation*}
We will not mention $z$ in the sequel, as in all our use cases, we will have $\{z\}=f^{-1}(\gamma(1))$.

This works equally well with Borel-Moore homologies if $f$ is proper.


\subsubsection{Gauss-Manin connection for homology with coefficients in a local system}\label{subsublocalsystemGaussManin}

Let $E\lra B$ be a topological fibration and $\Line$ a local system on $E$. Let $b\in B$ be some base point.
We will denote $E_b$ the fiber over $b$.
Here, we describe the action of $\pi_1(B,b)$ on $H_*(E_b;\Line)$ corresponding to the Gauss-Manin connection.

Let $\gamma:[0,1]\ra B$ be a loop at $b$. Then as $\gamma^*E$ is trivial over $[0,1]$, a choice of trivialization
gives a homeomorphism $\phi:E_b\simeq E_b$ via the identifications $(\gamma^*E)_0=E_b$
and $(\gamma^*E)_1=E_b$. Now, the local system $\Line$ on $\gamma^*E$
is isomorphic to $p^*\Line_{\mid E_b}$ for $p:\gamma^*E\simeq (\gamma^*E)_0\times [0,1]\ra (\gamma^*E)_0=E_b$.
Hence we have an isomorphism $\psi:\phi_*(\Line_{\mid E_b})\simeq \Line_{\mid E_b}$.

Now the action of $\gamma$ on $H_*(E_b;\Line)$ is given by the composition:
\[\begin{tikzcd}
	{H_*(E_b;\Line)} & {H_*(E_b;\phi_*(\Line))} & {H_*(E_b;\Line)}
	\arrow["{H_*(\phi)}", from=1-1, to=1-2]
	\arrow["{H_*(E_b;\psi)}", from=1-2, to=1-3]
\end{tikzcd}\]


\subsection{Homologies of configuration spaces}

We will be concerned with configuration spaces over the punctured compact disk and punctured open disk.
For convenience, we will represent disks as squares.

\begin{definition}
    For $M$ a manifold, and $m\geq 0$, we define the unordered configuration space of $m$ points on $M$ by:
    \begin{equation*}
        \Conf{m}{M}=\{\{z_1,\dotsc,z_m\}\subset M\;\mid\; z_i\neq z_j\text{ for }i\neq j\}.
    \end{equation*}
    If $F$ is a subset of $M$, we define $\Conf{m}{M,F}$ to be the pair $(A,B\subset A)$ where $A=\Conf{m}{M}$
    and:
    \begin{equation*}
        B = \{\{z_1,\dotsc,z_m\}\in \Conf{m}{M}\;\mid\; \exists i,\; z_i\in F\}.
    \end{equation*}
\end{definition}

\begin{definition}
    As in \Cref{definitiondisks}, we will denote by $D_n$ the open disk with $n$ points removed.
    We shall use the following model:
    \begin{equation*}
        D_n=(0,1)^2\setminus \{\left(\frac{1}{n+1},\frac{1}{2}\right),\dotsc,\left(\frac{n}{n+1},\frac{1}{2}\right)\}
    \end{equation*}
    and the notation $p_k=\frac{k}{n+1}$.
    We shall use a similar model for $\oD_n$:
    \begin{equation*}
        \oD_n=[0,1]^2\setminus \{\left(\frac{1}{n+1},\frac{1}{2}\right),\dotsc,\left(\frac{n}{n+1},\frac{1}{2}\right)\}.
    \end{equation*}
    We will use the notations $\partial_-\oD_n$ for $[0,1]\times\{0\}\subset \overline{D}_n$
    and $\partial_+\oD_n$ for $\{0,1\}\times [0,1]\cup [0,1]\times \{1\}$, as in \Cref{disk}.
\end{definition}

\begin{figure}
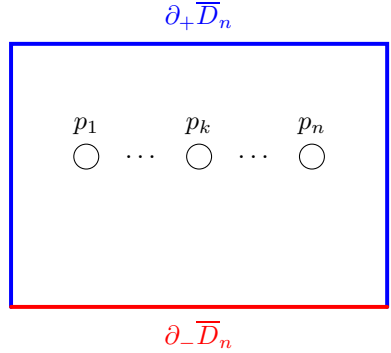

    \ctikzfig{disk}
    \caption{The disk $\oD_n$ and the two parts of its boundary.}
    \label[figure]{disk}
\end{figure}

Now, let $I=(0,1)\subset D_n$ be a closed submanifold. Then, for $m\geq 1$, we have a closed submanifold:
\begin{equation*}
    \Conf{m}{I}\subset \Conf{m}{D_n}.
\end{equation*}
Fix a local system $\Line$ of rank $1$ on $\Conf{m}{D_n}$, and a base point $x=\{x_1,\dotsc,x_m\}$ in $\Conf{m}{D_n}$.
Then for any path $\gamma:[0,1]\lra \Conf{m}{D_n}$ from $x$ to $\Conf{m}{I}$, the Borel-Moore fundamental class of 
$\Conf{m}{I}$ described in \Cref{symplexBMhomology} induces, by \Cref{fundamentalclasslocalsystem}, a class:
\begin{equation*}
    [I^m,\gamma]\in H_m^{BM}(\Conf{m}{D_n};\Line).
\end{equation*}
Note that a path $\gamma$ is simply $m$ paths $\gamma_1,\dotsc,\gamma_m$ where $\gamma_i$ goes from $x_i$ to $I$
and $\gamma_i(t)\neq\gamma_j(t)$ for all $t$ and $i\neq j$.

We can generalize the classes $[I^m,\gamma]$ slightly as follows. Let $I_1,\dotsc,I_k\subset D_n$ be $k$ disjoint
closed submanifolds, each diffeomorphic to $(0,1)$. Let $m\geq 1$ and $m=m_1+\dotsb+m_k$ be a partition.
Then we have a closed submanifold:
\begin{equation*}
    \Conf{m_1}{I_1}\times\dotsb\times\Conf{m_k}{I_k}\subset \Conf{m}{D_n}.
\end{equation*}
Now, as above, for any path $\gamma$ from $x$ to $\Conf{m_1}{I_1}\times\dotsb\times\Conf{m_k}{I_k}$, we have a class:
\begin{equation*}
    [I_1^{m_1},\dotsc,I_k^{m_k},\gamma]\in H_m^{BM}(\Conf{m}{D_n};\Line).
\end{equation*}

A variation on this is the following. If $(I_1,\partial I_1),\dotsc,(I_k,\partial I_k)\subset (\oD_n,\partial_-\oD_n)$ are $k$ disjoint
closed submanifolds with boundary, each diffeomorphic to $([0,1],\{0,1\})$ or $((0,1],\{1\})$, then we have a submanifold:
\begin{equation*}
    \Conf{m_1}{I_1,\partial I_1}\times\dotsb\times\Conf{m_k}{I_k,\partial I_k}\subset \Conf{m}{\oD_n,\partial_-\oD_n}.
\end{equation*}
Now, as above, for any path $\gamma$ from $x$ to $\Conf{m_1}{I_1,\partial I_1}\times\dotsb\times\Conf{m_k}{I_k,\partial I_k}$, we have a class:
\begin{equation*}
    [I_1^{m_1},\dotsc,I_k^{m_k},\gamma]\in H_m^{BM}(\Conf{m}{\oD_n,\partial_-\oD_n};\Line).
\end{equation*}

For a drawing of such a class $[I_1^{m_1},\dotsc,I_k^{m_k},\gamma]$, see \Cref{cycle_example}.

\begin{figure}
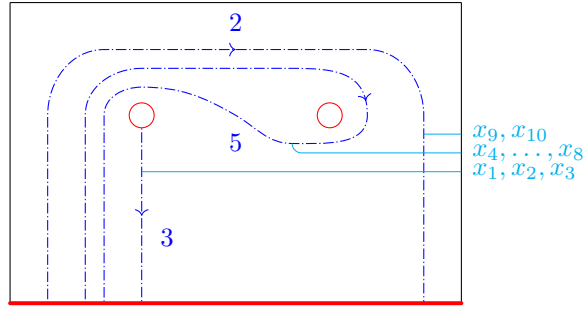

    \ctikzfig{cycle_example}
    \caption{A cycle $[I_1^3,I_2^5,I_3^2,\gamma]$ in  $H_{10}^{BM}(\Conf{10}{\oD_2,\partial_-\oD_2};\Line)$.
    The path $\gamma$ is represented by $3$ paths, each corresponding respectively to $3$, $5$ and $2$ parallel
    copies. Together these copies make a path in $\Conf{10}{\oD_2}$ from $x=\{x_1,\dotsc,x_{10}\}$
    to $\Conf{3}{I_1}\times\Conf{5}{I_2}\times\Conf{2}{I_2}$.}
    \label[figure]{cycle_example}
\end{figure}


\subsubsection{\texorpdfstring{An $U_{\zeta_r}^{\overline{F}}$-action on some homologies of configuration spaces}
{An U_\zeta_r^F-action on some homologies of configuration spaces}}\label{subsubUaction}

Most results in this section are from \cite{martelHomologicalModelUq2022}, but we use different normalizations out of necessity.
See Jules Martel's paper for details.

Fix an odd integer $r\geq 3$ and $n\geq 1$.
Let $a_1,\dotsc,a_n\in\{0,1,\dotsc,r-2\}$ be some colors. In the following, for $m\geq 0$, we will define $b\in\Z$ by
$m=\frac{a_1+\dotsb+a_n-b}{2}$.

For $m\geq 0$, set $\Line_m=\Line_\Dc^\mathcal{O}$ with $\Dc=(D^n,b,a_1,\dotsc,a_n)$ the $\Z[\zeta_r]$-local system on $\Conf{m}{\oD_n}$ as in \Cref{notationLine}.
We take this moment to explicit the local system $\Line_m$. Fix a base point $x=\{x_1,\dotsc,x_m\}$ in $\Conf{m}{D_n}$ as in \Cref{local_system}.
The fundamental group of $\Conf{m}{\oD_n}$ is the
braid group on $\oD_n$. It is generated by $\sigma_i$ for $1\leq i\leq m-1$ and $\tau_k$ for $1\leq k\leq n$,
where $\sigma_i$ swaps the $x_i$ and $x_{i+1}$ fixing all the other points, and $\tau_k$ takes $x_1$ around the puncture $p_k$,
fixing all the other points. See \Cref{local_system} for details.

\begin{figure}
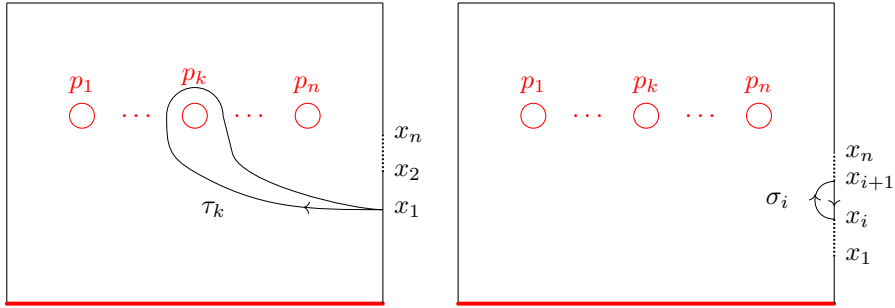

    \ctikzfig{local_system}
    \caption{The elements $\tau_k$ and $\sigma_i$ of $\pi_1(\Conf{m}{D_n},x)$.}
    \label[figure]{local_system}
\end{figure}

Now the representation $\rho_m$ associated to $\Line_m$ is:
\[\begin{array}{rcl}
    \pi_1(\Conf{m}{D_n},x) &\lra        & \Z[\zeta_r]^\times \\
    \sigma_i               &\longmapsto & -q^{-2} \\
    \tau_k                 &\longmapsto & q^{2a_k}.
\end{array}\]
Here $q^{\frac{1}{2}}=-\zeta_r$ as usual.

Let us begin by describing a space on which $\UE$ acts.

\begin{definition}{{\cite[6.1.3]{martelHomologicalModelUq2022}}}
    We define the $\WE{\ua}{m}$ by:
    \begin{equation*}
        \WE{\ua}{m} = H_m^{BM}(\Conf{m}{\oD_n,\partial_-\oD_n};\Line_m).
    \end{equation*}
    And set:
    \begin{equation*}
        \WEt{\ua} = \bigoplus_{m\geq 0} \WE{\ua}{m}.
    \end{equation*}
\end{definition}

We now recall the $\UE$-action on $\WEt{\ua}$ of \cite[6.1]{martelHomologicalModelUq2022}.
We make $K$ and $\Kr$ act by:
\begin{equation}\label{KactionE}
    Kc_m=q^bc_m\text{ and }\Kr c_m = \qbin{r}{b}c_m\text{ for any }c_m\in \WE{\ua}{m}.
\end{equation}

The action of $E$ is given by a differential in homology. Following \cite[2.3.2]{derenziHomologicalConstructionQuantum2023},
consider the following the subspaces of $\Conf{m}{\oD_n}$:
\begin{align*}
    Y_m &= \{\{z_1,\dotsc,z_m\}\;\mid\; \exists i,\:z_i\in \partial_-\oD_n\}, \\
    Z_m &= \{\{z_1,\dotsc,z_m\}\;\mid\; \exists i\neq j,\:z_i,z_j\in \partial_-\oD_n\}.
\end{align*}
Then the long exact sequence of the triplet $(\Conf{m}{\oD_n},Y_m,Z_m)$ in Borel-Moore homology gives a map:
\begin{equation*}
    \partial_m:H_m^{BM}(\Conf{m}{\oD_n},Y_m;\Line_m)\lra H_{m-1}^{BM}(Y_m,Z_m;\Line_{m\mid Y_m}).
\end{equation*}
Notice that by \Cref{BMpairs}, $H_{m-1}^{BM}(Y_m,Z_m)=H_{m-1}^{BM}(Y_m\setminus Z_m)$. But $Y_m\setminus Z_m$ is diffeomorphic to
$[0,1]\times\left(\Conf{m-1}{\oD_n}\setminus Y_{m-1}\right)$ and the restriction of local system $\Line_m$ is just $\Line_{m-1}$.

Hence, we have a map:
\begin{equation*}
    \partial_m:H_m^{BM}(\Conf{m}{\oD_n,\partial_-\oD_n};\Line_m)\lra H_{m-1}^{BM}(\Conf{m-1}{\oD_n,\partial_-\oD_n};\Line_{m-1}).
\end{equation*}

We make $E$ act by:
\begin{equation}\label{EactionE}
    Ec_m=q^{\frac{1}{2}\sum_ia_i}\partial_mc_m\text{ for any }c_m\in \WE{\ua}{m}.
\end{equation}

Consider the band $B=[0,1]^2$ and $\partial_-B=[0,1]\times \{0,1\}$. Then by gluing $(B,\partial_-B)$ to $(\oD_n,\partial_-\oD_n)$
along $\{1\}\times [0,1]\subset B$ and $\partial_+\oD_n\subset \oD_n$, we again get $\oD_n$ (see \Cref{gluing_B_D}).

\begin{figure}
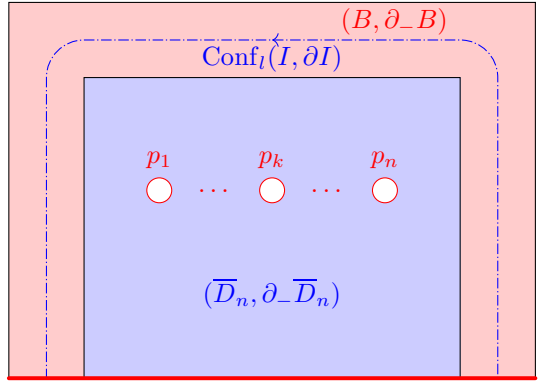

    \ctikzfig{gluing_B_D}
    \caption{The gluing of $(B,\partial_-B)$ to $(\oD_n,\partial_-\oD_n)$ and the 
    submanifold $\Conf{l}{I,\partial I}$ of $(B,\partial_-B)$.}
    \label[figure]{gluing_B_D}
\end{figure}

Thus for $m,l\geq 0$, we have a subspace:
\begin{equation*}
    \Conf{l}{B,\partial_-B}\times \Conf{m}{\oD_n,\partial_-\oD_n}\subset \Conf{m+l}{\oD_n,\partial_-\oD_n}
\end{equation*}
which induces a map on Borel-Moore homologies:
\begin{multline}\label{equationcupFaction}
    \cup:H_l^{BM}(\Conf{l}{B,\partial_-B};\Z)\otimes H_m^{BM}(\Conf{m}{\oD_n,\partial_-\oD_n};\Line_m) \\
    \lra H_{m+l}^{BM}(\Conf{m+l}{\oD_n,\partial_-\oD_n};\Line_{m+l}).
\end{multline}

Now, consider the class $[B]_l$ in $H_l^{BM}(\Conf{l}{B,\partial_-B};\Z)$ corresponding to the submanifold $\Conf{l}{I,\partial I}$
with $I=\{\frac{1}{2}\}\times [0,1]$.
We make $\Fb{l}$ act by:
\begin{equation}\label{FactionE}
    \Fb{l}c_m=q^{\frac{l(1-l)}{2}+\frac{1}{2}l\sum_ia_i}[B]_l\cup c_m\text{ for any }c_m\in \WE{\ua}{m}.
\end{equation}

\begin{theorem}[{\cite[Theorem 1]{martelHomologicalModelUq2022}}]
    The actions described in \Cref{KactionE,EactionE,FactionE} make $\WEt{\ua}$ a $\UE$-module.
\end{theorem}

We now turn to the Poincaré dual space on which $\UF$ acts.

\begin{definition}
    We define the $\WF{\ua}{m}$ by:
    \begin{equation*}
        \WF{\ua}{m} = H_m(\Conf{m}{\oD_n,\partial_-\oD_n};\Line_m).
    \end{equation*}
    And set:
    \begin{equation*}
        \WFt{\ua} = \bigoplus_{m\geq 0} \WF{\ua}{m}.
    \end{equation*}
\end{definition}

The space $H_m(\Conf{m}{\oD_n,\partial_-\oD_n};\Line_m)$ is Poincaré dual to $H^{BM}_m(\Conf{m}{\oD_n,\partial_-\oD_n};\Line_m)$
in the following sense. By Poincaré duality with respect to parts of the boundary, we have a perfect pairing:
\begin{equation}\label{equationintersectionform}
    H_m(\Conf{m}{\oD_n,\partial_-\oD_n};\Line_m)\times H^{BM}_m(\Conf{m}{\oD_n,\partial_+\oD_n};\overline{\Line_m})\lra \Z[\zeta_r]
\end{equation}
where $\overline{\Line_m}$ is the conjugate bundle.
Now, there is an orientation reversing diffeomorphism $\mathrm{flip}:(\oD_n,\partial_-\oD_n)\ra(\oD_n,\partial_+\oD_n)$,
given by the composition of the flip along $[0,1]\times\bigl\{\frac{1}{2}\bigr\}$ and a stretch near the boundary.
This diffeomorphism sends $\Line_m$ to $\overline{\Line_m}$. Hence we have a perfect bilinear pairing:
\begin{equation}\label{equationrelativedualitypairing}
    (\cdot,\cdot):H_m(\Conf{m}{\oD_n,\partial_-\oD_n};\Line_m)\times H^{BM}_m(\Conf{m}{\oD_n,\partial_-\oD_n};\Line_m)\lra \Z[\zeta_r].
\end{equation}

Let us now describe an action of $\UF$ on $\WFt{\ua}$. We make $K$ and $\Kr$ act by:
\begin{equation}\label{KactionF}
    Kc_m=q^bc_m\text{ and }\Kr c_m = \qbin{r}{b}c_m\text{ for any }c_m\in \WF{\ua}{m}.
\end{equation}

The action of $E$ is given by a restriction in homology. Consider the following the subspaces of $\Conf{m}{\oD_n}$:
\begin{align*}
    Y^i_m &= \{z=\{z_1,\dotsc,z_m\}\;\mid\;  |z\cap\partial_-\oD_n|\geq i\}, \\
    Y_m^+ &= \{z=\{z_1,\dotsc,z_m\}\;\mid\;  |z\cap\partial_+\oD_n|\geq 1\}.
\end{align*}
Note that $\Conf{m}{\oD_n,\partial_+\oD_n}=(\Conf{m}{\oD_n},Y_m^+)$.
Consider the restriction map: 
\begin{equation*}
    H_c^*(\Conf{m}{\oD_n},Y_m^+;\Line_m)\ra H_c^*(Y^i_m,Y^i_m\cap Y_m^+;\Line_{m\mid Y^i_m}).
\end{equation*}
As $\partial Y^i_m=Y^{i+1}_m\cup(Y^i_m\cap Y_m^+)$, its Poincaré dual in degree $m$ is:
\begin{equation*}
    \partial_m^i:H_m(\Conf{m}{\oD_n},Y^1_m;\Line_m)\lra H_{m-i}(Y^i_m,Y^{i+1}_m;\Line_{m\mid Y^i_m}).
\end{equation*}
Note that we used Poincaré duality for homologies relatively to a decomposition the boundary
(see, for example, \cite[VI.9.Problem 3]{bredonTopologyGeometry1993}).
Now, one sees that the pairs $(Y^i_m,Y^{i+1}_m)$ and $\Conf{m-i}{\oD_n,\partial_-\oD_n}$ have the same homologies.
Indeed, we need only check that the Poincaré dual pairs have same cohomology with compact support.
Now $(Y^i_m,Y^{i+1}_m)$ is Poincaré dual in $Y^i_m$ to $(Y^i_m,Y^i_m\cap Y_m^+)$, and
$H^*_c(Y^i_m,Y^i_m\cap Y_m^+)=H^*_c(Y^i_m\setminus(Y^i_m\cap Y_m^+))$. Moreover,
$Y^i_m\setminus(Y^i_m\cap Y_m^+)$ is homeomorphic to $\Conf{m-i}{\oD_n\setminus(\partial_+\oD_n)}\times\Conf{i}{(0,1)}$.
So that $H^*_c(Y^i_m,Y^i_m\cap Y_m^+)\simeq H^{*-i}_c(\Conf{m-i}{\oD_n,\partial_+\oD_n})$.
Now the pair $\Conf{m-i}{\oD_n,\partial_+\oD_n}$ is Poincaré dual in $\Conf{m-i}{\oD_n}$ to 
$\Conf{m-i}{\oD_n,\partial_-\oD_n}$. Hence $H_*(Y^i_m,Y^i_{m+1})\simeq H_{*}(\Conf{m-i}{\oD_n,\partial_-\oD_n})$.

We thus have a map:
\begin{equation*}
    \partial_m^i:H_m(\Conf{m}{\oD_n,\partial_-\oD_n};\Line_m)\lra H_{m-i}(\Conf{m-i}{\oD_n,\partial_-\oD_n};\Line_{m-i}).
\end{equation*}

We make $\Eb{l}$ act by:
\begin{equation}\label{EactionF}
    \Eb{l}c_m=q^{\frac{l(1-l)}{2}+\frac{1}{2}l\sum_ia_i}\partial_m^lc_m\text{ for any }c_m\in \WF{\ua}{m}.
\end{equation}

As above for the $\UE$-action, for $m\geq 0$, we have a subspace:
\begin{equation*}
    (B,\partial_-B)\times \Conf{m}{\oD_n,\partial_-\oD_n}\subset \Conf{m+1}{\oD_n,\partial_-\oD_n}
\end{equation*}
which induces a map on homologies:
\begin{multline*}
    \cup:H_1(B,\partial_-B;\Z)\otimes H_m(\Conf{m}{\oD_n,\partial_-\oD_n};\Line_m) \\
    \lra H_{m+1}(\Conf{m+1}{\oD_n,\partial_-\oD_n};\Line_{m+1}).
\end{multline*}

Now, consider the class $[B]_1$ in $H_1(\Conf{l}{B,\partial_-B};\Z)$ corresponding to the submanifold $(I,\partial I)$
with $I=\bigl\{\frac{1}{2}\bigr\}\times [0,1]$.
We make $\Fb{1}$ act by:
\begin{equation}\label{FactionF}
    \Fb{1}c_m=q^{\frac{1}{2}\sum_ia_i}[B]_1\cup c_m\text{ for any }c_m\in \WF{\ua}{m}.
\end{equation}

\begin{theorem}\label{theoremhomologicalUactionduality}
    Consider the bilinear form of \Cref{equationrelativedualitypairing}:
    \begin{equation*}
        (\cdot,\cdot): \WFt{\ua}\times \WEt{\ua}\lra \Z[\zeta_r].
    \end{equation*}
    Then for this form $K^{-1}$ is left-adjoint to $K$, $\Kr$ is auto-adjoint, $\Eb{l}$ is left-adjoint to $\Fb{l}$ and 
    $\Fb{1}$ is left-adjoint to $E$.
\end{theorem}

\begin{proof}
    The statement about $K$ and $\Kr$ is obvious. We outline a proof that $\Eb{l}$ is left-adjoint to $\Fb{l}$.
    The proof that $\Fb{1}$ is left-adjoint to $E$ is essentially the same.

    Choose $m\geq 0$. We need to consider the conjugate of $\Fb{l}$ by the map $\mathrm{flip}$.
    Applying $\mathrm{flip}$ to the gluing of $B$ and $\oD_n$ considered above, we have a submanifold:
    \begin{equation*}
        \Conf{l}{I}\times \Conf{m}{\oD_n}\subset\Conf{l}{B}\times \Conf{m}{\oD_n}\subset \Conf{m+l}{\oD_n}
    \end{equation*}
    as in \Cref{gluing_B_D_reverse}. One sees that the inclusion of $\Conf{l}{I}\times \Conf{m}{\oD_n}$ factors
    up to proper homotopy through $Y^l_{m+l}$.
    Hence we get the following commutative diagram:
    \[\begin{tikzcd}
        & {H_m^{BM}(\Conf{m}{\oD_n,\partial_+\oD_n};\Line_m)} \\
        {H_{m+l}^{BM}(Y_{m+l}^l,Y_{m+l}^l\cap Y_{m+l}^+;\Line_{m+l})} \\
        & {H_{m+l}^{BM}(\Conf{m+l}{\oD_n,\partial_+\oD_n};\Line_{m+l})}
        \arrow["{\cdot\cup\mathrm{flip}([B]_l)}", from=1-2, to=3-2]
        \arrow[from=1-2, to=2-1]
        \arrow[from=2-1, to=3-2]
    \end{tikzcd}\]
    
    \begin{figure}
        \ctikzfig{gluing_B_D_reverse}
        \caption{The submanifold $\Conf{l}{I}\times \Conf{m}{\oD_n}$ of $\Conf{m+l}{\oD_n}$.}
        \label[figure]{gluing_B_D_reverse}
    \end{figure}
    
    Now, the Poincaré dual diagram is:
    \[\begin{tikzcd}
        & {H_m(\Conf{m}{\oD_n,\partial_-\oD_n};\Line_m)} \\
        {H_{m+l}(Y_{m+l}^l,Y_{m+l}^{l+1};\Line_{m+l})} \\
        & {H_{m+l}(\Conf{m+l}{\oD_n,\partial_-\oD_n};\Line_{m+l})}
        \arrow["{\partial_m^l}"', from=3-2, to=1-2]
        \arrow[from=2-1, to=1-2]
        \arrow[from=3-2, to=2-1]
    \end{tikzcd}\]
    As the maps $\cdot\cup\mathrm{flip}([B]_l)$ and $\partial_m^l$ are Poincaré dual, they are adjoint for 
    the intersection forms of \Cref{equationintersectionform}.
    But up to the same renormalization factor, $\cdot\cup\mathrm{flip}([B]_l)$ is the conjugate by $\mathrm{flip}$ of $\Fb{l}$,
    and $\partial_m^l$ is $\Eb{l}$. Hence $\Eb{l}$ is left-dual to $\Fb{l}$ for $(\cdot,\cdot)$.
\end{proof}

\begin{corollary}
    The actions described in \Cref{KactionF,EactionF,FactionF} make $\WFt{\ua}$ a $\UF$-module.
\end{corollary}


\subsubsection{Bases for semi-relative homologies}\label{subsubbases}

Most results in this section are taken from \cite{martelHomologicalModelUq2022}, but we use different normalizations out of necessity.
We keep the notations of the previous subsubsection.
In this subsubsection, we denote by $I_k$ the segment $(p_k,(\frac{k}{n+1},0)]\subset\oD_n$ going vertically from $\partial_-\oD_n$ to $p_k$.

\begin{definition}[{\cite[3.9]{martelHomologicalModelUq2022}}]
    Let $m\geq 0$ and $m_1+\dotsb+m_n=m$ a partition. Denote by $A(\underline{m})=A(m_1,\dotsc,m_n)$ the homology class:
    \begin{equation*}
        q^{\sum_{i<j}a_im_j-\frac{1}{2}\sum_{i,j}a_im_j}[I_1^{m_1},\dotsc,I_n^{m_n},\gamma]\in H_m^{BM}(\Conf{m}{\oD_n,\partial_-\oD_n};\Line_m)
    \end{equation*}
    where $\gamma$ is as in \Cref{A_classes}.
\end{definition}

\begin{figure}
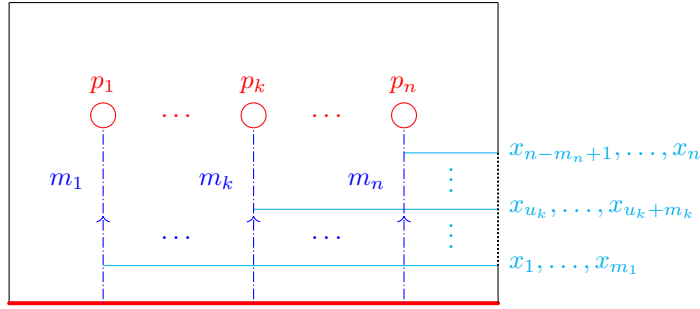

    \ctikzfig{A_classes}
    \caption{The cycle $[I_1^{m_1},\dotsc,I_n^{m_n},\gamma]$ defining $A(m_1,\dotsc,m_n)$. Here $u_k$ is $m_1+\dotsb+m_{k-1}$.}
    \label[figure]{A_classes}
\end{figure}

\begin{theorem}[{\cite[Theorems 2 and 3]{martelHomologicalModelUq2022}}]\label{theoremmartel}
    Consider the map $g_{\underline{a}}$ defined by:
    \begin{align*}
        \VE{a_1}\otimes\dotsb\otimes\VE{a_n}&\lra \WEt{\underline{a}}=\bigoplus_mH_m^{BM}(\Conf{m}{\oD_n,\partial_-\oD_n};\Line_m)\\
        e_{m_1}\otimes\dotsb\otimes e_{m_n} &\longmapsto A(m_1,\dotsc,m_n).
    \end{align*}
    Then we have that:
    \begin{description}
        \item[(1)] $g_{\ua}$ is an isomorphism. In particular for $m\geq 0$,
        the $A(m_1,\dotsc,m_n)$ for $m_1+\dotsb+m_n=m$ form a basis of
        $H_m^{BM}(\Conf{m}{\oD_n,\partial_-\oD_n};\Line_m)$;
        \item[(2)] $g_{\ua}$ is a $\UE$-map, where the action on $\WEt{\underline{a}}$ is that of
        \Cref{subsubUaction};
        \item[(3)] $g_{\ua}$ commutes to the action of $PB_n=\PMod{\oD_n}$, where the action on
        the source is given by the operator $\mathcal{R}$ as in \Cref{braidgroupactionquantumgroup}
        and the action on the target is the Gauss-Manin connection as in \Cref{subsublocalsystemGaussManin}.
    \end{description}
\end{theorem}

\begin{remark}
    Here, the $PB_n=\PMod{\oD_n}$-action is slightly different from that of \cite{martelHomologicalModelUq2022}.
    In Martel's paper, the action of the braid $\sigma_i$ on $\VE{a_1}\otimes\dotsb\otimes\VE{a_n}$ is corrected by
    $q^{-\frac{a_ia_{i+1}}{2}}$ and the action on the homological side is given by an identification
    $\psi':\phi_*\Line\simeq\Line$ different from $\psi$ of \Cref{subsublocalsystemGaussManin}.
    These two differences compensate each other.
\end{remark}

Let us now turn to the dual basis. For $m_i\geq 0$, denote by $J_{i,1},\dotsc,J_{i,m_i}$ parallel copies of the arc from $\partial_-\oD_n$
to itself making one turn around $p_i$, as in \Cref{B_classes}. With these we define $B$ classes dual to the $A$ classes.

\begin{definition}
    Let $m\geq 0$ and $m_1+\dotsb+m_n=m$ a partition. Denote by $B(\underline{m})=B(m_1,\dotsc,m_n)$ the homology class:
    \begin{multline*}
        q^{\sum_{i\leq j}a_im_j-\frac{1}{2}\sum_{i,j}a_im_j}[J_{1,1},\dotsc,J_{1,m_1},\dotsc,J_{n,1},\dotsc,J_{n,m_n},\gamma]\\\in H_m(\Conf{m}{\oD_n,\partial_-\oD_n};\Line_m)
    \end{multline*}
    where $\gamma$ is as in \Cref{B_classes}.
\end{definition}

\begin{figure}
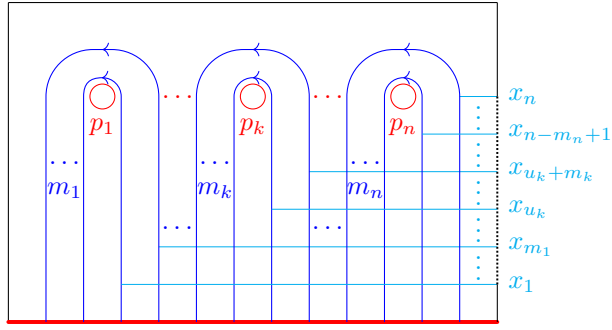

    \ctikzfig{B_classes}
    \caption{The cycle $[J_{1,1},\dotsc,J_{1,m_1},\dotsc,J_{n,1},\dotsc,J_{n,m_n},\gamma]$ defining $B(m_1,\dotsc,m_n)$.
    Here $u_k$ is $m_1+\dotsb+m_{k-1}$.}
    \label[figure]{B_classes}
\end{figure}

We get the following dual statements to the ones of \Cref{theoremmartel}.

\begin{theorem}
    Consider the map $h_{\ua}$ defined by:
    \begin{align*}
        \VF{a_1}\otimes\dotsb\otimes\VF{a_n}&\lra \WFt{\underline{a}}=\bigoplus_mH_m(\Conf{m}{\oD_n,\partial_-\oD_n};\Line_m)\\
        e_{m_1}\otimes\dotsb\otimes e_{m_n} &\longmapsto B(m_1,\dotsc,m_n).
    \end{align*}
    Then we have that:
    \begin{description}
        \item[(1)] $h_{\ua}$ is an isomorphism. In particular for $m\geq 0$,
        the $B(m_1,\dotsc,m_n)$ for $m_1+\dotsb+m_n=m$ form a basis of
        $H_m(\Conf{m}{\oD_n,\partial_-\oD_n};\Line_m)$;
        \item[(2)] $h_{\ua}$ is a $\UF$-map, where the action on $\WFt{\underline{a}}$ is that of
        \Cref{subsubUaction};
        \item[(3)] $h_{\ua}$ commutes to the action of $PB_n=\PMod{\oD_n}$, where the action on
        the source is given by the operator $\mathcal{R}$ as in \Cref{braidgroupactionquantumgroup}
        and the action on the target is the Gauss-Manin connection as in \Cref{subsublocalsystemGaussManin}.
    \end{description}
\end{theorem}

\begin{proof}
    Consider the perfect duality pairing $(\cdot,\cdot):\WFt{\ua}\times\WEt{\ua}\ra \Z[\zeta_r]$.
    Then an intersection computation shows that:
    \begin{equation*}
        (B(\underline{m}),A(\underline{m'}))=q^{f(\um)}\delta_{\um,\underline{m'}}
    \end{equation*}
    with $f$ as in \Cref{bilinearformquantum}. \Cref{pairing} shows the geometric intersection
    of $B(\underline{m})$ with $\mathrm{flip}(A(\underline{m'}))$. One sees that it is empty if $\underline{m}\neq \underline{m'}$.
    If $\underline{m}= \underline{m'}$, consider the path $\gamma_{\um}$ in \Cref{pairing}.
    The value of the intersection is then the product of $\rho_m(\gamma_{\um})$ with $q^{\sum_{i\leq j}a_im_j-\frac{1}{2}\sum_{i,j}a_im_j}$
    and $q^{\sum_{i<j}a_im_j-\frac{1}{2}\sum_{i,j}a_im_j}$. This is computed to be $q^{f(\um)}$.
    This proves point \textbf{(1)} and identifies the homological form
    with the form of \Cref{bilinearformquantum}.

    \begin{figure}
        \ctikzfig{pairing}
        \caption{The intersection of $B(\underline{m})$ with $\mathrm{flip}(A(\underline{m'}))$ and the path $\gamma_{\um}$.}
        \label[figure]{pairing}
    \end{figure}

    Now, as mentioned in \Cref{theoremhomologicalUactionduality}, the actions of $\UF$ on $\WFt{\ua}$
    and of $\UE$ on $\WEt{\ua}$ are adjoint in the same way that the actions 
    of $\UF$ on $\VF{a_1}\otimes\dotsb\otimes\VF{a_n}$ and of $\UE$ on $\VE{a_1}\otimes\dotsb\otimes\VE{a_n}$
    are (see \Cref{bilinearformquantum}). This, together with \Cref{theoremmartel} implies point \textbf{(2)}.

    The homological intersection form of \Cref{equationintersectionform} is invariant for the action of $\PMod{\oD_n}$.
    Now, notice that conjugation by the map $\mathrm{flip}$ of \Cref{subsubUaction} is:
    \begin{align*}
        \PMod{\oD_n} &\lra \PMod{\oD_n} \\
        \sigma_i     &\longmapsto \sigma_i^{-1}.
    \end{align*}
    Hence for each $i$, $\sigma_i$ is auto-adjoint for $(\cdot,\cdot):\WFt{\ua}\times\WEt{\ua}\ra \Z[\zeta_r]$.
    This is also the case for the intersection form of \Cref{bilinearformquantum}, as the left-adjoint $(P\circ\mathcal{R})^*$ to the operator
    $P\circ\mathcal{R}$ is:
    \begin{align*}
        &\left(P\circ q^{H\otimes H/2}\sum_{n\geq 0}q^{\frac{n(n-1)}{2}}E^n\otimes \overline{F}^{(n)}\right)^* \\
        &= \left(\sum_{n\geq 0}q^{\frac{n(n-1)}{2}}\left((E^n\otimes K^n)(K^n\otimes K^{-n})(K^{-n}\otimes\overline{F}^{(n)})\right)^*\right)q^{H\otimes H/2}\circ P \\
        &= \left(\sum_{n\geq 0}q^{\frac{n(n-1)}{2}}(1\otimes \Eb{n})(K^n\otimes K^{-n})((\overline{F}^{(1)})^n\otimes 1)\right)q^{H\otimes H/2}\circ P \\
        &= P\circ \left(\sum_{n\geq 0}q^{\frac{n(n-1)}{2}}(S\otimes S)(\Eb{n}\otimes (\overline{F}^{(1)})^n)\right)q^{H\otimes H/2} \\
        &= P\circ (S\otimes S)\mathcal{R}\\
        &=P\circ\mathcal{R}\text{ by \cite[(4.19)]{ohtsuki2001quantum}.}
    \end{align*}
    Here $S$ is the antipode. We used that $\overline{F}^{(1)}\otimes 1$ is left-adjoint to $E\otimes K$
    and $1\otimes \Eb{n}$ is left-adjoint to $K^{-n}\otimes (\overline{F}^{(1)})^n$ (see \Cref{bilinearformquantum}).
    This, together with \Cref{theoremmartel} implies point \textbf{(3)}.
\end{proof}

\begin{proposition}\label{propositionhomologicalred}
    Let $m\geq 0$, there is a natural map:
    \begin{equation*}
        H_m(\Conf{m}{\oD_n,\partial_-\oD_n};\Line_m)\lra H_m^{BM}(\Conf{m}{\oD_n,\partial_-\oD_n};\Line_m).
    \end{equation*}
    Hence we have a natural homological map:
    \begin{equation*}
        \WFt{\ua}\lra \WEt{\ua}.
    \end{equation*}
    Then the following diagram of $(\Ufin,\Mod{\oD_n})$-bimodules commutes
    \[\begin{tikzcd}
        {\VF{a_1}\otimes\dotsb\otimes\VF{a_n}} & {\VE{a_1}\otimes\dotsb\otimes\VE{a_n}} \\
        {\WFt{\ua}} & {\WEt{\ua}}
        \arrow["{\mathrm{red}_{\ua}}", from=1-1, to=1-2]
        \arrow[from=2-1, to=2-2]
        \arrow["{h_{\ua}}"', from=1-1, to=2-1]
        \arrow["{g_{\ua}}", from=1-2, to=2-2]
    \end{tikzcd}\]
    where $\mathrm{red}_{\ua}$ is $\otimes_i\mathrm{red}_{a_i}$ and $\mathrm{red}_{a_i}$ is the map of \Cref{imageFtoEquantum}.
\end{proposition}

\begin{proof}
    It is essentially sufficient to compute the image of $B(m_1,\dotsc,m_n)$ in the basis of $A$-classes and check that it matches $\mathrm{red}_{\ua}$.
    Due to the definition of $B$ classes, such a calculation immediately reduces to the case $n=1$.
    
    In this case we see from the definition of the quantum group actions that $\WFt{\ua}\ra \WEt{\ua}$ is a $\Ufin$ map.
    But $B(m)=(\Fb{1})^mB(0)$, and $B(0)$ is sent to $A(0)$. Hence $B(m)$ is sent to $(\Fb{1})^mA(0)$. This matches $\mathrm{red}_{a}$.
\end{proof}


\subsubsection{A cellular complex for Borel-Moore homology}\label{subsubcellularcomplex}

Let $M$ be a manifold with some local system $\Line$. By a Borel-Moore cellular decomposition for $M$, we mean a finite decomposition of $M$ into disjoint open balls.
Consider such a decomposition and denote by $B_{n,1},\dotsc B_{n,i_n}$ the open balls of dimension $n$.
Then we have a cellular complex:
\begin{equation}\label{equationBMcomplex}
    0\lla \bigoplus_iH_0^{BM}(B_{0,i};\Line)\lla \dotsb\lla \bigoplus_iH_n^{BM}(B_{n,i};\Line)\lla \dotsb
\end{equation}
computing the Borel-Moore homology of $M$ with coefficients in $\Line$. To see this, filter the complex $C_*^{lf}(M;\Line)$ of
locally finite singular chains on $M$ as follows:
\begin{equation*}
    F_kC_*^{lf}(M;\Line) = C_*^{lf}(\bigcup_{n\leq k,\: i}B_{n,i};\Line).
\end{equation*}
Then the spectral sequence $E^r_{p,q}$ associated to this filtration has $E^1$ page:
\begin{equation*}
    E^1_{p,q}= H_{p+q}^{BM}(\bigcup_{n\leq p,\: i}B_{n,i},\bigcup_{n<p,\: i}B_{n,i};\Line).
\end{equation*}
By \Cref{BMpairs}:
\begin{align*}
    H_{p+q}^{BM}(\bigcup_{n\leq p,\: i}B_{n,i},\bigcup_{n<p,\: i}B_{n,i};\Line)&=H_{p+q}^{BM}(\bigcup_{i}B_{p,i};\Line) \\
                                                                         &=\begin{cases}
                                                                            \bigoplus_i H_{p}^{BM}(B_{p,i};\Line) &\text{ if }q=0 \\
                                                                            0&\text{ otherwise.}
                                                                         \end{cases}
\end{align*}
Now this spectral sequence converges to $H_*^{BM}(M;\Line)$, hence the claim.

We will now describe a procedure to obtain such Borel-Moore cellular decompositions for the configuration
spaces $\Conf{m}{D_n}$. The procedure is actually more general: from a Borel-Moore cellular decomposition of a manifold $M$,
one obtains a Borel-Moore cellular decomposition for $\Conf{m}{M}$. But we will only detail it for $D_n$.

Let us begin with a decomposition of $D_n$. $D_n$ is $(0,1)^2\setminus\{p_1,\dotsc,p_n\}$ with $p_k=\left(\frac{k}{n+1},\frac{1}{2}\right)$.
Set $b_k=(\frac{k}{n+1},0)$. Define $I_k$ to be the segment $(p_k,b_k)$ and $S$ to be $D_n\setminus(I_1\cup\dotsb\cup I_n)$.
Since $S\simeq (0,1)^2$, the subsets $I_1,\dotsc,I_n,S$ form a Borel-Moore cellular decomposition of $D_n$.

Fix $m\geq 0$. As a set, $\Conf{m}{D_n}$ has the following decomposition:
\begin{equation}\label{equationdecompositionconfDn}
    \Conf{m}{D_n}=\bigsqcup_{m_1+\dotsb+m_n+s=m}\Conf{m_1}{I_1}\times \dotsb \times\Conf{m_n}{I_n}\times \Conf{s}{S}.
\end{equation}
As mentioned in \Cref{symplexBMhomology}, each $\Conf{m_i}{I_i}$ is on open ball of dimension $m_i$,
hence $\Conf{m_1}{I_1}\times \dotsb \times\Conf{m_n}{I_n}$ is an open ball of dimension $m-s$.
However, $\Conf{s}{S}$ may not be an open ball. Hence we shall decompose it into open balls.

Now as $S\simeq (0,1)^2$, we have:
\begin{equation*}
    \Conf{s}{S}\simeq \bigsqcup_{k,s_1+\dotsb +s_k=s}\Conf{s_1}{(0,1)}\times\dotsb\times\Conf{s_k}{(0,1)}\times\Conf{k}{(0,1)}.
\end{equation*}
where the last map is given by separating the points $x_1,\dotsc,x_s$ according their second coordinates.
This decomposition is sometimes called the Fox-Neuwirth-Fuks stratification of $\R^2$ and is described in \cite{foxBraidGroups1962a}.
This gives a Borel-Moore cellular decomposition of $\Conf{s}{S}$.
See \Cref{cell_of_S} for an example. Together with \Cref{equationdecompositionconfDn},
this gives a Borel-Moore cellular decomposition of $\Conf{m}{D_n}$:

\begin{figure}
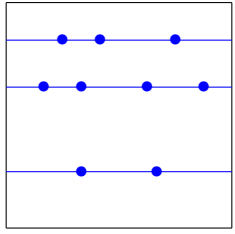

    \ctikzfig{cell_of_S}
    \caption{An element in the cell $\left(\Conf{3}{(0,1)}\times\Conf{4}{(0,1)}\times\Conf{2}{(0,1)}\right)\times\Conf{3}{(0,1)}$
    of $\Conf{9}{S}$.}
    \label[figure]{cell_of_S}
\end{figure}

\begin{proposition}
    Let $m,n\geq 0$, then $\Conf{m}{D_n}$ has a Borel-Moore cellular decomposition where cells are indexed
    by non-negative integers:
    \begin{equation*}
        (\um,\us)=(m_1,\dotsc,m_n\mid s_1,\dotsc,s_k)
    \end{equation*}
    summing to $m$. The cell $(\um\mid \us)$ is:
    \begin{equation*}
        \Conf{m_1}{I_1}\times \dotsb \times\Conf{m_n}{I_n}\times\Conf{s_1}{(0,1)}\times\dotsb\times\Conf{s_k}{(0,1)}\times\Conf{k}{(0,1)}
    \end{equation*}
    and has dimension $m+k$.
\end{proposition}

For examples of cells, see \Cref{example_cell}.

\begin{figure}
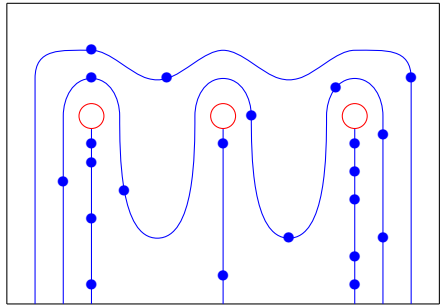

    \ctikzfig{example_cell}
    \caption{An element in the cell $(4,2,5\mid 3,8)$ of $\Conf{22}{D_3}$.}
    \label[figure]{example_cell}
\end{figure}

Let us now use this decomposition to compute the groups $H_m^{BM}(\Conf{m}{D_n};\Line_m)$.
Notice that by \Cref{BMpairs}, we have a map:
\begin{equation*}
    H_m^{BM}(\Conf{m}{\oD_n,\partial_-\oD_n};\Line_m)\ra H_m^{BM}(\Conf{m}{\oD_n,\partial\oD_n};\Line_m)\simeq H_m^{BM}(\Conf{m}{D_n};\Line_m).
\end{equation*}

\begin{theorem}\label{theoremBMhomology}
    Let $n\geq 1$ and $a_1,\dotsc,a_n\in\Lset$. Let $m\geq 0$ and $b\in\Z$ such that $m=\frac{a_1+\dotsb+a_n-b}{2}$.
    Then we have a commutative diagram:
    \[\begin{tikzcd}
        {H_m^{BM}(\Conf{m}{\oD_n,\partial_-\oD_n};\Line_m)} & {H_m^{BM}(\Conf{m}{D_n};\Line_m)} \\
        {\left(\bigotimes_i\hat{V}^{E}_{a_i}\right)_b} & {\left(\bigotimes_i\hat{V}^{E}_{a_i}\right)_b/(\mathrm{im}\: \overline{F}^{(1)}+\mathrm{im}\: \overline{F}^{(r)})}
        \arrow[two heads, from=2-1, to=2-2]
        \arrow[two heads, from=1-1, to=1-2]
        \arrow[tail reversed, from=1-1, to=2-1]
        \arrow[tail reversed, from=1-2, to=2-2]
    \end{tikzcd}\]
    where two-sided arrows are isomorphisms. The module $\left(\bigotimes_i\hat{V}^{E}_{a_i}\right)_b$ is the weight $b$ space of $\bigotimes_i\hat{V}^{E}_{a_i}$,
    as defined in \Cref{weightspaces}.
\end{theorem}

\begin{proof}
    The decomposition of the space $\Conf{m}{D_n}$ we consider has no $(m-1)$-cells.
    Hence we need only compute the image of the $(m+1)$-cells into the free module over the $m$-cells.
    The $m$ cells are indexed by $(m_1,\dotsc,m_n)$ summing to $m$. By definition of the cells and the $A$ classes,
    the cell corresponding to $(m_1,\dotsc,m_n)$ is the image of $A(m_1,\dotsc,m_n)$.

    Now consider a $(m+1)$-cell $(m_1,\dotsc,m_n\mid s)$. We want to show that its image in $H_m^{BM}(\Conf{m}{D_n};\Line_m)$
    is the one of $\Fb{s}A(m_1,\dotsc,m_n)$ up to a non-zero factor. For that, we use a cellular decomposition
    of $\Conf{m}{\oD_n\setminus\partial_-\oD_n}$. Notice that $\oD_n\setminus\partial_-\oD_n$ has the decomposition
    $I_1,\dotsc,I_n,\mathrm{int}(\partial_+\oD_n),S$, ie. with an extra $1$-cell compared to $D_n$.
    So, by proceeding as in the case of $\Conf{m}{D_n}$, we obtain a cellular decomposition of $\Conf{m}{\oD_n\setminus\partial_-\oD_n}$
    indexed by $(\um\mid \us\mid l)=(m_1,\dotsc,m_n\mid s_1,\dotsc,s_k\mid l)$ where $l$ is the number of points on $\mathrm{int}(\partial_+\oD_n)$
    in each element of the cell.
    We have a restriction map from the cellular complex $C_*^+$ of \Cref{equationBMcomplex} for $\Conf{m}{\oD_n\setminus\partial_-\oD_n}$ to
    the complex $C_*$ for $\Conf{m}{D_n}$, that induces the map:
    \begin{equation*}
        H_m^{BM}(\Conf{m}{\oD_n,\partial_-\oD_n};\Line_m)\ra H_m^{BM}(\Conf{m}{D_n};\Line_m)
    \end{equation*}
    when taking homology. Moreover, there is a lift $C_*\ra C^+_*$, $c\mapsto c^+$ given by $(\um,\us)\mapsto (\um,\us,0)$.
    Hence, to compute the image of a cell $c=(m_1,\dotsc,m_n\mid s)$ in $C_m$,
    we need only compute the image of its lift $c^+=(m_1,\dotsc,m_n\mid s\mid 0)$ in $C^+_m$.
    Now, one sees that:
    \begin{equation*}
        \partial (c^+) = \pm(m_1,\dotsc,m_n\mid 0\mid s) + (\partial c)^+.
    \end{equation*}
    Hence, when we look at the image of $\partial (c^+)$ in $H_m^{BM}(\Conf{m}{\oD_n,\partial_-\oD_n};\Line_m)$,
    we have $[\partial (c^+)]=0$ and hence $[(\partial c)^+]=\pm [(m_1,\dotsc,m_n\mid 0\mid s)]$.
    Now, $[(m_1,\dotsc,m_n\mid 0\mid s)]$ is $\Fb{s}A(m_1,\dotsc,m_n)$ up to an invertible constant and $[(\partial c)^+]$ is $\partial c$
    under the identification $C_m\simeq H_m^{BM}(\Conf{m}{\oD_n,\partial_-\oD_n};\Line_m)$. This concludes.
\end{proof}

\begin{remark}
    The complex associated to this decomposition for the local system $\Line_m$ is
    the weight $m$ part of the Hochschild complex of the sub-algebra of $\UE$ generated by the $\Fb{k}$ for $k\geq 1$
    with coefficients in $\VE{a_1}\otimes\dotsb\otimes\VE{a_n}$.
    For a generic value of $q$, the divided powers $\Fb{k}$ can be replaced by powers $\overline{F}^k$,
    and we recover the Hochschild complex described by Schechtman and Varchenko in \cite[3.]{schechtmanQuantumGroupsHomology1991}.
\end{remark}

\begin{corollary}\label{corollaryabsolutehomology}
    Under the hypothesis of \Cref{theoremBMhomology}, we have a commutative diagram:
    \[\begin{tikzcd}
        {H_m(\Conf{m}{D_n};\Line_m)} & {H_m(\Conf{m}{\oD_n,\partial_-\oD_n};\Line_m)} \\
        {\left(\bigotimes_i\hat{V}^{\overline{F}}_{a_i}\right)_b\cap\mathrm{ker}\: E^{(1)}\cap\mathrm{ker}\: E^{(r)}} & {\left(\bigotimes_i\hat{V}^{\overline{F}}_{a_i}\right)_b}
        \arrow[tail reversed, from=1-2, to=2-2]
        \arrow[tail reversed, from=1-1, to=2-1]
        \arrow[hook, from=1-1, to=1-2] 
        \arrow[hook, from=2-1, to=2-2]
    \end{tikzcd}\]
    where two-sided arrows are isomorphisms.
\end{corollary}

\begin{proof}
    The statement is dual to that of \Cref{theoremBMhomology} under the homological bilinear form
    $(\cdot,\cdot):\WFt{\ua}\times\WEt{\ua}\ra \Z[\zeta_r]$ and the bilinear form of \Cref{bilinearformquantum}.
\end{proof}


\subsubsection{Proof of \Cref{theoremgeometricconstruction}, \Cref{theoremintegralstructure} and \Cref{theoremgluinggeometric}}
\label{subsubmainproofs}

We can now prove \Cref{theoremgeometricconstruction}.
We prove the equivalent homological reformulation (\Cref{theoremhomologicalconstruction}).
Let $r\geq 3$ be an odd integer, $n\geq 2$ and $b,a_1,\dotsc,a_n\in\{0,1,\dotsc,r-2\}$,
such that $m=\frac{a_1+\dotsb+a_n-b}{2}$ is an integer.
Let $\Dc=(D^n,b,a_1,\dotsc,a_n)$ be the associated colored disk.
Notice that we defined $\Line_m$ to be $\Line_{\Dc}^\mathcal{O}$.
Then, by \Cref{theoremBMhomology}, \Cref{corollaryabsolutehomology} and \Cref{propositionhomologicalred},
we have a commutative diagram of representations of $\Mod{D_n}$ over $\Z[\zeta_r]$:
\[\begin{tikzcd}
	{H_m(\Conf{m}{D_n};\Line_m)} & {H_m^{BM}(\Conf{m}{D_n};\Line_m)} \\
	{\left(\bigotimes_i\hat{V}^{\overline{F}}_{a_i}\right)_b\cap\mathrm{ker}\: E^{(1)}\cap\mathrm{ker}\: E^{(r)}} & {\left(\bigotimes_i\hat{V}^{E}_{a_i}\right)_b/(\mathrm{im}\: \overline{F}^{(1)}+\mathrm{im}\: \overline{F}^{(r)})} \\
	{\hom_{U_{\zeta_r}^{\overline{F}}}(\hat{V}^{\overline{F}}_b,\bigotimes_i\hat{V}^{\overline{F}}_{a_i})} & {\hom_{U_{\zeta_r}^{E}}(\bigotimes_i\hat{V}^{E}_{a_i},\hat{V}^{E}_b)^\vee}
	\arrow[tail reversed, from=1-1, to=2-1]
	\arrow[tail reversed, from=2-1, to=3-1]
	\arrow[tail reversed, from=1-2, to=2-2]
	\arrow[from=1-1, to=1-2]
	\arrow[from=2-1, to=2-2]
	\arrow["{\mathrm{red}}", from=3-1, to=3-2]
	\arrow["f", from=2-2, to=3-2]
\end{tikzcd}\]
where two-sided arrows are isomorphisms and $f$ is the quotient by torsion.
As, by \Cref{imageofred}, $\mathrm{im}(\mathrm{red})\otimes\Q$ is $\Nusl(\Dc)$ and $\mathrm{im}(\mathrm{red})$
is $(q-q^{-1})^m\NuslO(\Dc)$ when $r$ is prime, this proves \Cref{theoremhomologicalconstruction}
and \Cref{theoremintegralstructure}.

\begin{remark}\label{remarkvanishings}
    If $b<0$ in the diagram above, the image of $\mathrm{red}$ is $0$.
    Indeed, elements in the image of $\mathrm{red}$ are in particular weight $b$
    elements of $V_{a_1}\otimes \dotsb\otimes V_{a_n}$ in the kernel of $E$ and $E^{(r)}$.
    But no finite dimensional representation of $U_{\zeta_r}$ has such non-zero elements
    (see the proof of \Cref{vermaproperty}).

    If $a_i=r-1$ for some $i$ in the diagram above,
    then the image of $\mathrm{red}$ is again $0$. Indeed,
    by \Cref{propositionhomologicalred}, it is also the image of:
    \begin{equation*}
        \hom_{U_{\zeta_r}}(V_b,\bigotimes_iV_{a_i})\lra  \hom_{U_{\zeta_r}}(\bigotimes_iV_{a_i},V_b)^\vee
    \end{equation*}
    and by \Cref{qtracedefinitionisevaluationdefinition}, this is just:
    \begin{equation*}
        \bhom{V_b}{\bigotimes_iV_{a_i}}
    \end{equation*}
    But now, $\id_{V_{a_i}}=\id_{V_{r-1}}$ has quantum trace $[r]=0$.
    So $\bhom{V_b}{\bigotimes_iV_{a_i}}=0$.
\end{remark}

Let us now turn to \Cref{theoremgluinghomological}.
Let $\gamma$ be a simple closed curve in $D_n$. Without loss of generality, we assume that the punctures inside $\gamma$
are $p_1,\dotsc,p_{n_1}$ with $n_1+n_2=n$.
Define $\Dc_{c}'$ and $\Dc_{c}''$ for each $0\leq c\leq r-2$ as in \Cref{subsubrootingtopological}.
Let $C=D_n$ and $C_\gamma = C\setminus \gamma$.
One has $C_\gamma\simeq D_{n_1}\sqcup D_{n_2+1}$.

Let $m_1+m_2=m$ and let $c$ be such that $m_1=\frac{a_1+\dotsb+a_{n_1}-c}{2}$, or equivalently $m_2=\frac{c+a_{n_1+1}+\dotsb+a_{n}-b}{2}$.
Notice that if $c<0$ or $c\geq r-1$, then $\Hmid(\Conf{m_1}{D_{n_1}};\Line_{\Dc'_c})=0$ or $\Hmid(\Conf{m_2}{D_{n_2+1}};\Line_{\Dc''_c})=0$.
Hence, we may assume that $0\leq c\leq r-2$.

To show that the diagram of \Cref{theoremgluinggeometric} commutes, it is sufficient to show that the following diagram commutes:

\[\begin{tikzcd}
	{\left(\left(\hat{V}^{\overline{F}}_{a_1}\otimes\dotsb\otimes\hat{V}^{\overline{F}}_{a_{n_1}}\right)_c\cap\mathrm{ker}\: E^{(1)}\cap\mathrm{ker}\: E^{(r)}\right)\otimes\left(\hat{V}^{\overline{F}}_{c}\otimes\hat{V}^{\overline{F}}_{a_{n_1+1}}\otimes\dotsb\otimes\hat{V}^{\overline{F}}_{a_{n}}\right)_b} & {\left(\hat{V}^{\overline{F}}_{a_{1}}\otimes\dotsb\otimes\hat{V}^{\overline{F}}_{a_{n}}\right)_b} \\
	{H_{m_1}(\Conf{m_1}{D_{n_1}};\Line_{\Dc_{c}'})     \otimes H_{m_2}(\Conf{m_2}{\oD_{n_2+1},\partial_-\oD_{n_2+1}};\Line_{\Dc_{c}''})} & {H_m(\Conf{m}{\oD_n,\partial_-\oD_{n}};\Line_{\Dc})}
	\arrow["{\mathrm{comp}}", from=1-1, to=1-2]
	\arrow[tail reversed, from=1-2, to=2-2]
	\arrow[tail reversed, from=1-1, to=2-1]
	\arrow["\cup"', from=2-1, to=2-2]
\end{tikzcd}\]
where two-sided arrows are isomorphisms and $\mathrm{comp}$ satisfies:
\begin{equation*}
    \mathrm{comp}(u\otimes\left(e_k\otimes v\right))=\left((\overline{F}^{(1)})^ku\right)\otimes v.
\end{equation*}

Let us now choose $x$ in $H_{m_1}(\Conf{m_1}{D_{n_1}};\Line_{\Dc_{c}'})$ and $y$ in $H_{m_2}(\Conf{m_2}{\oD_{n_2+1},\partial_-\oD_{n_2+1}};\Line_{\Dc_{c}''})$.
As the $B$-classes form a basis for $H_{m_2}(\Conf{m_2}{\oD_{n_2+1},\partial_-\oD_{n_2+1}};\Line_{\Dc_{c}''})$, we will assume that
$y$ is one, say $y=B(k,s_{1},\dotsc,s_{n_2})$ with $m_2=k+s_1+\dotsb+s_{n_2}$. The class $x\cup y$
is represented in \Cref{gluing}.

\begin{figure}
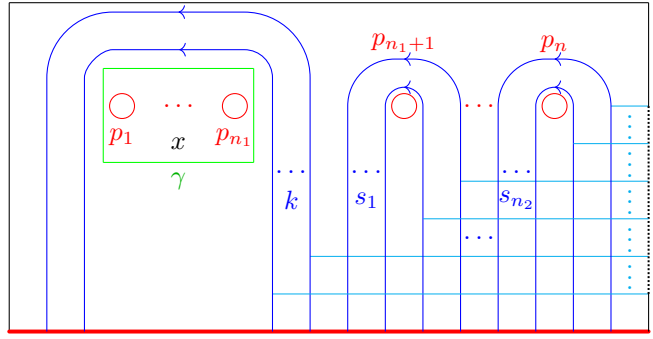

    \ctikzfig{gluing}
    \caption{The class $x\cup y$ for $y=B(k,s_{1},\dotsc,s_{n_2})$.}
    \label[figure]{gluing}
\end{figure}

Now consider the gluing map $\oD_{n_1}\cup\oD_{n_2}\simeq \oD_{n}$ as in \Cref{gluing_monoidal}.
It induces a map:
\begin{multline*}
    \times:H_{m_1}(\Conf{m_1}{\oD_{n_1}\partial_-\oD_{n_1}};\Line_{\Dc_{c}'})\otimes H_{m_2}(\Conf{m_2}{\oD_{n_2},\partial_-\oD_{n_2}};\Line_{\Dc_{c}''}) \\
    \lra H_m(\Conf{m}{\oD_n,\partial_-\oD_{n}};\Line_{\Dc}).
\end{multline*}
Set $y'=B(s_{1},\dotsc,s_{n_2})$ in $H_{m_2}(\Conf{m_2}{\oD_{n_2},\partial_-\oD_{n_2}};\Line_{\Dc_{c}''})$.
Then we have $x\cup y=((\overline{F}^{(1)})^kx)\times y'$. This matches the map $\mathrm{comp}$, and concludes the proof of \Cref{theoremgluinghomological}.

\begin{figure}
    \ctikzfig{gluing_monoidal}
    \caption{The gluing $\oD_{n_1}\cup\oD_{n_2}\simeq \oD_{n}$.}
    \label[figure]{gluing_monoidal}
\end{figure}


\subsubsection{The Bigelow basis for Temperley-Lieb representations}\label{subsubBigelow}

In this section, we describe a homological basis for $\Hmid(\Conf{m}{D_n};\Line_m)$
in the case where $a_1=\dotsb=a_n=1$. In this case, as already mentioned in \Cref{TLrepresentations},
$\Nusl(\Dc)$ is a representation of the Temperley-Lieb algebra. The homological basis is inspired
by the one used by Stephen Bigelow for the Lawrence-Krammer representation in \cite{bigelowLawrenceKrammerRepresentation2003}.
We then explain how to lift this homological basis to $H_m(\Conf{m}{D_n};\Line_m)$, using Stephen Bigelow's "noodle" homology classes.

Throughout this section, we will use a generic parameter $q$, not a root of unity.
The local system $\Line_m$ can be lifted from $\Z[\zeta_r]$ to $\Z[q,q^{-1}]$
by the same formulas of \Cref{subsubUaction}.
We will denote $\Lq_m$ this lift.

From the remainder of this section, let us fix $m\geq 0$, $b\in\N$ and $n\geq 0$ such that $m=\frac{n-b}{2}$.
We will denote $\Dc=(D_n,b,1,\dotsc,1)$ the associated colored disk. The module $S(\Dc)$ has a basis
given by balanced sequences of "$\bm{\cdot}$", "$\bm{(}$" and "$\bm{)}$" of length $n$ with $m$ pairs of parentheses,
such that no "$\bm{\cdot}$" lies inside a matching pair of parentheses.

\begin{notation}
Let us denote $W$ the set of such sequences. For each $w$ in $W$, let $\pr{w}$ be the set of pairs of indices of matching
parentheses in $w$. We will call a sequence $\um\in \{0,1\}^n$ compatible with $w$ if:
\begin{description}
    \item[(1)] whenever the character at index $i$ in $w$ is "$\bm{\cdot}$", $m_i=0$;
    \item[(2)] for each $(o,c)$ in  $\pr{w}$, $\{m_o,m_c\}=\{0,1\}$\footnote{Here $o$ stands for \textit{opening} and $c$ for \textit{closing}.}.
\end{description}
We will denote by $\cp{w}$ the set of sequences $\um\in \{0,1\}^n$ compatible with $w$.
\end{notation}

\begin{definition}
    Let $w\in W$ and $\um\in\cp{w}$. Then denote by $f(w,\um)$ the quantity:
    \begin{equation*}
        -m(2+n)+\sum_iim_i + \frac{u(\um)-v(\um)}{2}
    \end{equation*}
    where $u(\um)$ is the number of pairs $(o,c)\in \pr{w}$ such that $m_o=1$ and 
    $v(\um)$ is the number of pairs $(o,c)\in \pr{w}$ such that $m_c=1$.
\end{definition}

\begin{definition}\label{definitionCclasses}
    Let $w\in W$ and $\um\in\cp{w}$. For each pair $(o,c) \in \pr{w}$, define $I_{o,c}\subset D_n$ to be the open segment
    from $p_c$ to $p_o$, going under each $I_{o',c'}$ for $(o',c') \in \pr{w}$ such that $o<o'<c'<c$. See \Cref{C_classes} for an example.
    Let $\gamma$ be the path from the basepoint $x$ in $\Conf{m}{D_n}$ to $\prod_{(o,c)}I_{(o,c)}$,
    where the $i$-th path $\gamma_i$ of $\gamma$ goes horizontally from $x_i$ to near $p_j$,
    where $j$ is the index of the $i$-th $1$ in $\um$.
    See \Cref{C_classes} for an example.
    
    Denote by $C'(w,\um)$ the cycle $[(I_{o,c})_{(o,c)\in\pr{w}},\gamma]$ in $H_m^{BM}(\Conf{m}{\oD_n,\partial_-\oD_n};\Lq_m)$.
\end{definition}

\begin{figure}
    \ctikzfig{C_classes}
    \caption{An example of $w$, $\um$ and associated class $C'(w,\um)$.}
    \label[figure]{C_classes}
\end{figure}

We now construct lifts of the classes $C'(w,\um)$ to $H_m(\Conf{m}{D_n};\Lq_m)$, using Stephen Bigelow's "noodle" classes
(as in \cite{bigelowHomologicalDefinitionJones2002}).

\begin{definition}
    Let $w\in W$. Consider the segments $I_{(o,c)}\subset D_n$ for $(o,c)\in \pr{w}$ as in \Cref{definitionCclasses}.
    For each $(o,c)$, choose in a neighborhood $I_{(o,c)}$ not meeting the other segments an immersion of 
    $S^1$ in $D_n$ as in \Cref{Cn_classes}. Denote it by $J_{(o,c)}$. Note that $\prod_{(o,c)}J_{(o,c)}\subset \Conf{m}{D_n}$
    can be lifted to the covering of $\Conf{m}{D_n}$ corresponding to $\rho$.
    Hence, for $\um\in\cp{w}$, we can define a class $\Cn'(w,\um)$ by:
    \begin{equation*}
        \Cn'(w,\um) =[\prod_{(o,c)}J_{(o,c)},\gamma]\in H_m(\Conf{m}{D_n};\Lq_m)
    \end{equation*}
    where $\gamma$ is as in \Cref{definitionCclasses} and \Cref{Cn_classes}.
\end{definition}

\begin{figure}
    \ctikzfig{Cn_classes}
    \caption{An example of $w$, $\um$ and associated class $\Cn'(w,\um)$.}
    \label[figure]{Cn_classes}
\end{figure}

\begin{proposition}
    Let $w\in W$. Then $q^{f(w,\um)}C'(w,\um)$ is independent of $\um\in\cp{w}$ and will be denoted $C(w)$.
    Similarly $q^{m+f(w,\um)}\Cn'(w,\um)$ is independent of $\um\in\cp{w}$ and will be denoted $\Cn(w)$.
\end{proposition}

\begin{proof}
    Let $\um,\um'\in W$ that differ only at indices $o$ and $c$ for $(o,c)\in\pr{w}$,
    say $m_o=1$ and $m'_c=1$. Using the relation in \Cref{switch_relation}, we see that
    $C'(w,\um')=q^{2(\frac{c-o-1}{2})}C'(w,\um)=q^{c-o-1}C'(w,\um)$, as $\frac{c-o-1}{2}$ counts the number of intervals $I_{(o',c')}$
    lying above $I_{(o,c)}$. Now $f(w,\um)-f(w,\um')=1+o-c$. Hence $q^{f(w,\um)}C'(w,\um)=q^{f(w,\um')}C'(w,\um')$.
    The proof for $\Cn(w)$ is the same.
\end{proof}

\begin{figure}
    \ctikzfig{switch_relation}
    \caption{A local relation in the calculus of classes in $H_m(\Conf{m}{D_n};\Lq_m)$.}
    \label[figure]{switch_relation}
\end{figure}

\begin{proposition}\label{propositionrelationCntoC}
    Let $w\in W$, then the image of $\Cn(w)$ in $H_m^{BM}(\Conf{m}{D_n};\Lq_m)$ is $(q-q^{-1})^mC(w)$.
    For $\um\in\cp{w}$, the image of $\Cn'(w,\um)$ in $H_m^{BM}(\Conf{m}{D_n};\Lq_m)$ is $(1-q^{-2})^mC'(w,\um)$.
\end{proposition}

The proof is simply the computation of \Cref{noodle_to_fork}.

\begin{figure}
    \ctikzfig{noodle_to_fork}
    \caption{The proof of \Cref{propositionrelationCntoC}.}
    \label[figure]{noodle_to_fork}
\end{figure}

\begin{theorem}\label{theoremtwomodels}
    The map:
    \begin{align*}
        S(\Dc) &\lra         H_m(\Conf{m}{\oD_n};\Lq_m) \\
        e_w    &\longmapsto  \Cn(w)
    \end{align*}
    fits into the following commutative diagram:
    \[\begin{tikzcd}
        {S(\Dc)} & {H_m(\Conf{m}{\oD_n};\Lq_m)} \\
        {(q-q^{-1})^m\left(V_1^{\otimes n}\right)_b} & {H_m^{BM}(\Conf{m}{\oD_n,\partial_-\oD_n};\Lq_m)}
        \arrow["{(q-q^{-1})^m\mathrm{Op}_{\Dc}(e_0)}"', from=1-1, to=2-1]
        \arrow["i", hook, from=2-1, to=2-2]
        \arrow[from=1-2, to=2-2]
        \arrow[from=1-1, to=1-2]
    \end{tikzcd}\]
    where $V_1$ is the $q$-generic dimension $2$ representation of $U_{\Z[q^{\pm 1}]}$ given by the same formulas as
    in \Cref{definitionfinitedimensionalmodules}. Moreover, the diagonal map:
    \begin{equation*}
        S(\Dc) \lra H_m^{BM}(\Conf{m}{\oD_n,\partial_-\oD_n};\Lq_m)
    \end{equation*}
    is a map of $\Mod{D_n}$-representations.
\end{theorem}

\begin{proof}
    Given $w\in W$, we will compute $C(w)$ in the basis of $A$-classes and check that it matches
    the map $\mathrm{Op}_{\Dc}(e_0)$. This will prove the commutation of the diagram.
    From \cite[(3.23)]{ohtsuki2001quantum}, $\mathrm{Op}_{\Dc}(e_0)$
    is given by placing at each cup the map:
    \begin{align*}
        \Z[q^{\pm \frac{1}{2}}]&\lra         V_1\otimes V_1 \\
        1                      &\longmapsto  -q^{-\frac{1}{2}}e_0\otimes e_1 + q^{\frac{1}{2}}e_1\otimes e_0.
    \end{align*}
    Hence the image of $e_w$ in $V_1^{\otimes n}$ is:
    \begin{equation*}
        \sum_{\um\in \cp{w}} (-1)^{v(\um)}q^{\frac{u(\um)}{2}-\frac{v(\um)}{2}}e_{m_1}\otimes\dotsb\otimes e_{m_n}.
    \end{equation*}
    On the other hand, $C(w)$ can be computed, using repeatedly the relation in \Cref{cut_relation}, to be:
    \begin{align*}
        C(w) &= \sum_{\um\in \cp{w}} q^{f(w,\um)}(-1)^{v(\um)}A'(\um) \\
             &= \sum_{\um\in \cp{w}} q^{\frac{u(\um)}{2}-\frac{v(\um)}{2}}(-1)^{v(\um)}A(\um).
    \end{align*}
    where $A'(\um)$ denotes the class $A(\um)$ without the renormalization constant.
    Thus $C(w)$ and $e_w$ have the same image in $H_m^{BM}(\Conf{m}{\oD_n,\partial_-\oD_n};\Lq_m)$.

    Now, $(q-q^{-1})^m\mathrm{Op}_{\Dc}(e_0)$ and $i$ in the diagram are maps of 
    $\Mod{D_n}$-representations. Thus, so is their composition.
\end{proof}

\begin{figure}
    \ctikzfig{cut_relation}
    \caption{A local relation near $\partial_-\oD_n$ in the calculus of classes in $H_m(\Conf{m}{\oD_n,\partial_-\oD_n};\Lq_m)$.}
    \label[figure]{cut_relation}
\end{figure}


\subsubsection{Proof of \Cref{theoremintersectionform}}\label{subsubproofintersection}

Over $\Z[q^{\pm 1}]$ or $\Q(q)$, when discussing Hermitian forms we consider the involution such that $\overline{q}=q^{-1}$.
We keep the notations of the previous subsubsection. In particular, $\Dc=(D_n,b,1,\dotsc,1)$ such that $m=\frac{n-b}{2}$ is an integer.

\begin{proposition}
    The $\Mod{D_n}$-representation $S(\Dc)\otimes \Q(q)$ is irreducible.
\end{proposition}

\begin{proof}
    The module $S(\Dc)\otimes \Q(q)$ is irreducible as a $\Mod{D_n}$-representation if and only if it is irreducible as a representation
    of the Temperley-Lieb algebra $TL_n$. But it is just the representation of $TL_n$ associated to the Young diagram $(m+b,m)$,
    which is irreducible over $\Q(q)$ (see \cite[4]{westburyRepresentationTheoryTemperleyLieb1995} for example).
\end{proof}

\begin{corollary}\label{corollaryuniqueform}
    If $h'$ is a $\Mod{D_n}$-invariant sesquilinear form on $S(\Dc)$ then $h'=a h_{\Dc}$ for some $a\in \Q(q)$.
\end{corollary}

\begin{proof}
    Such a form $h'$ induces a $\Mod{D_n}$-invariant map $S(\Dc)\ra \overline{S(\Dc)}^\vee$. Now, $h_{\Dc}$ is perfect on $S(\Dc)\otimes \Q(q)$.
    Hence we have a $\Mod{D_n}$-invariant map $\overline{S(\Dc)}^\vee\otimes \Q(q)\ra S(\Dc)\otimes\Q(q)$. Composing the two maps,
    we get a $\Mod{D_n}$-invariant endomorphism of $S(\Dc)\otimes \Q(q)$. Now, as $S(\Dc)\otimes \Q(q)$ is irreducible, by Schur's lemma
    this endomorphism is $a\cdot \id$ for some $a\in \Q(q)$. Thus $h'=ah_{\Dc}$.
\end{proof}

\begin{proposition}\label{propositiongenericform}
    Consider the map:
    \begin{equation*}
        f:S(\Dc) \lra H_m(\Conf{m}{D_n};\Lq_m)
    \end{equation*}
    induced by $\Cn$-classes as in \Cref{theoremtwomodels}. Let $s_{\Dc}$ be the intersection form on $H_m(\Conf{m}{D_n};\Lq_m)$.
    Then we have:
    \begin{equation*}
        f^*s_{\Dc} = (q-q^{-1})^m h_{\Dc}.
    \end{equation*}
\end{proposition}

\begin{proof}
    Let $p:H_m(\Conf{m}{D_n};\Lq_m)\ra \Hmid(\Conf{m}{D_n};\Lq_m)$ be the projection. Then $s_\Dc$ is defined on $\Hmid(\Conf{m}{D_n};\Lq_m)$.
    As by \Cref{theoremtwomodels}, $p\circ f$ is a map of $\Mod{D_n}$-representations, $f^*s_{\Dc}$ is a $\Mod{D_n}$-invariant sesquilinear form on $S(\Dc)$.
    Hence, by \Cref{corollaryuniqueform}, $f^*s_{\Dc}=ah_{\Dc}$ for some $a\in \Q(q)$. Let us now compute $a$ by computing both forms on
    well chosen elements. Let $w$ be the sequence:
    \begin{equation*}
        "\bm{(}","\bm{)}",\dotsc,"\bm{(}","\bm{)}","\bm{\cdot}",\dotsc,"\bm{\cdot}"
    \end{equation*}
    with $n$ characters, $2m$ of which are parentheses.
    Set:
    \begin{equation*}
        \um=(0,1,0,1,\dotsc,0,1,0,\dotsc,0)\in\cp{w}.
    \end{equation*}
    Then the second Skein relation yields $h_{\Dc}(e_w,e_w)=(-q-q^{-1})^m$. Now by \Cref{propositionrelationCntoC},
    $s_{\Dc}(\Cn'(w,\um),\Cn'(w,\um))=(1-q^{-2})^ms_{\Dc}(C'(w,\um),\Cn'(w,\um))$.
    The geometric intersection of $C'(w,\um)$ and $\Cn(w,\um)$ is represented in \Cref{intersection_computation},
    and is computed to be $(-1-q^2)^m$. We get:
    \begin{align*}
        s_{\Dc}(\Cn(w),\Cn(w)) &= s_{\Dc}(\Cn'(w,\um),\Cn'(w,\um))\\
                               &=(1-q^{-2})^m(-1-q^2)^m \\
                               &=(-q)^m(1-q^{-2})^m(-q-q^{-1})^m \\
                               &=(q-q^{-1})^mh_{\Dc}(e_w,e_w).
    \end{align*}
    Hence $a=(q-q^{-1})^m$.
\end{proof}

\begin{figure}
    \ctikzfig{intersection_computation}
    \caption{The geometric intersection of $C'(w,\um)$ and $\Cn(w,\um)$ as in the proof of \Cref{propositiongenericform} for $n=6$ and $m=2$.}
    \label[figure]{intersection_computation}
\end{figure}

\begin{corollary}\label{corollaryformsforTL}
    Over $\Q(\zeta_r)$, under the isomorphism:
    \begin{equation*}
        \Nusl(\Dc) \simeq \Hmid(\Conf{m}{D_n};\Line_{\Dc})
    \end{equation*}
    we have:
    \begin{equation*}
        s_{\Dc} = (q-q^{-1})^m h_{\Dc}.
    \end{equation*}
    where $h$ is the Hermitian structure of $\Nusl$ and $s_{\Dc}$ is the intersection form.
\end{corollary}

\begin{proof}
    Specializing coefficients by $A_b\ra\Z[\zeta_r]$, we get a commutative diagram:
    \[\begin{tikzcd}
        {S(\Dc)} & {H_m(\Conf{m}{D_n};\Lq_m)} \\
        & {H_m(\Conf{m}{D_n};\Line_m)} \\
        {\Nusl(\Dc)} & {\Hmid(\Conf{m}{D_n};\Line_{\Dc})}
        \arrow["{\mathrm{ev}_{\zeta_r}}", from=1-2, to=2-2]
        \arrow["p", from=2-2, to=3-2]
        \arrow[tail reversed, from=3-1, to=3-2]
        \arrow["\pi"', two heads, from=1-1, to=3-1]
        \arrow["f", hook, from=1-1, to=1-2]
    \end{tikzcd}\]
    where ${\mathrm{ev}_{\zeta_r}}$ and $p$ are compatible with intersection forms.
    Now, as $\pi$ is surjective, the relation $f^*s_{\Dc} = (q-q^{-1})^m h_{\Dc}$ in $S(\Dc)$
    induces $s_{\Dc} = (q-q^{-1})^m h_{\Dc}$ in $\Nusl(\Dc)$.
\end{proof}

To conclude the proof, we need to reduce the equality between the forms to the case $a_1=\dotsb=a_n=1$
considered in \Cref{corollaryformsforTL}.
Choose $\Dc=(D_n,b,a_1,\dotsc,a_n)$ any colored disk. Then denote, for $i=1,\dotsc,n$,
by $\Dc_i$ the disk $(D_{a_i},a_i,1,\dotsc,1)$. Set $\Dc'=(D_{a_1+\dotsb +a_n},b,1\dotsc,1)$. 
Then, applying \Cref{theoremgluinghomological} repeatedly, we get a commutative diagram:
\[\begin{tikzcd}
	{\Nusl(\Dc)\otimes\bigotimes_i\Nusl(\Dc_i)} & {\Nusl(\Dc')} \\
	{\Hmid(\Conf{m}{D_n};\Line_{\Dc})\otimes \bigotimes_i\Hmid(\Conf{0}{D_{a_i}};\Line_{\Dc_i})} & {\Hmid(\Conf{m}{D_{a_1+\dotsb+a_n}};\Line_{\Dc'})}
	\arrow[hook, from=1-1, to=1-2]
	\arrow[tail reversed, from=1-1, to=2-1]
	\arrow[tail reversed, from=1-2, to=2-2]
	\arrow[hook, from=2-1, to=2-2]
\end{tikzcd}\]
where two-sided arrows are isomorphisms. Moreover, the top inclusion is compatible with the Hermitian structure $h$,
and the bottom inclusion is compatible with intersection forms.

Now, by \Cref{corollaryformsforTL}, the equality of forms is valid
for $\Dc'$ and the $\Dc_i$. Actually, $\Nusl(\Dc_i)\simeq \Q(\zeta_r)$ canonically with trivial form $(a,b)\mapsto a\overline{b}$.
Hence, $h_{\Dc}$ is the restriction of $h_{\Dc'}$ and $s_{\Dc}$ is the restriction of $s_{\Dc'}$.
As $s_{\Dc'} = (q-q^{-1})^m h_{\Dc'}$, by restriction we have $s_{\Dc} = (q-q^{-1})^m h_{\Dc}$. This proves \Cref{theoremintersectionform}.


\section{Integral variations of Hodge structures.}\label{sectionproofHodge}

\subsection{Remarks on stacks.}

Although we use stacks, we will not need Hodge theory of stacks and will always be able to reduce to the case of varieties as all
stacks considered will be of the following kind.

\begin{definition}
    An stack over $\C$ is a global finite quotient if it has a finite Galois cover by a variety.
\end{definition}

Note that a global finite quotient may have non trivial generic inertia. If $\mathcal{S}=[V/G]$,
then we always have $\mathrm{H}^*(\mathcal{S};\Q)=\mathrm{H}^*(V;\Q)^G$.
There is a subtelty in defining cohomology of $\mathcal{S}$ with coefficients in $\Z$
but we shall not be concerned by it as we will only be interested in the torsion free part and we have
$\mathrm{H}^*(\mathcal{S};\Z)/\mathrm{torsion}=(\mathrm{H}^*(V;\Z)/\mathrm{torsion})^G$. This can be deduced from the spectral sequence of
the fibration $V\ra\mathcal{S}\ra\mathrm{B}G$.

\begin{theorem}\label{globalquotient}
    Let $r\geq 1$, $m,n\geq 0$ with $m+n\geq 3$. Then $\Mrpb{m}{n}{r}$ is a global finite quotient.
\end{theorem}

\begin{proof}
    If $n>0$, we see that $\Mrpb{m}{n}{r}$ is a trivial gerbe over $\Mrpb{m+n-1}{1}{r}$. Hence we may assume that $n$ is $0$ or $1$.
    Let us treat the case $n=1$ first. In that case $\pi_1(\Mrpb{m}{1}{r})$ is the quotient of the pure braid group on $m$ strands
    by $r$-th powers of Dehn twists.

    Its abelianization is isomorphic to $(\Z/r\Z)^{\binom{m}{2}}$ with coordinates generated by Dehn twists $T_{i,j}$
    around pairs of points $1\leq i<j\leq m$ (see \cite[9.3]{farbPrimerMappingClass2011}).
    Now, a non-trivial element of a local isotropy group corresponds
    to a product $T_{\gamma_1}^{a_1}\dotsb T_{\gamma_k}^{a_k}$ of Dehn twists around disjoint non-isotopic {\scc}
    $\gamma_1,\dotsc,\gamma_k$ on the disk, with $a_1,\dotsc,a_k\neq 0\;(\mathrm{mod}\;r)$.
    Let $\gamma_l$ be the outermost curve. Then it contains a point $j$ that is contained in no other $\gamma_{l'}$.
    Let $i$ be another point inside $\gamma_l$. Then $T_{\gamma_1}^{a_1}\dotsb T_{\gamma_k}^{a_k}$ has coefficient $a_l\neq 0$
    on the coordinate corresponding to $\{i,j\}$.

    Hence the abelianization map $\pi_1(\Mrpb{m}{1}{r})\ra (\Z/r\Z)^{\binom{m}{2}}$ is injective on local isotropy groups. The corresponding
    finite Galois cover has no inertia and is thus a variety.

    In the case $n=0$, $\pi_1(\Mrp{m}{r})$ is the quotient of the pure braid group on $m-1$ strands by its center and the abelianization
    is the quotient of $(\Z/r\Z)^{\binom{m-1}{2}}$ by the vector $(1,\dotsc,1)$. We can repeat the argument above,
    noticing that we may find $3$ distinct points $i$, $i'$ and $j$ with $i$ and $j$ inside $\gamma_l$, $j$
    inside no other $\gamma_{l'}$ and $i'$ outside $\gamma_l$. Then $T_{\gamma_1}^{a_1}\dotsb T_{\gamma_k}^{a_k}$
    has coefficient $0$ on coordinate $\{j,i'\}$ and $a_l\neq 0$ on coordinate $\{j,i\}$ in $(\Z/r\Z)^{\binom{m-1}{2}}$.
    Hence $T_{\gamma_1}^{a_1}\dotsb T_{\gamma_k}^{a_k}$ is non-zero in the quotient by $(1,\dotsc,1)$.
    Again, the abelianization describes a Galois cover by a variety.
\end{proof}

Note that generalizations of this fact to higher genus are proved using $\mathfrak{sl}_2$ quantum representations by Eyssidieux and Funar,
see \cite[Thm 1.1]{eyssidieuxOrbifoldKahlerGroups2021}.
In the following subsections, we will mostly omit the reduction to the case of varieties when it is relatively easy.

\subsection{Proof of existence of integral variations of Hodge structures.}

Here we prove \Cref{corollaryHodge} by explaining how the geometric construction shows that the quantum representations have geometric origin\footnote{See \Cref{definitiongeometricorigin}.} and yields the integral variation of Hodge structures.

Let $b,a_1,\dotsc,a_n\in\{0,1,\dotsc,r-2\}$ such that $m=\frac{a_1+\dotsb+a_n-b}{2}$ is a non-negative integer.
As $\Line=\NuabO(b;\ua,(r-2)^{\bullet m})\otimes \epsilon$ has finite monodromy with image in $\mu_{2r}$. The monodromy action on $\mu_{2r}$ corresponds to an abelian covering:
\begin{equation*}
    f:\widetilde{\mathcal{M}}\lra\sMrpo{m+n}{r}/S_m.
\end{equation*}
Then $\Line=(f_*\underline{\Z})^\epsilon\subset f_*\underline{\Z}$ is the isotropy subspace corresponding to the signature representation of $\mu_{2r}$,
with its natural $\Z[\zeta_r]=\Z[\mu_{2r}]^\epsilon$-structure.
Now let $p$ be the forgetful map:
\begin{equation*}
    p:\sMrpo{m+n}{r}/S_m\lra \sMrpo{n}{r}.
\end{equation*}
Then, as $f$ is a finite covering:
\begin{equation*}
    R(p\circ f)_*=Rp_*\circ Rf_*=Rp_*\circ f_*\text{ and }R(p\circ f)_!=Rp_!\circ Rf_!=Rp_!\circ f_*.
\end{equation*}
Hence \Cref{theoremgeometricconstruction} says that $\Nusl(b;\ua)$ is a sub local system of
the weight $0$ part of $R^m(p\circ f)_*\underline{\Q}$. Consider the maps:
\begin{equation*}
    \overline{f}:\overline{\widetilde{\mathcal{M}}}\lra\Mrpo{m+n}{r}/S_m\text{ and }\overline{p}:\Mrpo{m+n}{r}/S_m\lra \Mrpo{n}{r}
\end{equation*}
where $\overline{f}$ is the covering associated to $\Line$.
Then $\overline{p}\circ \overline{f}$ is proper and $\left(R^m(\overline{p}\circ \overline{f})_*\underline{\Z}\right)_{\mid{\sMrpo{n}{r}}}$
surjects onto the weight $0$ part of $R^m(p\circ f)_*\underline{\Z}$. Hence, on $\sMrpo{n}{r}$,
$\Nusl(b;\ua)$ is a sub local system of a quotient of
$R^m(\overline{p}\circ \overline{f})_*\underline{\Q}$, and is of geometric origin.
The non-framed version of $\Nusl(b;\ua)$ is also of geometric origin by \Cref{propprojectiveequivalence}.

By general Hodge theory, local systems of geometric origin
support integral variations of Hodge structures (see \cite{griffithsPeriodsIntegralsAlgebraic1970} for example).
However, for the convenience of the reader, we give an argument here in our case.
Now, as $p\circ f$ is smooth and a topological fibration, $R^m(p\circ f)_*\underline{\Z}$ supports an integral variation of
mixed Hodge structures. Similarly for $R^m(p\circ f)_!\underline{\Z}$. Moreover the map:
\begin{equation*}
    R^m(p\circ f)_!\underline{\Z}\lra R^m(p\circ f)_*\underline{\Z}
\end{equation*}
is a map of Hodge structures and factors through the degree $0$ part of the weight filtration of $R^m(p\circ f)_*\underline{\Z}$.
Hence its image supports an integral variation of (pure) Hodge structures.
But over $\Q$, the $\epsilon$-isotropic part of this image for the action of $\mu_{2r}$ is isomorphic to $\Nusl(b;\ua)$
by \Cref{theoremgeometricconstruction}.
So $\Nusl(b;\ua)$ supports an integral variation of Hodge structures on $\sMrpo{n}{r}$.
Now, by a result of Griffiths (see \cite[(9.5)]{griffithsPeriodsIntegralsAlgebraic1970}), the structure extends to $\Mrpo{n}{r}$.


\subsection{Proof of \Cref{gluingishodge}}\label{subproofgluinghodge}

In this subsection, we prove that gluing isomorphisms are isomorphisms of rational variations of Hodge structures.

Let $\Mb_m$ be the pre-image of $\justM=\sMrpo{n}{r}$ in $\Mrpo{n+m}{r}/S_m$
and $\Mpb_m$ be the pre-image of $\Mp$ in $\Mrpo{n+m}{r}/S_m$.
We shall use the notations $\pb:\Mb_m\ra \justM$ and $\pb':\Mpb_m\ra \Mp$
for the projections.

\begin{proposition}\label{pbsmooth}
  The map $\pb:\Mb_m\ra \justM$ is smooth.
\end{proposition}
\begin{proof}
  Let $[C]$ be a smooth curve in $\Mrpo{n}{r}$ and $[C_m]$ be a nodal curve in $\Mrpo{n+m}{r}$ over $[C]$.
  Let us show that $\pi_m:\Mrpo{n+m}{r}\ra\Mrpo{n}{r}$ is smooth at $[C_m]$ by induction on $m$.

  Let $\Gamma$ be the graph of $C_m$, where components represent vertices, nodes edges and markings half-edges (see, for example \cite[0.2]{pandharipandeRelationsOverlineMathcal2015}).
  Note that since we are in genus $0$, this graph is a tree.

  Let $a_1,\dotsc,a_n$ be the half-edges corresponding to the markings of $[C]$ and let $b_1,\dotsc,b_m$ be the half-edges that are forgotten by $\pi_m$.
  We claim that there is a vertex $v$ in the graph that is connected to at most one edge $e$, at least a half-edge $b_i$, and maybe other half-edges.
  Indeed, if this was not the case we could find in $\Gamma$ a vertex $v$ contracted by $\pi_m$ such that $v$ is connected to exactly one edge and only half-edges in the $a_i$.
  This would contradict the assumption that $[C]$ is smooth.

  Denote by $f_{b_i}:\Mrpo{n+m}{r}\ra\Mrpo{n+m-1}{r}$ the map forgetting $b_i$. Then $[C_m]$ is a smooth point of $f_{b_i}$.
  But $\pi_m$ factorizes through $f_{b_i}$. We conclude by induction that $\pi_m$ is smooth at $[C_m]$.

  Now $u:\Mrpo{n+m}{r}\ra \Mrpo{n+m}{r}/S_m$ is an étale cover, so $u([C_m])$ is a smooth point of $\Mrpo{n+m}{r}/S_m\ra\Mrpo{n}{r}$. Hence the smoothness of $\pb$.
\end{proof}

To prove \Cref{gluingishodge}, we need only prove that the map:
\begin{equation}
  R^mp'_!\Line'_{\mid \B}\lra \sheafHmid'_{\mid \B} \tag{\ref{equationgeometricmap} revisited}
\end{equation}
is strict for the Hodge filtration.
To show this, we shall decompose it into maps that are known to be maps of (mixed)
Hodge structures. But first we need to introduce some more theory of variations of Hodge
structures.

\begin{theorem}[monodromy, {\cite{landmanPicardLefschetzTransformationAlgebraic1973}}]
  Let $X$ be a projective orbifold and $D\subset X$ a divisor with normal crossing.
  Let $(E,\nabla)$ be a geometric integral variation of Hodge structure on $X\setminus D$,
  then the eigenvalues of the monodromy of $\nabla$ around the components of $D$
  are roots of unity.
\end{theorem}

Actually, this theorem is still true when $(E,\nabla)$ is not assumed to be geometric.
But we shall not need this generalization.

Thanks to the monodromy theorem, the local system $R^m\pb_*\Line$
on $\justM$ has unipotent monodromy around boundary divisors after pullback
by $\Mrpo{n}{r'}\ra\Mrpo{n}{r}$ for some $r|r'$.
Let us pullback our situation by this map. 
The morphism of \Cref{equationgeometricmap} is strict for the Hodge filtration if and only if its pullback is.
\textbf{From now until the end of this section we will replace $\justM=\Mrpo{n}{r}$ by $\justM=\Mrpo{n}{r'}$. Similarly for $\justM'$.
We also replace $\justM_m$ and $\justM'_m$ by their pullback.}

\begin{definition}[Deligne canonical extension, {\cite[pp. 91-5]{deligneEquationsDifferentiellesPoints1970}}]
  Let $(E,\nabla)$ be an integral variation of Hodge structures on $U\setminus D$ where
  $U$ is a global finite quotient of a smooth quasi-projective variety and $D$ is a smooth divisor. Assume $\nabla$
  has unipotent monodromy around every connected component of $D$.

  Then $(E,\nabla)$ has a unique holomorphic extension $E^\ext$ to $U$ such that $\nabla$ has logarithmic
  singularities at $D$ with nilpotent residues. Moreover, the Hodge filtration of $E$ extends uniquely to a
  holomorphic filtration of $E^\ext$.
\end{definition}

\begin{proposition}[Clemens-Schmid]\label{propositionschmid}
  Let $p:Y\ra U$ be a proper map of smooth analytic Deligne-Mumford stacks such that étale-locally on $U$, $Y$ is a global finite quotient of projective varieties.
  Assume $D\subset U$ is a smooth divisor such
  that $p_{sm}:p^{-1}(U\setminus D)\ra U\setminus D$ is smooth and $p$ is locally topologically trivial over $D$.
  Let $\Line$ be a local system on $Y$ with finite monodromy and $m\geq 0$. Then the degeneration map:
  \begin{equation*}
    \nu: \left(R^mp_*\Line\right)_{\mid D}\lra \left(R^mp_{sm*}\Line\right)^\ext_{\mid D}
  \end{equation*}
  is strict for the Hodge filtrations.
\end{proposition}
\begin{proof}[Reduction to {\cite{steenbrinkLimitsHodgeStructures1976}}.]
    Note that this result can be proved étale locally. Hence we can reduce to the case where $U$ is a polydisk $\Disk^k$ with $D=\{0\}\times\Disk^{k-1}$.
    Now the result can be proved after pullback to a disk $\Disk$ with $D=\{0\}$, by functoriality of the Deligne canonical extension.
    Hence we assume that $U$ is such a disk.
    By the assumption $Y$ is $Y'/G$ for $G$ a finite group and $Y'$ a projective variety. Let $\nu'$ be the map associated $Y'\ra U$.
    By \cite{steenbrinkLimitsHodgeStructures1976}, $\nu'$ is strict for the Hodge filtrations. Now $\nu'$ is $G$-equivariant and $G$ preserves the Hodge filtrations.
    But $\nu$ is obtained from $\nu'$ by taking $G$-invariants, which is exact. Hence $\nu$ is strict for the Hodge filtrations.
\end{proof}

The map $\nu$ is actually a map of mixed Hodge structures on the fibers, but we do not need this here. 

Let us describe the map $\nu$. Let $V\ra U$ be an étale map from a polydisk $V\simeq \Disk^k$
with $D\cap U\simeq \{0\}\times\Disk^{k-1}$.  Then there exists a retraction:
\begin{equation*}
  r: Y_{\mid V}\lra Y_{\mid V\cap D}
\end{equation*}
over the map $\Disk^k\ra\Disk^k$, $(x,v)\mapsto (0,v)$. Actually, $r$ is a homotopy equivalence.
Fix $v\in\Disk^{k-1}$ and $\phi\in H^m(Y_{(0,v)};\Line)$. For each $x\in\Disk$, we have a map:
\begin{equation*}
  H^m(Y_{(0,v)};\Line)\xrightarrow{r^*} H^m(Y_{\Disk\times\{v\}};\Line) \lra H^m(Y_{(x,v)};\Line).
\end{equation*}
Let $\phi_x$ be the image of $\phi$ by this map. Then the family $(\phi_x)_{x\in\Disk\setminus \{0\}}$
extends to a section of $\left(R^mp_{sm*}\Line\right)^\ext$ on $\Disk\times\{v\}$.
In particular, this describes a map $\left(R^mp_*\Line\right)_{(0,v)}\ra \left(R^mp_{sm*}\Line\right)^\ext_{(0,v)}$.

We can now prove that the map of \Cref{equationgeometricmap} is strict for the Hodge filtrations.
It decomposes as in the following commutative diagram.
\[\begin{tikzcd}
	{R^mp'_!\Line'_{\mid \B}} & {R^m\pb_*'\Line'_{\mid \B}} \\
	& {\left(R^m\pb_*\Line\right)^\ext_{\mid \B}} \\
	{\sheafHmid'_{\mid \B}} & {\left(W_mR^mp_*\Line\right)^\ext_{\mid \B}}
	\arrow[from=1-1, to=1-2]
	\arrow["\nu", from=1-2, to=2-2]
	\arrow[two heads, from=2-2, to=3-2]
	\arrow[hook, from=3-1, to=3-2]
	\arrow["{(*)}"', two heads, from=1-1, to=3-1]
\end{tikzcd}\]

To apply \Cref{propositionschmid} to the map $\nu$ of the diagram, we need to check that étale-locally on $\justM'$,
$\justM'_m$ is the quotient of a projective variety by a finite group.
So we may work before pullback by $\Mrpo{n}{r'}\ra\Mrpo{n}{r}$.
But $\justM'$ is a gerbe over the locus in $\Mrp{n+1}{r}$
of curves with at most one node, and gerbes are étale-locally trivial. Hence we may replace $\justM'$ by this Zariski open in $\Mrp{n+1}{r}$
and $\justM'_m$ by a Zariski open in $\Mrp{n+m+1}{r}$. By \Cref{globalquotient},
$\Mrp{n+m+1}{r}$ is a global finite quotient. The corresponding Galois cover is projective as it is finite over the coarse
space $\M{n+m+1} $of $\Mrp{n+m+1}{r}$.
So we can apply \Cref{propositionschmid} to $\nu$.

All the arrows except $(*)$ are already known to be strict for the Hodge filtrations.
Hence $(*)$ is also strict for the Hodge filtrations.


\subsection{Proof of \Cref{intersectionformpolarizes}}

In this subsection, we prove that the TQFT hermitian form polarizes the integral variations of Hodge structures on the
$\SO$ quantum representations.

\begin{proposition}[{\cite[II.7]{voisinHodgeTheoryComplex2002}}]\label{primitivepolarization}
Let $p:Y\ra X$ a proper smooth map of smooth global finite quotients of relative dimension $m$.
Let $[\omega]$ be a Kähler class in $H^2(Y,\Z)$ and $(\Line,h)$
an integral unitary local system of finite monodromy on $Y$ with unitary form $h$.
Let $s$ be the sesquilinear intersection form on $R^mp_*\Line$ induced by $h$ (see \Cref{subsubhermitianstructure}).
Denote by $PR^mp_*\Line$ the kernel of $\alpha\mapsto \alpha\wedge [\omega]$ in $R^mp_*\Line$.

Then $PR^mp_*\Line$ is a sub-IVHS of $R^mp_*\Line$ and $Q=(-1)^{\frac{m(m-1)}{2}}s$
is a polarization on $PR^mp_*\Line$.
\end{proposition}

The sheaf $PR^mp_*\Line$ is called the primitive part of $R^mp_*\Line$. 

The form $(-1)^{\frac{m(m-1)}{2}}s$ is defined over $\Z[\zeta_r]$.
It remains to prove that the Hodge decomposition is orthogonal and that $i^{p-q}(-1)^{\frac{m(m-1)}{2}}s$
is positive on $H^{p,q}$. For this, we will relate the form and the decomposition to
the primitive part of the cohomology of a fibration in smooth proper Kähler manifolds and apply
\Cref{primitivepolarization}.

We will use the notations $\justM=\sMrpo{n}{r}$, $\justM_m=\sMrpo{m+n}{r}/S_m$.
Let $\Mb_m$ be the pre-image of $\justM$ in $\Mrpo{m+n}{r}/S_m$
and $\pb:\Mb_m\ra\justM$ the projection. Notice that $\pb$ is proper, but also smooth by \Cref{pbsmooth}.
Let $\Line=\Nuab(b;a_1,\dotsc,a_n,-2,\dotsc,-2))\otimes\epsilon$. Then we have a factorization:
\begin{equation*}
  R^mp_!\Line \xrightarrow{\alpha} R^m\pb_*\Line \xrightarrow{\beta} R^mp_*\Line
\end{equation*}
with $\Nu(b;a_1,\dotsc,a_n)\simeq \myim{\beta\circ\alpha}$.

Let $\sline$ be the intersection form on $R^m\pb_*\Line$.
Now $\Mrpo{m+n}{r}/S_m$ is a projective orbifold.
Hence it has an integral Kähler class $[\omega]\in H^2\left(\Mrpo{m+n}{r}/S_m;\Z\right)/\mathrm{torsion}$.
Let us show that the image of $\alpha$ is in the primitive part $PR^m\pb_*\Line$.

The coarse space of $\Mrpo{m+n}{r}$ is $\M{m+n+1}$. Now, $H^2(\M{m+n+1};\Q)$
is generated by the classes of the normal crossing boundary divisor $\partial\M{m+n+1}=\M{m+n+1}\setminus \sM{m+n+1}$
(see \cite[6 p.550]{keelIntersectionTheoryModuli1992}).
As $H^2(\M{m+n+1};\Q)\simeq H^2\left(\Mrpo{m+n}{r};\Q\right)$ and $H^2\left(\Mrpo{m+n}{r};\Q\right)^{S_m}\simeq H^2\left(\Mrpo{m+n}{r}/S_m;\Q\right)$,
this is also the case for $H^2\left(\Mrpo{m+n}{r}/S_m;\Q\right)$. Hence $[\omega]$ is rationally a sum of boundary classes.
Now, for each $x\in \justM$, $\pb^{-1}(x)\setminus p^{-1}(x)$ is included in the boundary strata of $\Mrpo{m+n}{r}/S_m$.
Hence, for $u$ in $H_c^*(p^{-1}(x);\Line)$, $u\wedge [\omega]$ is $0$ in $H^*(\pb^{-1}(x);\Line)$.
So the image of $\alpha$ is in the primitive part $PR^m\pb_*\Line$.

Now $\alpha$ is compatible with $s$ and $\sline$. We can thus deduce the result from \Cref{primitivepolarization}.


\bibliographystyle{plain}
\bibliography{biblio}

\end{document}